\documentclass[letterpaper, 11pt,  reqno]{amsart}
\usepackage[margin=1.2in,marginparwidth=1.5cm, marginparsep=0.5cm]{geometry}

\usepackage{amsmath,amssymb,amscd,amsthm,amsxtra, esint, xcolor}
\usepackage[implicit=true]{hyperref}

\usepackage[shortlabels]{enumitem}

\allowdisplaybreaks[2]

\usepackage{color}

\definecolor{gr}{rgb}   {0.,   0.69,   0.23 }
\definecolor{bl}{rgb}   {0.,   0.5,   1. }
\definecolor{mg}{rgb}   {0.85,  0.,    0.85}
\definecolor{yl}{rgb}   {0.8,  0.7,   0.}
\definecolor{or}{rgb}  {0.7,0.2,0.2}

\newtheorem{theorem}{Theorem} [section]

\newtheorem{lemma}[theorem]{Lemma}
\newtheorem{proposition}[theorem]{Proposition}
\newtheorem{remark}[theorem]{Remark}

\newtheorem{definition}[theorem]{Definition}

\newcommand{\I}{\hspace{0.5mm}\text{I}\hspace{0.5mm}}

\newcommand{\II}{\text{I \hspace{-2.8mm} I} }

\newcommand{\noi}{\noindent}
\newcommand{\Z}{\mathbb{Z}}
\newcommand{\R}{\mathbb{R}}
\newcommand{\C}{\mathbb{C}}
\newcommand{\T}{\mathbb{T}}

\newcommand{\HI}{\textup{HI}}
\newcommand{\lo}{\textup{lo}}
\newcommand{\LO}{\textup{LO}}

\let\P= \undefined
\newcommand{\P}{\mathbf{P}}

\newcommand{\Q}{\mathbf{Q}}

\newcommand{\GG}{\mathcal{G}}

\newcommand{\F}{\mathcal{F}}

\newcommand{\al}{\alpha}
\newcommand{\be}{\beta}
\newcommand{\dl}{\delta}

\newcommand{\V}{\mathcal{V}}

\newcommand{\eps}{\varepsilon}

\newcommand{\g}{\gamma}

\newcommand{\ld}{\lambda}

\newcommand{\s}{\sigma}

\newcommand{\ft}{\widehat}
\newcommand{\Ft}{{\mathcal{F}}}
\newcommand{\wt}{\widetilde}
\newcommand{\cj}{\overline}
\newcommand{\dx}{\partial_x}
\newcommand{\dy}{\partial_y}

\newcommand{\dt}{\partial_t}

\newcommand{\ta}{\theta}

\renewcommand{\l}{\ell}
\renewcommand{\o}{\omega}

\newcommand{\les}{\lesssim}
\newcommand{\ges}{\gtrsim}

\newcommand{\jb}[1]
{\langle #1 \rangle}

\newcommand{\ind}{\mathbf 1}

\renewcommand{\S}{\mathcal{S}}

\newcommand{\M}{\mathcal{M}}

\newcommand{\N}{\mathbb{N}}
\newcommand{\NN}{\mathcal{N}}

\newtheorem*{ackno}{Acknowledgements}

\renewcommand{\H}{\mathcal{H}}

\newcommand{\Pbhi}{\mathbf{P}_{\textup{hi}} }
\newcommand{\Pblo}{\mathbf{P}_{\textup{lo}} }
\newcommand{\Pbhip}{\mathbf{P}_{\textup{+,hi}} }
\newcommand{\PbHIp}{\mathbf{P}_{\textup{+,HI}} }
\newcommand{\PbHI}{\mathbf{P}_{\textup{HI}}}
\newcommand{\PbLO}{\mathbf{P}_{\textup{LO}}}

\newcommand{\TT}{\mathcal{T}}

\newcommand{\Id}{\textup{Id}}

\numberwithin{equation}{section}
\numberwithin{theorem}{section}

\makeatletter
\@namedef{subjclassname@2020}{\textup{2020} Mathematics Subject Classification}
\makeatother

\begin{document}
\baselineskip = 14pt

\title[The extended SBO system]{Local and global well-posedness for the extended Schr\"{o}dinger-Benjamin-Ono system}

\author[P.~Dai, J.~Forlano]
{Puti Dai and Justin Forlano}

\address{Puti Dai,
School of Mathematics, Monash University, VIC 3800, Australia}

\email{Puti.Dai1@monash.edu}

\address{Justin Forlano,
School of Mathematics, Monash University, VIC 3800, Australia}

\email{justin.forlano@monash.edu}

\subjclass[2020]{35A01, 35Q35, 35Q55}

\keywords{
well-posedness, global well-posedness, gauge transform, Benjamin-Ono, Schr\"{o}dinger-Benjamin-Ono system}


\begin{abstract} 
We study the well-posedness problem for the extended Schr\"{o}dinger-Benjamin-Ono system (eSBO) on the real line. This system couples a Schr\"{o}dinger field $u$ with a Benjamin-Ono type field $v$, including a term of the form $\dx(v^2)$. This latter term, just as in the case of the Benjamin-Ono equation, causes the system to become quasilinear and unsolvable via Picard iteration. 
We prove that eSBO is locally well-posed in $H^{s+\frac 12}(\R)\times H^{s}(\R)$ for any $s\geq 0$. In particular, this result covers the energy space at $s=\frac 12$, yielding global well-posedness in $H^{1}(\R)\times H^{\frac 12}(\R)$ with a small $L^2$-assumption on the Schr\"{o}dinger part of the initial data. 
\end{abstract}

\maketitle

\tableofcontents

\section{Introduction}

We study the Cauchy problem for the
extended Schr\"{o}dinger-Benjamin-Ono system  (eSBO):
\begin{equation}
\left\{
\begin{aligned}
  & i\dt u+\dx^2 u = uv+\be |u|^2 u,\\
  & \dt v-\H \dx^2 v= \dx(|u|^2-\rho v^2),\\
   & (u,v)|_{t=0}=(u_0,v_0),
\end{aligned}
 \right. \label{SBO}
\end{equation}
 where $u:\R\times \R\to \C$, $v:\R\times \R\to \R$, $\be,\rho\in \R$, and 
  $\H$ is the Hilbert transform which is defined by the singular integral operator
  \begin{align*}
\H f (x)  =\text{p.v.} \frac{1}{\pi} \int_{\R} \frac{ f(y)}{x-y}dy.
\end{align*} 
Physically, the eSBO system \eqref{SBO} arises as a particular case of the following general system which
models the interaction of a short wave (the Schr\"odinger part) and a long wave (the Benjamin-Ono part): 
\begin{equation}
\left\{
\begin{aligned}
  & i\dt u+\dx^2 u = \al uv+\be |u|^2 u,\\
  & \dt v+\nu P(\dx) v= \g \dx(|u|^2) + \ld \dx v^2,
\end{aligned}
 \right. \label{gSBO}
\end{equation}
where $\al, \be, \nu, \g, \ld\in \R$ and $P(\dx)$ is a constant coefficient differential operator.
Such systems arise ubiquitously in applications \cite{App1, App2, App3, App4, App5}, from capillary and gravity water waves, ion-acoustic waves, and sonic-Langmuir waves in plasmas.
In particular, special interest is taken to the following two cases:
 First, when $P(\dx) = -\H \dx^2$, $\be=0$ and $\ld=0$:
\begin{equation}
\left\{
\begin{aligned}
  & i\dt u+\dx^2 u = \al uv,\\
  & \dt v-\nu \H  \dx^2 v= \g \dx(|u|^2),
\end{aligned}
 \right. \label{SBo}
\end{equation}
which is known as the Schr\"{o}dinger-Benjamin-Ono system (SBO), first derived by Funakoshi-Oikawa \cite{App1}.  Second, when $P(\dx)= \nu \dx^3$ and, in general, $\ld\neq 0$, 
\begin{equation}
\left\{
\begin{aligned}
  & i\dt u+\dx^2 u = \al v+\be |u|^2 u,\\
  & \dt v-\nu\dx^3 v= \g \dx(|u|^2) + \ld \dx v^2,
\end{aligned}
 \right. \label{SKDV}
\end{equation}
which is the so-called Schr\"{o}dinger-Korteweg-de Vries system (SKdV). See \cite{CL} for further references on the applicability of SKdV \eqref{SKDV}.
Relative to \eqref{SBo} and \eqref{SKDV}, the eSBO system \eqref{SBO} involves the most difficult parts of both previous systems: $\ld=-\rho \neq 0$ and the weaker nonlocal dispersion $P(\dx)= -\H\dx^2$.

Our main goal in this article is to advance the well-posedness theory for the eSBO system \eqref{SBO}. To motivate the function spaces, we note that smooth solutions to \eqref{SBO} enjoy a number of conservation laws: 
    \begin{align}
    \begin{split}
  E_0 (v)& = \int_{\R} v(t,x)dx  \quad \text{and} \quad 
  E_1(u)  =\int_{\R} |u(t,x)|^2 dx\\
  E_2(u,v)& =\text{Im} \int_{\R} u(t,x)  \cj{ \dx u(t,x)} + \frac{1}{2}\int_{\R} v(t,x)^2 dx\\
  H(u,v) & = \frac12 \int_{\R} ( |\dx|^{\frac 12} v)^2 - \frac{\rho}{3} \int_{\R} v^3 dx + \int_{\R} v |u|^2 dx + \frac{\be}{2} \int_{\R}|u|^4 dx + \int_{\R} |\dx u|^2 dx,
  \end{split} \label{energy}
\end{align}
  where we will refer to $H$ as the energy and $E_1$ as the mass.
In view of the conserved energy $H$ in \eqref{energy},
it is natural to study the well-posedness of the system \eqref{SBO} for the initial data in 
\begin{align*}
(u_0,v_0) \in H^{s+\frac 12}(\R) \times H^{s}(\R)
\end{align*}
where $s\in \R$, where the case $s=\frac{1}{2}$ constitutes the energy space.

To discuss what is known, let us first recall the current well-posedness theory for SBO \eqref{SBo}, essentially corresponding to the case $\rho=0$ in \eqref{SBO}.  
Within the space $H^{s+\frac{1}{2}}(\R)\times H^{s}(\R)$, local well-posedness for any $s\geq -\frac 12$ and when $|\nu|\neq 1$ (the non-resonant case) was first shown in \cite{SBO1}.  Global well-posedness was also established in \cite{SBO1} for some range of parameters in $\al, \g$ and $\nu<0$ in $s\geq \frac 12$ using suitable conservation laws similar to \eqref{energy}. Pecher \cite{SBO2} then extended the local well-posedness to the resonant case $|\nu|=1$ for $s>-\frac 12$, and used the I-method to establish global well-posedness, for the same parameter ranges, for $s>-\frac 16$.  This latter global result was improved by Angulo-Matheus-Pilod \cite{SBO4} to any $s\geq -\frac 12$ using a method due to Colliander-Holmer-Tzirakis \cite{CHT}.
Later, Domingues \cite{SBO3} proved local well-posedness results in the non-resonant case in $H^{s'}(\R) \times H^{s}(\R)$ for sharp ranges of $(s,s')$, not necessarily restricted to $s'=s+\frac 12$. We also point out the results of Oh \cite{OhSBO} in the setting of periodic boundary conditions, in particular the construction of global-in-time solutions on the support of the Gibbs measure.
 
The local well-posedness arguments mentioned above are based on contraction mapping arguments using the Fourier restriction norm method \cite{BO93, BO932}. 
Consequently, in these regimes the data to solution map is analytic, and thus \eqref{SBo} is semilinear.
Moving back to \eqref{SBO}, as soon as $\rho\neq 0$, the Cauchy problem becomes quasilinear, in the sense that for any $s\in \R$ the data to solution map cannot be $C^2$. This result can be shown following the strategy in \cite{MST}, which showed such a failure for the Benjamin-Ono equation (BO):
\begin{align}
\dt v - \H \dx^2 v= \dx( v^2). \label{BO}
\end{align}
This implies that the previously cited results for \eqref{SBo} cannot be adapted to handle \eqref{SBO} when $\rho \neq 0$.  
Thus the low-regularity well-posedness analysis of \eqref{SBO} is more similar to that of BO \eqref{BO} rather than SBO \eqref{SBo}. 

By now, the well-posedness theory for \eqref{BO} is quite well-understood, see \cite{Iorio, KT, TAO04, IK, BP, MP, IT}, culminating in the global well-posedness in $H^{s}(\R)$ for any $s\geq 0$ by using a gauge transformation and Fourier analytic/dispersive arguments. Using the complete integrability of \eqref{BO}, this was pushed to the full scaling sub-critical regime $s>-\frac 12$ in \cite{KLV}. See also \cite{Molinet1, GKT-2} for analogous results in the periodic setting.

However, we point out that \eqref{SBO} is a nonlinear system with genuine nonlinear interactions between the Schr\"{o}dinger and BO-parts so it is not a-priori clear that results for \eqref{BO} can be adapted to \eqref{SBO}. Moreover, we do not know whether \eqref{SBO} is completely integrable. 
Recently, Linares-Mendez-Pilod \cite{LMP} provided the first result on the local well-posedness for eSBO \eqref{SBO} in $H^{s+\frac 12}(\R)\times H^{s}(\R)$ with $s>\frac 54$. Their method is based on the energy method with refined Strichartz estimates al\'{a} \cite{KT} and modified energies.

Our first contribution is to extend the local well-posedness for \eqref{SBO} to the range $s\geq 0$.
  
  \begin{theorem}\label{THM:LWP}
  Let $s\geq 0$ and $\rho \in \R$. Then, the extended SBO system \eqref{SBO} is locally well-posed in $H^{s+\frac 12}(\R)\times H^{s}(\R)$. More precisely, for every $(u_0,v_0)\in H^{s+\frac 12}(\R)\times H^{s}(\R)$, there exists $T=T(\|(u_0,v_0)\|_{H^{s+\frac 12}\times H^{s}})>0$ and a unique distributional solution $(u,v)\in C([0,T]; H^{s+\frac 12}(\R))\times C([0,T];H^{s}(\R))$ to \eqref{SBO} satsifying $J^{s}v\in \wt{L^4_{T,x}}$ and $w:=\dx\P_{+,\textup{hi}}[ e^{i\rho F[v]}]\in Z^{s}_{T}$, where $F$ is defined in \eqref{F} and the space $Z^{s}_{T}$ is defined in \eqref{ZsT}. Moreover, the data-to-solution map 
 $(u_0,v_0) \in H^{s+\frac 12}(\R) \times H^{s}(\R)\mapsto (u,v)\in C([0,T]; H^{s+\frac 12}(\R))\times C([0,T];H^{s}(\R))$ is continuous.
    \end{theorem}

The range $s\geq 0$ in Theorem~\ref{THM:LWP} appears to be optimal in view of our method, as discussed below. 
 In particular, our result covers the energy space at $s=\frac 12$, for which we can exploit the conservation laws \eqref{energy} to obtain a priori bounds on solutions. We thus obtain the first global well-posedness results for the eSBO system \eqref{SBO}.
  
\begin{theorem} \label{THM:GWP}
Let $\rho,\be \in \R$ and for $r>0$, let $B_{r}=\{u_0 \, :\, \|u_0\|_{L^{2}}\leq r\}$.
There exists $R>0$ such that the extended SBO system \eqref{SBO} is globally well-posed in $(H^{1}(\R)\cap B_{R}) \times H^{\frac 12}(\R)$.
\end{theorem}

It is unclear to us if the restriction to small $L^2-$norm for the Schr\"{o}dinger part can be relaxed.
Ultimately, this stems from the additional cubic term $\frac{\rho}{3}\int v^3 dx$ in $H(u,v)$, which presents a kind of criticality in the exponents as the $L^2$-norm of $v$ is not conserved. 
Instead, we have to combine the conservation of $E_1(u)$ and $E_{2}(u,v)$ to adequately control this cubic term. See \eqref{Energ1}. Notice that no such term arises for the SBO system \eqref{SBo}, and thus no small data restriction is imposed in that setting.

If we take the initial datum for the Schr\"{o}dinger part $u_0 \equiv 0$, and assume that $v_0 \in L^2(\R)$, then Theorem~\ref{THM:GWP} and the conservation of $E_{1}$ implies that $u(t)\equiv 0$ for all $t\in \R$, whence \eqref{SBO} reduces completely to the BO equation \eqref{BO} (modulo the factor of $-\rho$). This observation allows us to immediately transport the known ill-posedness result due to Biagoni-Linares \cite{BL} for BO to the system \eqref{SBO}. The ill-posedness here is that when $s<-\frac 12$, an extension of the solution map $\Phi: H^{s+\frac 12}\times H^{s}\to C([-T,T];H^{s+\frac 12})\times C([-T,T];H^{s})$, sending initial data $(u_0,v_0)$ to solutions $(u,v)=\Phi(u_0,v_0)$, must be discontinuous at $(u_0,v_0)=(0,-2\pi \dl_0)$ for any $T>0$.

\begin{theorem}\label{THM:ILL}
Let $s<-\frac 12$. Then, the eSBO equation \eqref{SBO} is ill-posed in $H^{s+\frac 12}(\R)\times H^{s}(\R)$. 
\end{theorem}

It remains an interesting problem to determine the well-posedness for the system \eqref{SBO} in the remaining range $-\frac{1}{2}\leq s<0$. As mentioned above, the only work that can cover the range $-\frac 12<s<0$ for BO \eqref{BO} is \cite{KLV} that relies on the complete integrability and thus is not applicable to \eqref{SBO}. On the circle, the situation for BO \eqref{BO} is also now completely resolved with well-posedness for $s>-\frac 12$ and ill-posedness for $s\leq -\frac 12$. See \cite{GKT-2} and the references therein. 

We now discuss our method of proof for Theorem~\ref{THM:LWP}. 
 To weaken the bad high$\times$low$\to$high interactions in term $\dx(v^2)$, we use a gauge transformation for the Benjamin-Ono part $v$, which is motivated by the theory of the BO equation by Tao \cite{TAO04} and further analysed in \cite{BP, MP}. 
 This gauged variable, which we call $w=w[v]$, satisfies \eqref{weqn} which is similar to the gauged equation for BO but with the extra term 
 $\dx \Pbhip[ e^{-i\rho F}|u|^2].$ Whilst $u$ has spatial regularity $s+\frac 12$, the  low$\times$low$\times$high$\to$ high interactions in this term cannot be controlled without additional smoothing. In this case, we gain the extra $\frac 12$-degree of smoothing by applying to the product $u\cj{u}$ the bilinear Strichartz estimate due to Bourgain \cite{BO98} and Ozawa-Tsutsumi \cite{OT}; see \eqref{bilin0}. Notice that it is then natural to posit that $u\in X^{s+\frac 12,b}_{T}$, for some $b>\frac 12$, where $X^{s,b}_{T}$ is the Fourier restriction norm space \cite{BO93}, which is defined in \eqref{Xsb}.
 
  Next, we consider the Schr\"{o}dinger part for which the main nonlinear term is the product $uv$. This product needs to be estimated in $X^{s+\frac 12, b-1}_{T}$ but only assuming that $v\in X^{s-\ta,\ta}_{T}$ for $0\leq \ta\leq 1$; namely, with a loss of spatial derivatives. The main issue is the low$\times$high interactions, as we need to observe $\frac 12$-degree of smoothing to control this product at regularity $s+\frac 12$.
 It turns out that the product $u \P_{+}v$ is highly non-resonant and thus a bilinear estimate can be established; see Lemma~\ref{LEM:SchrodingerII}. Here, by highly non-resonant we mean that the phase function is the square of the maximum frequency; see \eqref{IIphase}. In comparison, the phase function for the product $u \P_{-}v$ is worse behaved being the minimum times the maximum frequency, which is insufficient.
However, since $v$ is real-valued, $\P_{-}v= \cj{\P_{+}v}$ and we may rewrite this using in \eqref{recovery2} in terms of the gauged variable $w$ at the cost of a (smoother) commutator term $R(v,w)$:
  \begin{align*}
u\P_{-}v ``=" e^{-i\rho F} u\cj{w} +  u \cj{R(v,w)},
\end{align*}
  where we neglected some inessential additional frequency projections and constant factors.
  It turns out that, after some careful frequency decompositions, the commutator term enjoys \textit{exactly} an additional $\frac 12$-degree of regularity as compared to $w$ (measured in $\wt{L^{4}_{T,x}}$, see Lemma~\ref{LEM:s12R}) which allows us to control the low$\times$high$\to$high interactions. 
  
  It still remains to control the first term $e^{-i\rho F} u\cj{w}$ in $X^{s+\frac 12,b-1}_{T}$. As we cannot obtain favourable Fourier restriction norm properties for the exponential term $e^{-i\rho F}$ all the way down to $s=0$, we do not know how to exploit mulitlinear dispersion to gain additional smoothing for the low$\times$low$\times$high$\to$high interactions here. Again, we rely on the bilinear Strichartz estimate to absorb \textit{exactly} $\frac 12$ of the derivative. Unfortunately, the transference property for Fourier restriction norm spaces would require us to put $w\in X^{s,b}_{T}$ for $b>\frac 12$. When $s>0$ and $b>\frac 12$, we can prove the main bilinear estimate for the BO part 
  \begin{align}
\dx \Pbhi[ (\dx^{-1}w) \P_{-}\dx v]. \label{BObilin}
\end{align} 
However, due to a $\log$-loss when $s=0$, we would miss the endpoint space $s=0$. If we instead force $b=\frac 12$ for $w$, then we only gain $(\frac{1}{2}-\eps)$-smoothing, for any $\eps>0$, from the bilinear Strichartz estimate which is insufficient.
We point that this issue does not seem to be resolved by taking the $\l^1$-based Besov variants of $X^{s,b}$, as done in \cite{MP2}, or by considering the spaces used in \cite{MP}. 

Our desire to reach the  endpoint space $s=0$ led us to consider the adapted function spaces based on the framework of the atomic function spaces $U^2$ and their pre-dual $V^2$ as a replacement for $X^{s,\frac 12}$. These spaces were first used in \cite{U2V21, U2V22, U2V23} and have proven very effective for scaling critical problems. This allows us to obtain a transference for the bilinear Strichartz estimate without paying an $\eps$-derivative loss. The price to pay here is that we need to prove an analog of the bilinear estimate for \eqref{BObilin} in \cite[Proposition 3.5]{MP} but in the $U^2-V^2$-framework; see Lemma~\ref{LEM:BO1}.


\begin{remark}\rm
We point out that our argument still works for the local well-posedness for any $\nu \neq 0$ of the system:
\begin{equation}
\left\{
\begin{aligned}
  & i\dt u+\dx^2 u = uv+\be |u|^2 u,\\
  & \dt v-\nu \H \dx^2 v= \dx(|u|^2-\rho v^2),
\end{aligned}
 \right. \label{eSBOnu}
\end{equation}
The point is that we only use the bilinear strichartz estimate in the case of one input much larger than the other and here the bilinear estimate still holds with the two dispersion relations $-i\dx^2$ and $\pm i\nu \dx^2$. The only part that fails as written is \eqref{IIphase} when $\nu=-1$ as we only gain $N_{\max}N_{\min}$ rather than $N_{\max}^{2}$. However, this is easily rectified by replacing $\PbHIp v$ in the product  $u \PbHIp v$ by \eqref{recovery2} as we already did in \eqref{ueq}.

More generally, Theorem~\ref{THM:LWP} continues to hold for suitable dispersive perturbations of eSBO \eqref{SBO}, in the spirit of the more general system \eqref{gSBO}. More precisely, as long as the operator $P(\dx)-\H \dx^2$ is $L^2\to L^2$ bounded, viewing \eqref{gSBO} as a perturbation of \eqref{SBO} entails an additional harmless linear term in the BO part in \eqref{SBO}.
For example, Theorem~\ref{THM:LWP} also holds for the system:
\begin{equation}
\left\{
\begin{aligned}
  & i\dt u+\dx^2 u = uv+\be |u|^2 u,\\
  & \dt v-\mathcal{G}_{\dl} \dx^2 v= \dx(|u|^2-\rho v^2),
\end{aligned}
 \right. \label{SILW}
\end{equation}
where, for $\dl>0$, $\GG_{\dl}=\mathcal{T}_{\dl}-\dl^{-1}\dx^{-1}$, with $\TT_{\dl}$ a Fourier multiplier operator with multiplier $-i\coth(\dl \xi)$.
 Formally, \eqref{SILW} represents a coupling of a short wave with a long wave of ``finite depth" $\dl$, which is a system version of the intermediate long wave equation (ILW) $\dt v-\GG_{\dl}\dx^2 v= \dx(v^2)$. Recently, the ILW equation has attracted a lot of attention as it is a physically relevant model, which interpolates between BO \eqref{BO} and the KdV equation in the limits $\dl\to \infty$ and $\dl\to 0$, respectively. See \cite{IS, CLOP, GL, CLOZ} and references therein.
\end{remark}


This article is structured as follows. In Section~\ref{SEC:Prelim}, we introduce notation and prove product estimates. As we need to cover the gap $0\leq s \leq \frac{5}{4}$, the estimates become aesthetically more complicated as $s$ increases. In Section~\ref{SEC:main}, we first recall the Fourier restriction norm spaces and adapted spaces based on the $U^2$ and $V^2$-spaces. 
We then apply a gauge transformation to find the equations for the new gauged variable and study some properties of the original solution $v$ as well as the commutator term $R(v,w)$. The main multilinear estimates then appear in Section~\ref{SEC:tri}. Finally, in Section~\ref{SEC:LWP} we prove Theorems \ref{THM:LWP} and \ref{THM:GWP}.

\section{Preliminaries}\label{SEC:Prelim}

\subsection{Notation}

In this subsection, we introduce relevant notation, projections, and function spaces, which will be used throughout.

We use $A\les B$ to denote $A\leq C B$ for some constant $C>0$,
$A\ll B$ if there is a small $c>0$ such that $A\le cB$, and $A\sim B$ if both $A\les B$ and $B\les A$ hold. 
The notation $a\pm$ refers to $a\pm\eps$ for any $\eps>0$. 
For $a,b\in \R$, $a\land b$ and $a\lor b$ denote the minimum and the maximum between $a$ and $b$, respectively. We define $\Z_{\geq 0} =\N\cup\{0\}$.

Given a function $f$ on $\R$, we use $\F f$ and $\ft f$ to denote its Fourier transform
\begin{align}
\ft f(\xi) = \frac{1}{\sqrt{2\pi}} \int_\R f(x) e^{-i\xi x} dx. 
    \notag
\end{align}
For space-time functions $u:\R\times\R \to \C$, we may use the notation $\F_t u$ and $\F_x u$ to indicate the Fourier transform with respect to the time and space variables. We omit this indexing, when clear from context.

Let $s\in \R$ and $1 \le p \le \infty$. We define the $L^p$-based Sobolev spaces $W^{s, p}(\R)$ by the norm:
\begin{align*}
    \| f\|_{W^{s, p}} 
    = \| J^s f \|_{L^p}
    = \big\| \Ft^{-1} \big( \jb{\xi}^s \ft f(\xi) \big) \big\|_{L^p}, 
\end{align*}
where $J^s$ denotes the Bessel potential with Fourier multiplier $\jb{\xi}^s$, where $\jb{x} = (1+|x|^2)$ and $\Ft^{-1}$ stands for the inverse Fourier transform. 
We also use $\dot{W}^{s,p}(\R)$ for the homogeneous Sobolev spaces with norm
\begin{align*}
\| f\|_{\dot{W}^{s, p}} 
    = \| D^s f \|_{L^p}
    = \big\| \Ft^{-1} \big( |\xi|^s \ft f(\xi) \big) \big\|_{L^p}
    ,
\end{align*}
where $D^s$ is the Riesz potential, with Fourier multiplier $|\xi|^s$.
When $p=2$, we write $W^{s,2}(\R) = H^s(\R)$ for the $L^2$-based Sobolev spaces, with norm $    \| f \|_{H^s} = \| \jb{\xi}^s \ft f (\xi) \|_{L^2_\xi}. $
We define the inner product on $L^2(\R)$ by $\jb{f,g} = \int_{\R} \cj{f} g dx, $
and dyadic $L^p$-spaces, for $1<p<\infty$, via the norm
\begin{align}
\| f\|_{ \wt{L^p_{t,x}}} = \big\|  \| \P_{N}f\|_{L^p_{t,x}} \big\|_{\l^{2}_{N}},
\notag
\end{align}
where $N\in 2^{\Z_{\geq 0}}$.
By the Littlewood-Paley square function theorem and Minkowski's inequality, it holds that 
\begin{align}\label{Lpwt}
\|f\|_{L^{p}_{t,x}} \les \|f\|_{\wt{L^{p}_{t,x}}}
\end{align}
for any $2\leq p<\infty$.
When working with space-time functions, given $T>0$, we often use the shorthand notation $L^p_T W^{s, q}_x$ for $L^p([0,T); W^{s,q}(\R))$ and $L^p_T L^q_x$ for $L^p([0,T); L^q(\R))$.

We now introduce notation to perform Litlewood-Paley decompositions. 
Let $\eta:\R\to [0,1]$ be a smooth function supported on $[-2,2]$ and equal to $1$ on $[-1,1]$.
Given $N\in 2^{\Z_{\geq 0}}$, let $\eta_{N}(\xi)=\eta(\frac{\xi}{N})$ and $\psi_N(\xi)=\eta(\frac{\xi}{N})-\eta(\frac{2\xi}{N})$. 
Note that
\begin{align*}
\sum_{N\geq 1} \psi_{N}(\xi) = 1- \eta_{\frac 12}(\xi)
\quad \text{when } \xi \in \R\setminus\{0\}
.
\end{align*}
Moreover, we use $\P_{\le N}$ and $\P_N$ to denote
 the Littlewood-Paley projectors defined by
 \begin{align*}
  \F \,(\P_{\leq N} f)  &= \eta_{N} \ft f 
, 
\\
\F(\P_1 f ) 
&= 
\eta_1  \ft f 
\quad
\text{and}
\quad
\F\, \P_N f 
= \F \,\P_{\leq N} f -\F\, 
\P_{\leq \frac{N}{2}} f 
=\psi_N \ft f, \quad \text{when } N\ge 2 ,
 \end{align*}
 and $\P_{>N}:= 1-\P_{\leq N}$. Note that $\sum_{N\ge 1} \P_N f = f ,  $
which we will often use in our estimates, where by abuse of notation, we assume to sum over dyadic numbers in $2^{\Z_{\ge0}}$. 
 We use $\wt\P_{N}$ for the wider projector with multiplier $\wt{\psi}_{N}(\xi) =\psi_{N}(\frac{\xi}{2}) + \psi_{N}(\xi) +\psi_{N}(2\xi)$.
 We then set
 \begin{equation}
\begin{alignedat}{3}
 \F\,( \P_{+} f)(\xi) &= \ind_{\{\xi> 0\}} \ft f(\xi), 
 &
 \qquad   \F\, ( \P_{-} f)(\xi)  &= \ind_{\{\xi< 0\}} \ft f(\xi),
 \\
\Pbhi & = \sum_{N\geq 2}\P_{N},  
&
\quad 
\PbHI & = \sum_{N\geq 8}\P_N, 
\\
\Pblo &= \text{Id}-\Pbhi, 
& \quad 
\PbLO &=\text{Id} - \PbHI.
\end{alignedat}
\label{projs}
\end{equation}
We also define the shorthand $\P_{\pm, \text{hi}}=\P_{\pm}\Pbhi$ and similarly for $\P_\HI, \P_\lo, \P_\LO$.

For space-time functions $u: \R\times \R \to \C$, we define frequency projectors on the space-time Fourier variables $(\tau, \xi)$: given $K\in 2^{\Z_{\geq 0}}$, we define $\Q_{K}$ and $\Q_{\leq K}$ via
\begin{align}
\label{Qpro}
\mathcal{F}_{t,x}\{ \Q_{\leq K}u\}(\tau,\xi) = \eta_{K}(\tau +\xi^2)\ft u(\tau,\xi) \quad \text{and} \quad  \mathcal{F}_{t,x}\{ \Q_{ K}u\}(\tau,\xi) =\psi_{K}(\tau +\xi^2)\ft u(\tau,\xi).
\end{align}
We also define $\Q_{> K} : = \Id- \Q_{\leq K}$. 


\subsection{The primitive}

We adapt the gauge transform, originally due to Tao~\cite{TAO04}, and as presented by Burq-Planchon \cite{BP} and Molinet-Pilod~\cite{MP}
The aim is to first construct a spatial primitive $F=F[u,v]$ of the smooth solution $v$ to the BO part of the system \eqref{SBO} so that $\dx F= v$ and $F$ satisfies:
\begin{align}
\dt F -\H\dx^2 F  +\rho (\dx F)^2=|u|^2.
\label{Feqn}
\end{align}
Let $\psi \in C_{0}^{\infty}(\R)$ be such that $\int_{\R} \psi(x)dx =1$. Define
\begin{align}
F(t,x) =\int_{\R} \psi(y) \int_{y}^{x} v(t,z)dz dy + G(t) \label{F}
\end{align}
for some function $G$ to be chosen later. Then, using \eqref{SBO}, we have
\begin{align*}
\dt F(t,x)  & = \int_{\R} \psi(y) \int_{y}^{x} \dt v(t,z)dz dy + G'(t) \\
& =   \int_{\R} \psi(y) \int_{y}^{x} \Big[  \H\partial_{z}^2 v(t,z) -\rho \partial_{z} (v(t,z)^2) +\partial_{z}( |u(t,z)|^2)   \Big]dz dy + G'(t) \\
& =  \H \dx v(t,x) -\rho v(t,x)^2 + |u(t,x)|^2  \\
& \hphantom{XX}- \int_{\R} [\H \psi'(y)  v(t,y) -\psi(y) \rho v(t,y)^2 + \psi(y)|u(t,y)|^2] dy +G'(t).
\end{align*}
Choosing $G$ so that 
\begin{align*}
G(t) = \int_{0}^{t}  \int_{\R}  [\H \psi'(y)  v(t',y) -\psi(y) \rho v(t',y)^2 + \psi(y)|u(t',y)|^2] dy dt' 
\end{align*}
then verifies \eqref{Feqn}. Note that $G(0)=0$ and is real-valued.

We then consider the function $e^{-i\rho  F}$, which satisfies $|e^{-i\rho  F}|=1$ since $F$ is real-valued.
It follows that $e^{-i\rho  F}\in L^{\infty}(\R)\setminus L^2(\R)$. Consequently, $e^{-i\rho F}$ is merely a tempered distribution. Nonetheless, as $v\in L^2(\R)$, $\dx e^{-i\rho F}\in L^2(\R)$, and we have that for almost every $\xi\in\R$, 
\begin{align*}
\mathcal{F}_x\{e^{-i\rho F}\}(\xi) = \tfrac{1}{i\xi} \mathcal{F}_{x}\{ \dx e^{-i \rho F}\}(\xi).
\end{align*}
Whilst $\Pbhi e^{i\rho F}$, $\PbHI e^{i\rho F}$, and $\Pbhip e^{i\rho F}$ are well-defined and belong to $L^2(\R)$, due to the non-integrable singularity at the origin, the quantities $\P_{\pm}(e^{i\rho F})$, $\Pblo e^{i\rho F}$ and $\PbLO e^{i\rho F}$ are ill-defined. Here, we understand the latter two via the differences
\begin{align}
\Pblo e^{i\rho F} : = e^{i\rho F} -\Pbhi e^{i\rho F} \quad \text{and} \quad \PbLO e^{i\rho F} : = e^{i\rho F} -\PbHI e^{i\rho F} .\label{PbloE}
\end{align}
It follows that
\begin{align}
\PbHI \Pblo(e^{i\rho F}) = \PbHI e^{i\rho F} - \PbHI e^{i\rho F} = 0
\quad \text{and} \quad 
\dx \Pblo(e^{i\rho F}) = \Pblo \dx e^{i\rho F}, \label{PloeiF}
\end{align}  where the $\Pblo$ appearing on the right-hand side of \eqref{PloeiF} agrees with a Littlewood-Payley projection to frequencies $\{|\xi|\les 1\}$. Due to the definition \eqref{PbloE}, in the following we need to carefully handle estimates for terms including factors of the form $\Pblo e^{-i\rho F}$ or $\PbLO e^{-i\rho F}$.

\subsection{Product estimates}

We will need a product estimate involving products of functions with the exponential factors $e^{i\rho  F}$. First, we recall the fractional Leibniz rule \cite{CW,GO,BL}.

\begin{lemma}[Fractional Leibniz rule]
\label{LEM:leib}
Let $s\ge0$ and $1<p_j,q_j,r\leq \infty$, $j=1,2$,
such that $\frac1r=\frac{1}{p_j}+\frac{1}{q_j}$. 
Then, we have
\begin{align*}
  \| J^s(fg)\|_{L^r (\R)} \les \|J^s f\|_{L^{p_1}(\R)} \|g\|_{L^{q_1}(\R)}  + \|f\|_{L^{p_2}(\R)} \| J^s  g\|_{L^{q_2}(\R)}.
\end{align*}
\end{lemma}
\noi
We now give a product estimate for products involving an exponential factor $e^{i\rho F}$.
%

\begin{lemma}\label{LEM:prodest}
Let $2\leq q<\infty$,  $0\leq s\leq \frac 32$, and $\s:=\max(0,s-\frac 12-\frac 1q)$. For $j=1,2$, let $F_{j}$ denote two-real valued functions such that $\dx F_j= v_j$ for $v_j \in H^{\s}$ for $j=1,2$. Then, 
\begin{align}
\| J^{s}_x[ e^{ \pm \rho F_j}g]\|_{L^q} &\les (1 + \|v_j\|_{H^{\s}}^2) \|J^{s}_x g\|_{L^{q}}, \label{eFg1}\\
\| J^{s}_x[ ( e^{ \pm \rho F_1} -e^{ \pm \rho F_2})g]\|_{L^q} &\les (\|v_1-v_2\|_{H^{ \s}}+ \|e^{\pm i\rho F_1}-e^{\pm i \rho F_2}\|_{L^{\infty}}  ) \notag \\
& \hphantom{X}\times (1+\|v_1\|_{H^{\s}} + \|v_2\|_{H^{\s}})^3 \|J^{s}_x g\|_{L^{q}}. \label{eFg2}
\end{align}
\end{lemma}
\begin{proof}
This lemma was proved in \cite[Lemma 2.7]{MP} for $0\leq s\leq \frac 1q$ and in \cite[Lemma 3.3]{CLOP} when $0\leq s \leq \frac 12$. 
We extend these ranges to cover the range $\frac12<s\leq \frac 32$. 
We want to use the fractional Leibniz rule however there is an issue in making sense of $D^{s}_x \Pblo e^{\pm i\rho F_j}$ in $L^{\infty}$, especially when $s$ is low. See \cite[Section 2.2]{CFL}.
Namely, $\Pblo$ is defined via \eqref{PbloE} and not as a projection to frequencies $\{|\xi|\les 1\}$. 

For $N\gg 1$, we first claim that 
\begin{align}
\P_{N}[\Pblo e^{i\rho F }\cdot g] = \P_{N}[ \Pblo e^{i\rho F} \cdot \wt{\P}_{N}g]\label{MPeFg}
\end{align}
point-wise. Note that we use $\cdot$ to denote a product and to make clear which factors the projection operators are acting on. 
To see this, for any test function $h$, we have 
\begin{align}
\begin{split}
\jb{ h, \P_{N}[\Pblo e^{i\rho F} \P_{\gg N} g]} &= \jb{ \P_{\gg N}g \cdot  \P_{N} h, \Pblo e^{i\rho F}} \\
& = \jb{ \PbHI [ \P_{\gg N}g \cdot  \P_{N} h], \Pblo e^{i\rho F}} \\
& =  \jb{ \PbHI [ \P_{\gg N}g \cdot  \P_{N} h], \PbHI \Pblo e^{i\rho F}} =0,
\end{split} \label{argument}
\end{align}
where in the last equality we used \eqref{PloeiF}.
Similarly, we have that $\P_{N}[\Pblo e^{i\rho F}  \P_{\ll N}g] =0.$ Thus, \eqref{MPeFg} follows. Then, by writing 
\begin{align*}
\P_{N}[ \Pblo e^{i\rho F} \wt{\P}_{N}g](x) = \int_{\R} \psi_{N}(y) [\Pblo e^{i\rho F} \wt{\P}_{N}g](x-y)dy,
\end{align*}
and by the Littlewood-Payley square function theorem and Minkowski's inequality, we have 
\begin{align}
\begin{split}
\|D^{s}_x  \Pbhi( \Pblo e^{i\rho F} g)\|_{L^{q}}  & \sim \bigg\|   \bigg( \sum_{N \gg 1}N^{2s} | \P_{N}[\Pblo e^{i\rho F} g]|^{2} \bigg)^{\frac 12} \bigg\|_{L^{q}_x}  \\
& \les \int_{\R} |\psi_{N}(y)| \bigg\|  \bigg( \sum_{N \gg 1}N^{2s}  | \wt{\P}_{N} g(x-y)|^{2} \bigg)^{\frac 12} \bigg\|_{L^{q}_x} dy \\
& \les  \int_{\R} |\psi_{N}(y)| \bigg\|  \bigg( \sum_{N \gg 1} N^{2s}| \wt{\P}_{N} g|^{2} \bigg)^{\frac 12} \bigg\|_{L^{q}_x} dy  \les \|D^{s}_x g\|_{L^{q}_x},
\end{split} \label{PbloDsE}
\end{align}
where we used the Littlewood-Payley square function theorem again in the last inequality.

We now consider the contribution from $\Pbhi e^{ \pm i\rho F_j}$.
By the fractional Lebiniz rule, we have
\begin{align*}
\| D^s_x[ \Pbhi e^{\pm i\rho F_j} \cdot g]\|_{L^q} &\les \|D^s_x g\|_{L^q} \|\Pbhi e^{\pm i\rho F_j}\|_{L^{\infty}} + \|g\|_{L^{\infty}} \|D^s \Pbhi e^{iF}\|_{L^q} \\
& \les (1+ \|D^s \Pbhi e^{iF}\|_{L^q}) \|J^s_x g\|_{L^q}.
\end{align*}
When $s\leq \frac 12+\frac 1q$, Sobolev inequality implies 
\begin{align*}
\|D^s \Pbhi e^{iF}\|_{L^q} \sim \| D^{s+\frac 12-\frac 1q}  \Pbhi e^{iF}\|_{L^2 } \les  \| D^{s-\frac 12-\frac 1q}  \Pbhi [ v_j e^{\pm i \rho F_j}\|_{L^2 }  \les \|v_j\|_{L^2}.
\end{align*}
If instead $\frac 12+\frac 1q < s\leq 1 +\frac 1q$, then we additionally use the fractional Leibniz rule and \eqref{PbloDsE} to obtain
\begin{align*}
 \| D^{s-\frac 12-\frac 1q}  \Pbhi [ v_j e^{\pm i \rho F_j}\|_{L^2 }   
 & \les  \| D^{s-\frac 12-\frac 1q}  \Pbhi [ v_j  \Pbhi e^{\pm i \rho F_j}\|_{L^2 }   + \| D^{s-\frac 12-\frac 1q}  \Pbhi [ v_j \Pblo e^{\pm i \rho F_j}\|_{L^2 }    \\
 & \les  \|v_j\|_{H^{\max(0, s-\frac 12-\frac 1q)}} \big(  1+ \|D^{s-1-\frac 1q}\Pbhi[ v_j e^{\pm i\rho F_j}\|_{L^{2}} \big) \\
 &\les \|v_j\|_{H^{\max(0, s-\frac 12-\frac 1q)}}  \big( 1+  \|v_j\|_{L^2})
\end{align*}
Hence, when $0\leq s\leq 1$, we have  shown that 
\begin{align}
\|D^{s}_x [ e^{\pm i\rho F_j}g]\|_{L^q} \les (1+\|v_j\|_{H^{\max(0,s-\frac 12-\frac 1q)}}^2) \|J^s_x g\|_{L^{q}}. \label{eFg3}
\end{align}
Now we consider when $s>1$, for which we have 
\begin{align*}
\|D^{s}_x [ e^{\pm i\rho F_j}g]\|_{L^q}  &\sim \|D^{s-1}_x \dx  [ e^{\pm i\rho F_j}g]\|_{L^q}  \\
& \sim \|D^{s-1}_x   [  v_j e^{\pm i\rho F_j}g]\|_{L^q} + \|D^{s-1}_x   [ e^{\pm i\rho F_j} \dx g]\|_{L^q}
\end{align*}
For the second term, we apply \eqref{eFg1} since $s-1\leq \frac 12$. For the first term similar arguments as above splitting $e^{\pm i \rho F_j}$ using $\Pblo$ and $\Pbhi$ and using \eqref{PbloDsE} and fractional Leibniz rule, we obtain
\begin{align*}
 \|D^{s-1}_x   [  v_j e^{\pm i\rho F_j}g]\|_{L^q}  & \les \|v\|_{H^{s-\frac 12-\frac 1q}}(1+\|v\|_{L^2})\|g\|_{H^s}.
\end{align*}
Combining this with \eqref{eFg3} establishes \eqref{eFg1} in the range $0\leq s\leq \frac 32$.
Finally for the difference estimate \eqref{eFg2}, we follow a similar strategy as in \eqref{PbloDsE}. We omit the details.
\end{proof}

As we will need to use the dyadic spaces $\wt{L^{4}_{T,x}}$, Lemma~\ref{LEM:prodest} will not alone be sufficient. We need to refine the estimates by using frequency decompositions and gaining derivatives from the smoother terms $e^{\pm i \rho F}$.

\begin{lemma}\label{LEM:PNe}
Let $2\leq q<\infty$ and $0\leq s\leq \frac 32$. For $j=1,2$, let $F_{j}$ denote two-real valued functions such that $\dx F_j= v_j$ for $v_j \in L^2$ for $j=1,2$.
Let $\mathcal{V}:=\|v_1\|_{L^2}+\|v_2\|_{L^2}$, $V:= v_1-v_2$, and $G(F_1,F_2):=e^{\pm i \rho F_1}-e^{\pm i \rho F_2}$.
 Then, for any $N\in 2^{\N}$, it hold that
\begin{align}
\| \P_{N} e^{\pm i \rho F_1}\|_{L^q_x} & \les  N^{-2-\frac 1q} \mathcal{V}^4 + N^{-1-\frac 1q}(1+\mathcal{V})\mathcal{V} \|\P_{\ges N}v\|_{L^2} + N^{-1}\|\P_{ N}v\|_{L^q} \label{PNe1}\\
\| \P_{N} G(F_1,F_2)\|_{L^q_x} 
& \les  N^{-1}\|\wt{\P}_{N}V\|_{L^q}+N^{-1} \V \|\wt{\P}_{N}v_2\|_{L^q}\|G\|_{L^{\infty}}     \notag \\
& \hphantom{X} +N^{-2-\frac 1q} \V^3 (1+\V)^3  \Big\{  \|V\|_{L^2} + \V\|G\|_{L^{\infty}}   \Big\}  \notag \\
& \hphantom{X} + N^{-1-\frac 1q} (1+\V) \Big\{   \V (1+\V)\|\P_{\ges N}V\|_{L^2} + \V\|\P_{\ges N}v_2\|_{L^2} \|G\|_{L^{\infty}}    \notag \\
&\hphantom{XXXXXXXXX}+  (1+\V)( \|\P_{\ges N}v_1\|_{L^2}+\|\P_{\ges N}v_2\|_{L^2}  )\|V\|_{L^2}         \Big\}.   \label{PNe2}
\end{align}
\end{lemma}
\begin{proof}
We first establish \eqref{PNe1} in the case $q=\infty$. By Bernstein inequality, we have 
\begin{align}
\|\P_{N}e^{\pm i\rho F_1}\|_{L^{\infty}} 
& \les N^{-\frac 12} \| \P_{N}[ v_1 e^{\pm i\rho F_1}]\|_{L^{2}}\notag \\
& \les N^{-\frac 12} \big[   \| \P_{N}[  (\P_{\ll N}v_1) e^{\pm i\rho F_1}]\|_{L^{2}}  + \| \P_{N}[ (\P_{\ges N}v_1) e^{\pm i\rho F_1}]\|_{L^{2}}  \big] \notag\\
& \les N^{-\frac 12} \big[  \| \P_{\ll N}v\|_{L^{2}} \|\wt{\P}_{N}e^{\pm i\rho F_1}\|_{L^{\infty}} + \|\P_{\ges N}v\|_{L^2}    \big] \notag\\
&\les  \|v\|_{L^2} \| \wt{\P}_{N}e^{\pm i\rho F_1}\|_{L^2} + N^{-\frac 12} \|\P_{\ges N}v\|_{L^2} \notag \\
& \les  N^{-1} \|v\|_{L^2} \big\{\| \P_{\ll N}v\|_{L^{2}} \|\wt{\P}_{N}e^{\pm i\rho F_1}\|_{L^{\infty}} + \|\P_{\ges N}v\|_{L^2} \big\}  +N^{-\frac 12} \|\P_{\ges N}v\|_{L^2} \notag     \\
& \les N^{-\frac 32}\|v\|_{L^2}^{3} + N^{-\frac 12}(1+\|v\|_{L^2}) \|\P_{\ges N}v\|_{L^2} . \label{PNe3}
\end{align}
Now fix $2\leq q<\infty$. By employing similar arguments as above and using \eqref{PNe3}, we have 
\begin{align}
N\|\P_{N}e^{\pm i\rho F_1}\|_{L^{q}}  & \les  \|\P_{\ll N}v\|_{L^{q}}\|\wt{\P}_{N}e^{\pm i\rho F_1}\|_{L^{\infty}} + \|\wt{\P}_{ N}v\|_{L^q} +  \sum_{N_1\gg N} \|\P_{N_1}v\|_{L^{q}} \|\wt{\P}_{N_1}e^{\pm i\rho F_1}\|_{L^{\infty}} \notag \\
& \les N^{\frac 12-\frac 1q}\|v\|_{L^2} \|\wt{\P}_{N}e^{\pm i\rho F_1}\|_{L^{\infty}}+  \|\wt{\P}_{ N}v\|_{L^q} \notag  \\
&+\sum_{N_1\gg N} \big\{ N_{1}^{-1-\frac 1q}\|\P_{N_1}v\|_{L^2}\|v\|_{L^2}^3 + N_{1}^{-\frac{1}{q}}\|\P_{N_1}v\|_{L^2}(1+\|v\|_{L^2}) \|\P_{\ges N_1}v\|_{L^2}   \big\} \notag \\
& \les N^{-1-\frac 1q}\|v\|_{L^2}^{4} + N^{-\frac 1q}(1+\|v\|_{L^2}) \|v\|_{L^2} \|\P_{\ges N}v\|_{L^2} + \|\wt{\P}_{ N}v\|_{L^{q}}. \label{PNe22}
\end{align}
This proves \eqref{PNe1}. 
For \eqref{PNe2}, we have 
\begin{align}
\|\P_{N}G(F_1,F_2)\|_{L^q} & \les N^{-1} \|\P_{N}[ V e^{\pm i\rho F_1}]\|_{L^{q}} + N^{-1}\|\P_{N}[ v_2 G(F_1,F_2)] \|_{L^{q}},
\label{PNe2A}
\end{align}
and we are left to estimate both of the factors on the right-hand side of \eqref{PNe2A}. First, by a similar argument as in \eqref{PNe22} using \eqref{PNe3}, we have 
\begin{align}
\begin{split}
 \|\P_{N}[ V e^{\pm i\rho F_1}]\|_{L^{q}} & \les \|\wt{\P}_{N}V\|_{L^q} + N^{-1-\frac 1q}\mathcal{V}^3 \|V\|_{L^2}  +N^{-\frac{1}{q}} \|\P_{\ges N}v_1\|_{L^2} (1+\mathcal{V})\|V\|_{L^2}.
\end{split} 
\label{PNe2B}
\end{align}
For the second term on the right-hand side of \eqref{PNe2A}, we need an auxiliary estimate on $\P_{N}G(F_1,F_2)$ in $L^{\infty}_x$.
To start, we need a slight variant of \eqref{PNe3} which is proved in exactly the same way: 
\begin{align}
\begin{split}
\|\P_{N}[ g e^{\pm i\rho F_1}]\|_{L^{2}}  & \les N^{-\frac 32} \|v_1\|_{L^2}^{3}\|g\|_{L^2}  \\
& \qquad + N^{-\frac 12}(1+\|v_1\|_{L^2})\|\P_{\ges N} v_1\|_{L^2} \|g\|_{L^2} + \|\P_{\ges N}g\|_{L^2}. \label{PNe4}
\end{split} 
\end{align}
Similarly, using $G(F_1,F_2):=e^{\pm i\rho F_1}-e^{i\pm \rho F_2}$ and $V:=v_1-v_2$, we have 
\begin{align*}
\|\P_{N}[ v_2  G(F_1,F_2)]\|_{L^{2}}  & \les \mathcal{V} \|\wt{\P}_{N} G(F_1,F_2)\|_{L^{\infty}}  + \|\P_{\ges N}v_2\|_{L^2} \|G(F_1,F_2)\|_{L^{\infty}} \\
& \les N^{-\frac 12} \mathcal{V}  \big\{ \|\wt{\P}_{N}[ Ve^{\pm i\rho F_1}]\|_{L^2} + \|\wt{\P}_{N}[ v_2 G(F_1,F_2)]\|_{L^2}\big\}  \\
& \qquad +  \|\P_{\ges N}v_2\|_{L^2} \|G(F_1,F_2)\|_{L^{\infty}} \\
& \les N^{-1} \V (1+\V)^3\|V\|_{L^2} + N^{-1}\V^2 \|G(F_1,F_2)\|_{L^{\infty}} \\
& \qquad + N^{-\frac 12}\V \|\P_{\ges N}V\|_{L^2} +\|\P_{\ges N}v_2\|_{L^2} \|G(F_1,F_2)\|_{L^{\infty}}
\end{align*}
Then, by \eqref{PNe4}, 
\begin{align*}
\|\P_{N}G(F_1,F_2)\|_{L^{\infty}}& \les N^{-\frac 12} \|\P_{N}[ V e^{\pm i\rho F_1}]\|_{L^{2}} + N^{-\frac 12} \|\P_{N}[ v_2 G(F_1,F_2)]\|_{L^{2}} \\
& \les N^{-1} \V(1+\V)^3\|V\|_{L^2} +N^{-\frac 12}(1+\V)\|\P_{\ges N}V\|_{L^2} \\
& \qquad + N^{-\frac 12}\{ \|\P_{\ges N}v_2\|_{L^2} + N^{-\frac 12}\V^2\} \| G(F_1,F_2)\|_{L^{\infty}}. 
\end{align*}
Using this, we can obtain further decay in $N$:
\begin{align*}
\|\P_{N}G(F_1,F_2)\|_{L^{\infty}} & \les N^{-\frac 12}\|\P_{\ges N}V\|_{L^2} + N^{-\frac 12}\|V\|_{L^{2}} \|\wt{\P}_{N}e^{\pm i\rho F_1}\|_{L^{\infty}} \\
& \hphantom{X}  + N^{-\frac 12} \|\P_{\ges N}v_2\|_{L^2}\|G\|_{L^{\infty}}+N^{-\frac 12}\mathcal{V} \|\wt{\P}_{N}G\|_{L^{\infty}} \\
& \les N^{-\frac 12}(1+\V)^2  \|\P_{\ges N}V\|_{L^2}\\
& \hphantom{X} + N^{-\frac 12}(1+\V)\|\P_{\ges N}v_2\|_{L^2}\|G\|_{L^{\infty}} 
+ N^{-\frac 32} \V^{3}\|G\|_{L^{\infty}} \\
& \hphantom{X} + \big\{ N^{-\frac 32} \V^2 (1+\V)^3 + N^{-1}(1+\V)\|\P_{\ges N}v_2\|_{L^2}\big\} \|V\|_{L^2} . 
\end{align*}
This implies that 
\begin{align*}
\| \P_{N}[v_2 G]\|_{L^q} &\les \|\P_{\ll N}v_2\|_{L^q} \| \wt{\P}_{N}G\|_{L^{\infty}} + \|\wt{\P}_{N}v_2\|_{L^q} \|G\|_{L^{\infty}} + \sum_{N_1 \gg N} \|\P_{N_1}v_2\|_{L^q} \|\wt{\P}_{N_1}G\|_{L^{\infty}} \\
& \les N^{-\frac 1q} \V(1+\V)\Big\{ (1+\V)\|\P_{\ges N}V\|_{L^2} + \P_{\ges N}v_2\|_{L^2}\|G\|_{L^{\infty}} \Big\} \\
& \hphantom{X} + N^{-1-\frac 1q} \V^{3} \Big\{ \V\|G\|_{L^{\infty}}+(1+\V)^3 \|V\|_{L^2}\Big\} \\
& \hphantom{X} +N^{-\frac 12-\frac 1q}\V(1+\V)\|\P_{\ges N}v_2\|_{L^2}\|V\|_{L^2} +\| \wt{\P}_{N}v_2\|_{L^q}\|G\|_{L^{\infty}}.
\end{align*}
Combining this with \eqref{PNe2B} in \eqref{PNe2A}, we then obtain \eqref{PNe2}.
\end{proof}

\section{The gauge transform and properties of solutions}
\label{SEC:main}

\subsection{Fourier restriction norm spaces}

For $s,b\in \R$, and a dispersion relation $\o:\R\to \R$
we consider the Fourier restriction norm spaces $X_{\o}^{s,b}(\R\times \R)$ which are the completion of $\S(\R\times \R)$ under the norms \cite{BO93, BO932}:
\noi
\begin{align}
\| u\|_{X_{\o}^{s, b}(\R\times \R)} &= \big\| \jb{\tau-\o(\xi)}^{b}\jb{\xi}^{s}\ft u(\tau, \xi)\big\|_{L^2(\R\times \R)}, \label{Xsb}
\end{align}
We use the dispersion relations:
\begin{align*}
\o_{S_{\pm}}(\xi) = \pm\xi^2 \quad \text{and} \quad \o_{BO}(\xi) = \xi |\xi|.
\end{align*}
and write $X^{s,b}_{S_{\pm}}$ in place of $X^{s,b}_{\o_{S_{\pm}}}$, respectively, and similarly for $X^{s,b}_{BO}$. In particular, these spaces are adapted to the linear operators $S_{\pm}(t):=e^{\pm it\dx^2}$ and $e^{-it\H\dx^2}$, respectively.

We also define Besov-variations of these spaces by dyadically decomposing the modulation variable. We will do this for the homogenous variety in the modulation variable defining 
\begin{align*}
\| u\|_{\dot{X}_{\o}^{s,b}} = \big\|  |\tau -\o(\xi)|^b \jb{\xi}^s \ft u(\tau,\xi)\big\|_{L^2_{\tau,\xi}}.
\end{align*}
Two important such spaces are: 
\begin{align*}
\| u\|_{\dot{X}^{s,b,1}_{\o}} = \sum_{L\in 2^{\Z} } \| \Q_{L} u\|_{\dot{X}_{\o}^{s,b}}, \quad \|u\|_{\dot{X}^{s,b,\infty}_{\o}} =\sup_{L\in 2^{\Z}} \|\Q_{L}u\|_{\dot{X}^{s,b}_{\o}}. 
\end{align*}
We then have $\dot{X}^{s,b,1}_{\o}\subset \dot{X}^{s,b}_{\o}\subset \dot{X}^{s,b,\infty}_{\o}$.

In order to reach the $L^2$-endpoint for the BO part of \eqref{SBO}, we will use the $U^2$-$V^2$-spaces, which provide an effective replacement for the space $X^{s,\frac 12,1}_{\o}$. 
These spaces were first used in \cite{U2V21, U2V22, U2V23} and more details on their properties can be found in \cite{U2V22, U2V23, Obernotes}.
As we will only use these spaces for the gauged variable $w$, we only need to use the dispersion relation $\o_{S}$.

\begin{definition}\rm \label{DEF:UpVp}
Let $s\in \R$, $1\leq p<\infty$. Let $\mathcal{Z}$ be the collection of finite partitions $\{t_k\}_{k=0}^K$ of $\R$ with $-\infty< t_0< \ldots <t_{K}\leq \infty$. If $t_{K}=\infty$, we use the convention that $u(t_{K}):=0$ for $u:\R\to H^{s}(\R)$.

\noi
\textup{(i)} We define the space $V^{p}H^s=V^{p}(\R;H^s)$ as the collection of functions $u:\R\to H^s$ such that $\|u\|_{V^{p}H^s}<\infty$, where 
\begin{align*}
\|u\|_{V^p H^s} : = \sup_{  \{t_k\}_{k=0}^{K}\in \mathcal{Z} } \bigg( \sum_{k=1}^{K} \|u(t_k)-u(t_{k-1})\|_{H^s}^{p} \bigg)^{1/p}.
\end{align*}
We also define $V_{\text{rc}}^{p}(\R;H^s)$ to be the closed subspace of all right-continuous functions in $V^p H^s$ such that $\lim_{t\to -\infty}u(t)=0$.

\noi
\textup{(ii)}  A $U^p$-atom is defined by a step-function $a:\R\to H^s$ of the form
\begin{align*}
a(t) = \sum_{k=1}^{K} \phi_{k-1}\ind_{[t_{k-1},t_{k})}(t)
\end{align*}
where $\{t_{k}\}_{k=0}^{K}\in \mathcal{Z}$ and $\{ \phi_{k}^{K-1}\subset H^s$ with $\sum_{k=0}^{K-1} \|\phi_k\|_{H^s}^{p}=1$. We then define the atomic function space $U^p H^s= U^p(\R; H^s(\R))$ as the collection of functions $u:\R\to H^s$ of the form
\begin{align*}
u(t)=\sum_{j=1}^{\infty} \ld_j a_j (t) \quad \text{where}\,\, \{a_j\}_{j=1}^{\infty} \,\, \text{are} \,\,U^p-\text{atoms and} \,\, \{ \ld_j\}_{j=1}^{\infty}\in \l^1(\N;\mathbb{C})
\end{align*}
with the norm
\begin{align*}
\|u\|_{U^p H^s} : = \inf \big\{   \| \ld\|_{\l^1} \, :\, u=\sum_{j=1}^{\infty}\ld_j a_j, \,\, a_j \,\,\text{are} \,\, U^p H^s-\text{atoms}    \big\}
\end{align*}

\noi
\textup{(iii)} We define $U^p_{S_{-}}H^s=U^{p}_{S}H^s$ and $V^p_{S_{-}}H^s=V^p_{S}H^s$ to be the spaces of all functions $u:\R \to H^s$ such that the following norms are finite, respectively:
\begin{align*}
\| u\|_{U^p_{S}H^s} := \|S_{-}(-t) u\|_{U^p H^s} \quad \text{and} \quad \|u\|_{V^{p}_{S}H^s} := \|S_{-}(-t) u\|_{V^pH^s}.
\end{align*}
\end{definition}

We recall that the spaces $U^pH^s$, $V^p H^s$, and $V^p_{\text{rc}}H^s$ are Banach spaces (and so are $U^p_{S}H^s$ and $V_{S}^pH^s$). Moreover, the following embeddings hold: 
\begin{align}
U^{p}H^s \subset V^p_{\text{rc}}H^s \subset U^{q}H^s \subset L^{\infty}(\R;H^s) \label{Upembed}
\end{align}
for $1\leq p<q<\infty$.

We recall some further properties of the adapted function spaces $U^{p}_{S}H^s$ and $V_{S}^{p}H^s$.

\begin{lemma}\label{LEM:Upprops}
\textup{(i)} \textup{(Transference principle)} Suppose that we have 
\begin{align*}
\| T(S(t)\phi_1,\ldots, S(t)\phi_k)\|_{L^{p}_{t}L^{q}_{x}} \les \prod_{j=1}^{k} \|\phi_j\|_{L^2_x}
\end{align*}
for some $1\leq p,q\leq \infty$. Then, we have 
\begin{align*}
\| T(u_1, \ldots, u_k) \|_{L^{p}_{t}L^{q}_{x}} \les \prod_{j=1}^{k} \| u_j\|_{U^p_{S}L^2_x}
\end{align*}
\textup{(ii)} \textup{(Embeddings)} The following inclusions hold:
\begin{align}
\dot{X}^{s,\frac 12, 1} \subset U^2_{S}H^s \subset V^{2}_{S}H^s \subset \dot{X}^{s,\frac 12,\infty}. \label{U2Xsb}
\end{align}
\textup{(iii)} For any $M\in 2^{\Z}$ and $1\leq p<\infty$, it holds that 
\begin{align}
\| \Q_{\geq M} u\|_{L^{2}_{t,x}} &\les M^{-\frac 12} \|u\|_{V^{2}_{S}L^2_x}, \label{Qgain} \\
\| \Q_{\leq M} u\|_{V^{p}_{S}L^2_x},  \les \|u\|_{V^{p}_{S}L^2_x} \quad &\text{and} \quad \| \Q_{\leq M} u\|_{U^{p}_{S}L^2_x}  \les \|u\|_{U^{p}_{S}L^2_x} . \label{Qbddness}
\end{align}
\textup{(iv)}  Let $I \subset \R$ be an interval. Then it holds that 
\begin{align}
\|  \ind_{I} u \|_{V^p_{S}L^2_x} \leq 2\| u\|_{V^{p}_{S}L^2_x}  \quad \text{and}\quad \|  \ind_{I} u \|_{U^p_{S}L^2_x} \leq \| u\|_{U^{p}_{S}L^2_x}. \label{Indbdd}
\end{align}
\end{lemma}

\begin{proof}
The properties (i), (ii), (iii) are well-known and can be found in, for example, \cite{U2V22}, while \eqref{Indbdd} follow from the definitions of the $U^p$, $V^p$ spaces and \eqref{Upembed}.
\end{proof}

We will use the following versions of the adapted spaces in Definition~\ref{DEF:UpVp} (iii) with dyadic spatial frequency decomposition.

\begin{definition}\rm
\noi
\textup{(i)} Let $s\in \R$. We define $Z^{s}(\R)$ to be the space of all tempered distributions $u:\R\to H^s(\R)$ such that $\|u\|_{Z^{s}(\R)} <\infty$, where the $Z^s$-norm is defined by
\begin{align}
\| u\|_{Z^s(\R)}: = \bigg( \sum_{ \substack{N\geq 1 \\ \text{dyadic} } } N^{2s} \|\P_{N}u\|_{U^2_{S}L^2}^{2} \bigg)^{\frac 12}. \label{ZsT}
\end{align}

\noi
\textup{(ii)}  Let $s\in \R$. We define $Y^s(\R)$ as the space of all tempered distributions $u:\R\to H^s(\R)$ such that for every $N\in \N$, the map $t\mapsto \P_{N}u(t)$ belongs to $V^{2}_{\text{rc}}H^s$ and $\|u\|_{Y^s(\R)} <\infty$, where 
\begin{align*}
\|u\|_{Y^s(\R)}: = \bigg( \sum_{ \substack{N\geq 1 \\ \text{dyadic} } } N^{2s} \|\P_{N}u\|_{V^2_{S}L^2}^{2} \bigg)^{\frac 12}.
\end{align*}
\end{definition}

We then have the embeddings:
\begin{align}
U^2_{S}H^s \subset Z^{s}\subset Y^s \subset V^{2}_{S}H^s \subset U^{p}_{S}H^s \label{ZYembed}
\end{align}
for $p>2$. It is then immediate from \eqref{ZYembed} that \eqref{Qgain} and \eqref{Qbddness} also hold true for $Z^s$ and $Y^s$.

We now state linear and bilinear estimates adapted to the these spaces.

\begin{lemma}
\textup{(i)}  For $2\leq p \leq 6$, we have
\begin{align}
\| u\|_{\wt{L^{p}_{t,x}}} & \les \|u\|_{Y^{0}}, \label{L4}
\end{align}
\textup{(ii)}  Suppose that $N_1,N_2\in 2^{\N}$ such that $N_1 \gg N_2$. Then,
\begin{align}
\begin{split}
\| \P_{N_1} u \, \P_{N_2} v\|_{L^{2}_{t,x}} \les N_1^{-\frac 12} \min\big( & \| \P_{N_1}u\|_{U^{2}_{S}L^2_x} \|\P_{N_2}v\|_{X^{0,\frac 12+}_{S_{\pm}}},   \| \P_{N_1}u\|_{X^{0,\frac 12+}_{S_{\pm}}} \|\P_{N_2}v\|_{U^2_{S}L^2_x}, \\ 
&   \| \P_{N_1}u\|_{U^{2}_{S}L^2_x} \| \P_{N_2}v\|_{U^{2}_{S}L^2_x} ,\|\P_{N_1}v\|_{X^{0,\frac 12+}_{S_{\pm}}}  \|\P_{N_2}v\|_{X^{0,\frac 12+}_{S}} \big)
\end{split} \label{bilin1}
\end{align}
We also have the variant:
\begin{align}
\| \P_{N}[ u \cj{v}]\|_{L^{2}_{t,x}} \les N^{-\frac 12} \| u\|_{X^{0,\frac 12+}_{S}} \|v\|_{X^{0,\frac 12+}_{S}}. \label{bilin2}
\end{align}
\end{lemma}
\begin{proof}
We begin with \eqref{L4}. From the standard Strichartz estimate for the Schr\"{o}dinger equation 
\begin{align*}
\| \P_{N} S(t)\phi \|_{L^6_{t,x}} \les  \|\P_{N}\phi\|_{L^2_x},
\end{align*}
 the transference principle (Lemma~\ref{LEM:Upprops}), and \eqref{Upembed}, we have
\begin{align*}
\|\P_{N} u\|_{L^{6}_{t,x}} \les \|\P_{N}u\|_{U^{6}_{S}L^2_x} \les \|\P_{N}u\|_{V^{2}_{S}L^2_x}.
\end{align*}
Squaring both sides, summing over $N$, and interpolating with the trivial $p=2$ case, then yields \eqref{L4}. 
For \eqref{bilin1}, this follows from transference with the bilinear Strichartz estimate \cite{BO98, OT}
\begin{align}
\max_{\pm_1, \pm_2\in \{\pm1\}}\| S_{\pm_1}(t)\P_{N_1}\phi_1 \cdot S_{\pm_2}(t)\P_{N_2}\phi_2\|_{L^2_{t,x}} \les N_1^{-\frac 12} \|\phi_1\|_{L^2_x}\|\phi_2\|_{L^2_x}. \label{bilin0}
\end{align}
For the mixed bounds with one factor in $U^2_{S}L^2_x$ and the other in $X^{s,\frac 12+}$, one can first extend \eqref{bilin0} to the setting where $S(t)\phi_1$ is replaced by a space-time function $u\in U^2_S L^2_x$ using the atomic definition in Definition~\ref{DEF:UpVp}, and then replace $S(t)\phi_2$ by $v\in X^{0,\frac 12+}_{S}$ using the standard argument for transference in $X^{s,b}$. See \cite{TAObook}.
\end{proof}

Given a time interval $I \subset \R$, we define localised in time versions of these spaces as follows: if $u: I \times \M \to \C$, then 
\begin{align}
\|u\|_{B_{I}}: =\inf\{ \|\wt{u}\|_{B} \, : \, \wt{u}:\R\times \R \to \C, \,\, \wt{u}\vert_{I\times \R} = u\}, \label{localspace}
\end{align}
where $B$ is any one of the spaces $\{X_{\o}^{s,b}, U^2_{S}H^{s}_{x},  Z^{s}, Y^{s}\}$. We use the notation $U^{2}_{S;I}H^{s}_x =(U^{2}_{S}H^{s}_x)_{I}$. 
When $I=[0,T)$ for some $T>0$, we denote the spaces $B_{I}=B_{T}$. 

We recall the following linear estimates related to the Fourier restriction norm spaces. See~\cite{MP}, for example.

\begin{lemma}\label{LEM:linXsb}
Let $s, b\in \R$, $0<T\leq 1$.
\noi
\textup{(i)}  Let $0<\dl<\frac 12$. Then,
\begin{align}
\|  e^{it \o(\dx)} f\|_{X_{\o;T}^{s,b}} \les \|f\|_{H^{s}},  \quad \text{and} \quad 
\bigg\| \int_{0}^{t} e^{i(t-t')\o(\dx)} g(t')dt' \bigg\|_{X_{\o;T}^{s,\frac 12+\dl}} & \les T^{\dl} \|g\|_{X_{\o;T}^{s, -\frac 12 +2\dl}}
.\label{lin2}
\end{align}
\textup{(ii)} Given $0\le \dl <\frac18$, the following estimate holds
\begin{align}
\|u\|_{L^{4}_{T,x}}& \les \|u\|_{\wt{L^{4}_{T,x}}} \les  T^{\frac 14-2\dl - } \|u\|_{X_{S; T}^{0,\frac 12-2\dl}}. 
 \label{L4Strich}
\end{align}
\end{lemma}

Now we discuss properties of the time-localised spaces $Z^{s}$ and $Y^s$. 
We recall from \cite[Lemma A.1 and Remark A.2]{BOP} that for the norm $\|\cdot \|_{U^{2}_{S;I}L^2_x}$ defined via \eqref{localspace},
the infimum is achieved by the extension $\wt{u}= \ind_{I} (t) u$. It then follows from \eqref{Upembed} that for any interval $I=[a,b)$, it holds that 
\begin{align}
\| u\|_{L^{\infty}_{I}L^2_x} & \les \| \ind_{I}(t) u\|_{L^{\infty}(\R;L^2_x)} \les \| \ind_{I}(t) w\|_{U^{2}_{S}L^2_x} \sim \|u\|_{U^{2}_{S;I}L^2_x}. \label{LinftyU2}
\end{align}

By a similar argument as in \eqref{LinftyU2}, we have
\begin{align}
\| u\|_{L^{\infty}_{I}H^s_{x}} \les \| u\|_{Z^{s}_{T}}. \label{LinftyZs}
\end{align}
Thus, for the time localised spaces, the following embeddings hold: for $b>\frac 12$,
\begin{align}
X_{\o;T}^{s,b} \subset C([0,T);H^{s}(\R)) \quad \text{and} \quad Z_T^s \subset L^{\infty}([0,T);H^{s}(\R)) \label{YsCTHs}
\end{align}
We also define the norm $N^{s}(I)$ by
\begin{align*}
\| F\|_{N^{s}(I)} := \bigg\| \int_{0}^{t} S(t-t') F(t')dt' \bigg\|_{Z^{s}(I)}.
\end{align*}
We will use the resolution space $X^{s+\frac 12,\frac 12+}_{S; T}$ for the Schr\"{o}dinger part and $Z^{s}_{T}\cap C([0,T);H^{s}(\R))$ for the gauged variable for the BO part of \eqref{SBO}, namely, in the Banach subspace of continuous functions in time within $Z^{s}_{T}$.
To this end, we need the following linear estimates from \cite{U2V22, U2V23}.

\begin{lemma}\label{LEM:linZsb}
Let $s\geq 0$ and $0<T\leq \infty$. Then, we have 
\begin{align*}
 \|S(t)\phi\|_{Z^{s}([0,T))} \leq \|\phi\|_{H^s} \quad \text{and}\quad
 \|F\|_{N^{s}([0,T))}  \leq \sup_{  \substack{ h\in Y^{-s}([0,T)) \\ \|h\|_{Y^{-s}([0,T))}=1 }  } \bigg|\int_{0}^{T} \int_{\R} F(t,x) \cj{ h(t,x)} dx dt\bigg|
\end{align*}
for all $\phi\in H^s(\R)$ and $F\in L^1([0,T);H^{s}(\R))$.
\end{lemma}

Lastly, in order to perform a continuity argument later, we need to understand the continuity in time 
of the time restricted norms. We recall the following from \cite[Lemma A.8]{BOP}.

\begin{lemma} \label{LEM:Zcts}
Let $s\in \R$ and $I=[a,b)\subset \R$. Given $v\in Z^{s}(I)\cap C(I;H^{s}(\R))$, the mapping $t\in I\mapsto \|v\|_{Z^{s}([a,t))}$ is continuous.
\end{lemma}

\begin{remark}\label{RMK:Ncts} \rm
We also recall that the mapping $T\mapsto \| \ind_{[0,T)} F\|_{X^{s,b'}}$ for any $b\leq 0$, $s\in \R$, and smooth $F$ is continuous. For instance see \cite[Remark 3.6]{CLOP} in the case $[0,T]$, which easily applies to the case of open right end-point $[0,T)$. \end{remark}

%


\subsection{Gauge transformation} \label{SEC:Gauge}

Following \cite{BP, MP}, 
we define the gauged variables 
\begin{align}
W := \P_{+,\text{hi}}[ e^{-i  \rho F[u,v]} ] \qquad \text{and} \qquad w:=\dx W, 
\label{gauge}
\end{align}
where $F[u,v]$ denotes the primitive of $v$ as in \eqref{F} and satisfying \eqref{Feqn}.
Before we state the equation that $w$ satisfies, we first note that we will need to also make modifications to the nonlinearity in the Schr\"{o}dinger part of the system \eqref{SBO}.
To this end, we recall the recovery formula from \cite{MP}: 
\begin{align*}
v= \dx F = e^{i\rho F} e^{-i\rho F} \dx F  =\tfrac{i}{\rho} e^{i\rho F} w + e^{i\rho F} \P_{-,\text{hi}}[ e^{-i\rho F}v] + e^{i\rho F} \Pblo[ e^{-i\rho F} v],
\end{align*}
so that 
\begin{align}
\PbHIp v = \tfrac{i}{\rho} \PbHIp[ e^{i\rho F} w] + \PbHIp[ \Pbhip e^{i\rho F}   \P_{-,\text{hi}}(e^{-i\rho F}v)]  +\PbHIp [ e^{i\rho F} \Pblo( e^{-i\rho F} v)]. 
\label{recovery1}
\end{align}
In order to simplify the first term on the right-hand side of \eqref{recovery1}, we write
\begin{align}
\PbHIp v = \tfrac{i}{\rho} e^{i\rho F} \PbHI w + R(v,w), \label{recovery2}
\end{align}
where 
\begin{align}
\begin{split}
R(v,w) & = \tfrac{i}{\rho} \PbLO[ e^{i\rho F} w] +\tfrac{i}{\rho} \P_{-,\text{HI}}[ e^{i\rho F} w] + \tfrac{i}{\rho} e^{i\rho F} \PbLO w \\
& \hphantom{X} + \PbHIp[ \Pbhip e^{i\rho F} \cdot  \P_{-,\text{hi}}(e^{-i\rho F}v)]  +\PbHIp [ e^{i\rho F} \Pblo( e^{-i\rho F} v)].
\end{split} \label{R}
\end{align}
We will use the decomposition \eqref{recovery2} to rewrite the nonlinearity of the Schr\"{o}dinger part $uv$ in \eqref{SBO} to observe finer smoothing properties. See \eqref{ueq} below.

\begin{lemma}\label{LEM:gaugeeqns}
Let $T>0$, $(u,v) \in C([0,T);H^{\infty}(\R))^2$ be a smooth solution of~\eqref{SBO}.
Then, the variables $(u,w)$ defined in \eqref{gauge} satisfy:
\begin{align}
\dt w +i\dx^2 w &= 2\rho \dx \Pbhip[ (\dx^{-1}w) \P_{-}\dx v] \notag  \\
&\hphantom{X}+2\rho \dx \Pbhip[ \Pblo(e^{-i\rho F}) \P_{-}\dx v]-i\rho\dx  \Pbhip[ e^{-i\rho F} |u|^2]
\label{weqn},\\
\dt u - i\dx^2 u &= -\tfrac{1}{\rho} e^{-i\rho F} u\cj{\PbHI w}- i u\cj{R(v,w)} -i u\PbLO  v -i u \P_{+,\textup{HI}} v -i \be |u|^2 u. \label{ueq}
\end{align}
\end{lemma}
\begin{proof}
Using \eqref{Feqn}, we see that 
\begin{align*}
\dt W -\H \dx^2 W &  =\dt W + i\dx^2 W \\
& = -i\rho \Pbhip \Big[ e^{-i\rho F} \big\{ \dt F + i\dx^2 F +\rho (\dx F)^2 \big\} \Big] \\
& = 2\rho \Pbhip[ e^{-i\rho F} \P_{-}\dx v] -i\rho \Pbhip[ e^{-i\rho F} |u|^2] \\
& = 2\rho \Pbhip[ W \P_{-}\dx v] +2\rho \Pbhip[ \Pblo(e^{-i\rho F}) \P_{-}\dx v]-i\rho \Pbhip[ e^{-i\rho F} |u|^2],
\end{align*}
where in the last equality, we used that $\Pbhip[ \P_{-,\text{hi}}f \cdot \P_{-}g]=0$ for any $f,g$.
Lastly, \eqref{ueq} follows from using \eqref{recovery2} and that since $v$ is real-valued, $\P_{-\text{HI}}v= \cj{ \PbHIp v}$.
\end{proof}

\subsection{BO-part regularity properties}

First, we establish the $X^{s,b}$-regularity and $\wt{L^{4}_{T,x}}$-integrability of the BO part $v$ of a solution to SBO \eqref{SBO}. 

\begin{lemma}\label{LEM:CCMXsb}
Let  $s\geq 0$, $0<T\leq 1$, and $(u,v)$ be a $H^{\infty}(\R)\times H^{\infty}(\R) $-solution to \textup{SBO} \eqref{SBO} on $[0,T)$. Then, 
\begin{align}\label{vXsb}
\begin{split}
\sup_{0\leq \ta \leq 1} \|v\|_{X^{s-\ta,\ta}_{\textup{BO}; T}} 
&
\les 
\| v\|_{L^\infty_TH^s_x}+\|  v\|_{L^4_{T,x}}
\| J_x^{s} v\|_{L^4_{T,x}}
 +\|  u\|_{L^4_{T,x}}
\| J_x^{s} u\|_{L^4_{T,x}}
\end{split}
\end{align}
Moreover, for $0\leq s<\frac{7}{4}$, it holds that
\begin{align}
\| J^{s}_x v\|_{\wt{L^{4}_{T,x}}}  &\les  T^{\frac 14}\|v\|_{L^{\infty}_{T}L^2_x} + (1+\|v\|^{4}_{L^{\infty}_{T}L^2_x})\|w\|_{Y^{s}_{T}}  \notag \\
&\qquad +(1+\|v\|_{L^{\infty}_{T}L^2_x})^2 \big(  \|w\|_{Y^{0}_{T}} + T^{\frac 14}\|v\|_{L^{\infty}_{T}L^2_x}  \big) \|v\|_{L^{\infty}_{T}H^s_x}  
\label{vL4Hs} \\
\| v\|_{L^{\infty}_{T}H^s_x} & \les \|v_0\|_{L^2_x}+(1+\|v\|_{L^{\infty}_{T}H^{\max(0,s-1+)}_x}^4)\|w\|_{Z^{s}_{T}}+T(  \|u\|_{L^{\infty}_{T}L^2_x}^2 + \|v\|_{L^{\infty}_{T}L^2_x}^{2}) \notag  \\
& \qquad + (1+\|v\|_{L^{\infty}_{T}L^2_x})^4 \|v\|_{L^{\infty}_{T}L^2_x} \|v\|_{L^{\infty}_{T}H^s_x}.
\label{vL4Hs2}
\end{align}
\end{lemma}

\begin{proof}
We argue as in \cite[Proposition 3.2]{MP}. 
The main point is that for a suitable extension $\wt{v}$ on $[0,T]$ of $v$, it holds that 
\begin{align*}
\sup_{0\le \ta \le 1} 
\| \wt{v}\|_{X_{\text{BO}}^{ s-\ta, \ta}} \les \| \dt v -\H \dx^2 v\|_{L^{2}_{T}H^{ s-1}_x} + \|v\|_{L^{\infty}_{T}H^{ s}_x}.
\end{align*}
Inserting the equation for $v$ in \eqref{SBO} and using the fractional Leibniz rule, we obtain \eqref{vXsb}.

We move onto \eqref{vL4Hs}. As $v$ is real-valued, we have 
\begin{align}
\| J^{s}_x v\|_{\wt{L^{4}_{T,x}}} = 2\| J^{s}_x  \P_{+}v\|_{\wt{L^{4}_{T,x}}}  &\les \|\P_{+,\text{LO}}v\|_{L^4_T L^2_x} + \| D^{s}_x \PbHIp v\|_{\wt{L^{4}_{T,x}}}  \notag \\
& \les T^{\frac 14}\| \P_{+,\text{LO}} v\|_{L^{\infty}_{T}L^2_x}+ \| D^{s}_x \PbHIp v\|_{\wt{L^4_{T,x}}} . \label{vLpLq}
\end{align}
For the second term $D^{s}_x \PbHIp v$, we replace $ \PbHIp  v$ by \eqref{recovery1} and estimate each of the ensuing terms separately. 

To this end, let $N\gg 1$ be dyadic. By frequency considerations and the triangle inequality, we have 
\begin{align}
\| \P_{N}[ e^{i\rho F} w]\|_{L^{4}_{T,x}} & \les \|\wt{\P}_{N}w\|_{L^{4}_{T,x}} + \| \wt{\P}_{N}e^{i\rho F}\cdot \P_{\ll N}w\|_{L^{4}_{T,x}} + \|\P_{\ges N}e^{i\rho F} \cdot \P_{\gg N}w\|_{L^{4}_{T,x}}. \label{L4tx3}
\end{align}
By Bernstein's inequality, we have $\| \P_{\ges N}e^{i\rho F}\|_{L^{\infty}_{T,x}} \les N^{-\frac 12}\|v\|_{L^{\infty}_{T}L^{2}_{x}}$.
Therefore, with \eqref{L4}, we obtain
\begin{align*}
\|\P_{\ges N}e^{i\rho F} \cdot \P_{\gg N}w\|_{L^{4}_{T,x}}. \les N^{-\frac 12-s} \|v\|_{L^{\infty}_{T}L^2_x} \|J^{s}_x w \|_{L^{4}_{T,x}} \les N^{-\frac 12-s} \|v\|_{L^{\infty}_{T}L^2_x} \|w\|_{Y^{s}_{T}}.
\end{align*}
Next by \eqref{PNe1}, we 
\begin{align*}
\|& \wt{\P}_{N}e^{i\rho F}\cdot \P_{\ll N}w\|_{L^{4}_{T,x}}  \\
&\les \| \P_{\ll N} w\|_{ L^{4}_{T}L^{\infty}_{x}} 
\|\wt{\P}_{N}e^{i\rho F} \|_{L^{\infty}_{T}L^{4}_{x}}  \\
& \les  N^{\frac 14}\|w\|_{L^{4}_{T,x}}  
\big\{  N^{-2-\frac{1}{4}}\|v\|_{L^{\infty}_{T}L^{2}_x}^{4}  \\
& \hphantom{XXXXXXXXX}+ N^{-1-\frac 14} \|v\|_{L^{\infty}_{T}L^{2}_x} (1+\|v\|_{L^{\infty}_{T}L^{2}_x})\|\P_{\ges N}v\|_{L^{\infty}_{T}L^2_x} +N^{-1}\|\P_{N}v\|_{L^{\infty}_{T}L^4_x} \big\} \\
& \les \|w\|_{Y_{T}^{0}} \big\{ N^{-2} \|v\|_{L^{\infty}_{T}L^{2}_x}^{4}  + N^{-1-\max(0,s-1+)} \|v\|_{L^{\infty}_{T}L^{2}_x} (1+\|v\|_{L^{\infty}_{T}L^{2}_x})\|v\|_{L^{\infty}_{T}H^{\max(0,s-1+)}_x}\Big\} \\
& \quad + N^{-\frac 12-s}\|w\|_{Y^{0}_{T}} \| v\|_{L^{\infty}_{T}H^{s}_x}.
\end{align*}
Combining these estimates, we find that 
\begin{align}
\| J^{s}_x &\PbHIp[ e^{i\rho F} w]\|_{\wt{L^{4}_{T,x}}} \notag  \\
&\les (1+ \|v\|_{L^{\infty}_{T}L^2_x}) \|w\|_{Y^{s}_{T}} 
+ \|w\|_{Y^{0}_{T}}\big(  \|v\|_{L^{\infty}_{T}L^2_x}^4 +(1+ \|v\|_{L^{\infty}_{T}L^2_x})^2 \|v\|_{L^{\infty}_{T}H^s_x}    \big). \label{L4tx1}
\end{align}
Next, by H\"{o}lder, Bernstein, and \eqref{PNe1}, we have 
\begin{align}
 \| &\P_{N}\PbHIp[ \Pbhip e^{i\rho F} \cdot  \P_{-,\text{hi}}(e^{-i\rho F}v)]\|_{L^{4}_{T,x}} \notag  \\
 & \les  \sum_{N_1 \ges N\vee N_2} \|\P_{N_1}e^{i\rho F}\|_{L^{\infty}_{T}L^4_x} \|\P_{N_2}(e^{-i\rho F}v)\|_{L^{4}_{T}L^{\infty}_{x}} \notag \\
 & \les T^{\frac 14} \|v\|_{L^{\infty}_{T}L^2_x} \sum_{N_1 \ges N\vee N_2} N_{2}^{\frac 12}\big\{ N_1^{-\frac{9}{4}} \|v\|_{L^{\infty}_{T}L^2_x}^{4} +N_1^{-\frac{5}{4}-s} \|v\|_{L^{\infty}_{T}L^2_x}(1+\|v\|_{L^{\infty}_{T}L^2_x})  \|v\|_{L^{\infty}_{T}H^s_x} \notag \\
& \hphantom{XXXXXXXXXXXXXX}  + N^{-1-s} \|v\|_{L^{\infty}_{T}H^s_x}\big\} \notag \\
&\les T^{\frac 14} \|v\|_{L^{\infty}_{T}L^2_x} \big( \|v\|_{L^{\infty}_{T}L^2_x} ^4 + (1+\|v\|_{L^{\infty}_{T}L^2_x} )^2 \|v\|_{L^{\infty}_{T}H^s_x} \big). \label{L4tx2}
\end{align} 
A similar bound holds for the term $\PbHIp [ e^{i\rho F} \Pblo( e^{-i\rho F} v)]$. Combining \eqref{L4tx1} and \eqref{L4tx2} then proves \eqref{vL4Hs}.

We move onto \eqref{vL4Hs2}. As in \eqref{vLpLq}, we have 
\begin{align*}
\| v\|_{L^{\infty}_{T}H^s_x} & \les \|\P_{+,\text{LO}}v\|_{L^{\infty}_{T}L^2_x} + \|D^s_x \PbHIp v\|_{L^{\infty}_{T}L^2_x}.
\end{align*}
In the first term on the right-hand side above, we no longer gain a factor of $T$, so in order to complete the bootstrap argument later on, we need to obtain a higher-homogeneity of terms here.
To this end, we replace $v$ by the Duhamel formula:
\begin{align*}
\P_{+, \text{LO}}v(t) = S(t) \PbLO v_0 + \int_{0}^{t} S(t-t') \PbLO \dx( |u(t')|^2 -\rho v(t')^2)dt'.
\end{align*}
Then, by Bernstein's inequality, we have  
\begin{align*}
\| \P_{+, \text{LO}}v\|_{L^{\infty}_{T}L^2_x} \les \|v_0\|_{L^2_x} + T ( \|u\|_{L^{\infty}_{T}L^2_x}^2 + \|v\|_{L^{\infty}_{T}L^2_x}^{2}).
\end{align*}

By \eqref{eFg1} and \eqref{LinftyU2}, 
\begin{align*}
\| J^{s}_x [ e^{i \rho F}w]\|_{L^{\infty}_{T}L^2_x} \les (1 + \|v\|_{L^{\infty}_{T}H_x^{\max(0,s-1)}}^4) \| w\|_{L^{\infty}_{T}H^s_x} \les (1 + \|v\|_{L^{\infty}_{T}H_x^{\max(0,s-1)}}^4) \| w\|_{Z^{s}_{T}}.
\end{align*}
Next, by the fractional Leibniz rule, Bernstein's inequality and \eqref{PNe1}, we obtain
\begin{align*}
\| J^{s}\PbHIp[ \Pbhi e^{i\rho F} \cdot  \Pblo( e^{-i\rho F} v)]\|_{L^{\infty}_{T}L^{2}_x} 
&\les  \|J^{s}\Pbhi e^{i\rho F}\|_{L^{\infty}_{T}L^{2}_x}  \|v\|_{L^{\infty}_{T}L^2_x} \\
& \les  \|v\|_{L^{\infty}_{T}L^2_x} \sup_{t\in [0,T]}  \bigg(\sum_{N\geq 2} N^{s} \|\P_{N}e^{i\rho F}\|_{L^2_x} \bigg) \\
& \les  \|v\|^{5}_{L^{\infty}_{T}L^2_x} + \|v\|_{L^{\infty}_{T}L^2_x} \sup_{t\in [0,T]}\sum_{N\geq 2} N^{s-1}\|\P_{N}v\|_{L^2}   \\
& \quad + (1+\|v\|_{L^{\infty}_{T}L^2_x}) \|v\|_{L^{\infty}_{T}L^2_x}\sup_{t\in [0,T]} \sum_{N} N^{s-\frac{3}{2}} \|\P_{\ges N}v\|_{L^2}  \\
& \les  (1+\|v\|_{L^{\infty}_{T}L^2_x})^{4} \|v\|_{L^{\infty}_{T}L^2_x} \|v\|_{L^{\infty}_{T}H_x^{\max(0,s-1+)}}.
\end{align*}
Next, by a dyadic decomposition, Bernstein's inequality and \eqref{PNe1}, we have 
\begin{align*}
\|J_{x}^{s}\PbHIp[ \Pbhip e^{i\rho F}\cdot &\P_{-,\text{hi}}[e^{i\rho F}v]] \|_{L^{2}_{x}} 
 \les \sum_{N_1 \ges N\vee N_2} N^s \|\P_{N_1}e^{i\rho F}\|_{L^2_x} \|\P_{N_2}[ e^{i\rho F}v]\|_{L^{\infty}_x} \\
& \les \|v\|_{L^{2}_x}^{5} + \|v\|_{L^{2}_x} (1+\|v\|_{L^{2}_x})^2\sum_{N_1 \ges N\vee N_2}  N^{s}N_{2}^{\frac 12} N_{1}^{-1} \|\P_{\ges N_1}v\|_{L^{2}_x} \\
&\les \|v\|_{L^{2}_x}^{5} + \|v\|_{L^{2}_x}(1+\|v\|_{L^{2}_x})^2  \|v\|_{H^{s}_x}.
\end{align*}
Taking $L^{\infty}_{T}$ of both sides, we then obtain
\begin{align*}
\|J^{s}\PbHIp[ \Pbhip e^{i\rho F}\cdot \P_{-,\text{hi}}[e^{i\rho F}v]] \|_{L^{\infty}_{T}L^2_x} \les  (1+\|v\|_{L^{\infty}_{T}L^2_x})^4 \|v\|_{L^{\infty}_{T}L^2_x} \|v\|_{L^{\infty}_{T}H^s_x}.
\end{align*}
This completes the proof of \eqref{vL4Hs2}.
\end{proof}

We also have corresponding difference estimates.

\begin{lemma}\label{LEM:Xsbdiffs}
Let $0\leq s <\frac{3}{2}$ and $0<T\leq 1$. For $j=1,2$, let $(u_j,v_j) \in H^{\infty}(\R)\times H^{\infty}(\R)$ be two solutions to \eqref{SBO} on $[0,T)$. Denote by $w_j=w_j [u_j]$ the corresponding gauged functions given in \eqref{gauge}
 Let $U:=u_1-u_2$, $V:=v_1-v_2$, $W:=w_1-w_2$, $G:= e^{i\rho F_1}-e^{i\rho F_2}$, and for $0\leq \s\leq s$, $\V_{\s}:= \|v_1\|_{L^{\infty}_{T}H^{\s}_x}+ \|v_2\|_{L^{\infty}_{T}H^{\s}_x}$. Then, 
 \begin{align}
\sup_{0\leq \ta \leq 1} \|V\|_{X^{s-\ta,\ta}_{\textup{BO};T}}  \les & \|V\|_{L^{\infty}_{T}H^{s}_x} + \bigg( \sum_{j=1}^{2} \|v_j\|_{L^{4}_{T,x}}\bigg) \|J_x^{s}V\|_{L^{4}_{T,x}}  +\bigg( \sum_{j=1}^{2} \|u_j\|_{L^{4}_{T,x}}\bigg)\|J_x^{s}U\|_{L^{4}_{T,x}} \notag \\
& + \bigg( \sum_{j=1}^{2} \|J^{s}_{x} v_j\|_{L^{4}_{T,x}}\bigg) \|V\|_{L^{4}_{T,x}}+\bigg( \sum_{j=1}^{2} \|J^s_x u_j\|_{L^{4}_{T,x}}\bigg)\|U\|_{L^{4}_{T,x}} \label{Xsbdiff}
\end{align}
\begin{align}
 \| J^{s}_{x}V\|_{\wt{L^4_{T,x}}}  &\les    \|w_1\|_{Y^{s}_{T}} \big\{ \|V\|_{L^{\infty}_{T}L^2_x} + \V_{0}\|G\|_{L^{\infty}_{T,x}} \big\} +(1+\V_0)\|W\|_{Y^{s}_{T}} \notag \\
&\quad + (1+\V_{0})^3 \|V\|_{L^{\infty}_{T}H^s_x} + (1+\V_{0})^5 \V_{s} \|V\|_{L^{\infty}_{T}L^2_x} \notag \\
& \quad+ \V_{0}(1+\V_0)^6 \V_{s} \|G\|_{L^{\infty}_{T,x}} + \|W\|_{Y^{0}_{T}}\{ \V_0^4 +(1+\V_0)^2 \V_s\}.
\label{L4diff} \\
\|V\|_{L^{\infty}_{T}H^s_x}& \les  \|V(0)\|_{L^{2}_x} + T( \|u_1\|_{L^{\infty}_{T}L^2_x} +\|u_2\|_{L^{\infty}_{T}L^{2}_x}) \|U \|_{L^{\infty}_{T}L^2_x} + T\V_0 \|V\|_{L^{\infty}_{T}L^2_x} \notag \\
& \quad + (1+\V_{\max(0,s-1)})^3(\|V\|_{L^{\infty}_{T}H^{\max(0,s-1)}_x} + \|G\|_{L^{\infty}_{T,x}}) \|w_1\|_{Z^{s}_{T}} \label{LinftyHsdiff}  \\
&\quad  + (1+\V_{\max(0,s-1)})^4\|W\|_{Z^{s}_{T}} +\V_0 (1+\V_0)^6 \V_{s} \|G\|_{L^{\infty}_{T,x}} + (1+\V_0)^{5} \V_{s} \|V\|_{L^{\infty}_{T}L^2_x}.
\notag 
\end{align}
\end{lemma}

\begin{proof}
First, \eqref{Xsbdiff} follows exactly as for \eqref{vXsb} upon taking the difference of two solutions $V$. 
Taking the differences in \eqref{recovery1} we see two types of terms: those with a difference of exponential factors (i.e. one occurrence of $G$) and those with one occurrence of $W$ or $V$. The latter terms are controlled exactly as in Lemma~\ref{LEM:CCMXsb}. For the former terms, we adapt the previous computations. For the term $\P_{N}\PbHIp[ G w_1]$, we split according to \eqref{L4tx3}. Unless we need $G$ to contribute smoothing, we always take it out of the norm gaining the factor $\|G\|_{L^{\infty}_{T,x}}$. For the contribution from the term $\P_{\ges N}G \cdot \P_{\gg N}w_1$, we use the following simple estimate to gain additional derivatives: 
\begin{align}
\|\P_{N}G\|_{L^{\infty}_{T,x}} \les N^{-\frac 12}\{ \|V\|_{L^{\infty}_{T}L^2_x} + \|v_1\|_{L^{\infty}_{T}L^2_x} \|G\|_{L^{\infty}_{T,x}}\big\}.
\label{Ginfty}
\end{align}
For the term $\wt{\P}_{N}G\cdot \P_{\ll N}w_1$, we \eqref{PNe2} instead of \eqref{PNe1}. Altogether we find 
\begin{align*}
\| J^{s}_x \PbHIp[ e^{i\rho F_1} w_1-e^{i\rho F_2} w_2]\|_{\wt{L^{4}_{T,x}}}
& \les \|w_1\|_{Y^{s}_{T}} \big\{ \|V\|_{L^{\infty}_{T}L^2_x} + \V_{0}\|G\|_{L^{\infty}_{T,x}} \big\} +(1+\V_0)\|W\|_{Y^{s}_{T}} \\
& + (1+\V_{0})^3 \|V\|_{L^{\infty}_{T}H^s_x} + (1+\V_{0})^5 \V_{s} \|V\|_{L^{\infty}_{T}L^2_x} \\
& + \V_{0}(1+\V_0)^6 \V_{s} \|G\|_{L^{\infty}_{T,x}} + \|W\|_{Y^{0}_{T}}\{ \V_0^4 +(1+\V_0)^2 \V_s\}.
\end{align*} 
Next, for the remaining terms in the difference of \eqref{recovery1}, we argue as in \eqref{L4tx2} with \eqref{PNe2} instead of \eqref{PNe1}. These considerations yield \eqref{L4diff}.

To obtain \eqref{LinftyHsdiff}, we argue similarly, using in addition \eqref{eFg2} and \eqref{PNe2}. 
\end{proof}

We now provide estimates on the remainder terms $R(v,w)$, which demonstrate an additional $\frac 12$-degree of smoothing as compared to $w$. In particular, obtaining exactly this amount of smoothing in the Strichartz spaces $L^{4}_{T,x}$ is necessary for our argument to handle the second term on the right hand side of \eqref{ueq}. See Lemma~\ref{LEM:SchrodingerII} for an application of the following estimates. In order to obtain exactly $\frac 12$-degree of smoothing and cover high values of $s$, we heavily rely on careful usage of Lemma~\ref{LEM:PNe}. 

\begin{lemma}\label{LEM:s12R}
Let $0\leq s<\frac 32$, $0<T\leq 1$. Let $w\in Z_T^{0}$ such that $\P_{+}w=w$, $F$ be a real-valued function such that $\dx F= v$ for $J^{s}_x v \in L^{\infty}_{T}L^2_x \cap \wt{L^{4}_{T,x}}$, and set $\V:=\|v\|_{L^{\infty}_{T}L^2_x}$. 
 Then,
\begin{align}
\begin{split}
\| J^{s+\frac 12}_x R(v,w) \|_{\wt{L^4_{T,x}}} & \les \|w\|_{Y^{0}_{T}} + T^{\frac{1}{12}} \|w\|_{Y^{0}_{T}}(1+\V)^4 \|v\|_{L^{\infty}_{T}H^{s}_x} \\
& + \V^{2} \|v\|_{L^{4}_{T,x}} + T^{\frac 12}(1+\V)\V \|v\|_{L^{\infty}_{T}H^s_x} + \V \|J^{s}_x v\|_{\wt{ L^4_{T,x}}}.
\end{split}  \label{Rbd}
\end{align}
Moreover, for $j=1,2,$ let  $ w_j\in Y^{0}_{T}$ such that $\P_{+}w_j=w_j$, $F_j$ be two real-valued functions such that $\dx F_j= v_j$ for $J^{s}_x v_j \in L^{\infty}_{T}L^2_x\cap \wt{L^{4}_{T,x}}$. Let $G:=e^{i\rho F_1}-e^{i\rho F_2}$, $V:=v_1-v_2$, and $W:=w_1-w_2$.  Then,
\begin{align}
\|& J^{s+\frac 12}_x[  R(v_1,w_1)-R(v_2,w_2)] \|_{\wt{L^4_{T,x}}}  \notag \\
& \les (1+\|w_{1}\|_{Y^{0}_{T}})\|v_2\|_{L^{\infty}_{T}L^{2}_x}\Big\{ \|J^{s}_x V\|_{\wt{L^{4}_{T,x}}} + \V_{12}\|J^{s}_x v_2\|_{\wt{L^{4}_{T,x}}} \|G\|_{L^{\infty}_{T,x}}  \notag  \\
&  \quad + \V_{12}^{3}(1+\V_{12})^3 \|v_2\|_{L^{4}_{T,x}} \big(  \|V\|_{L^{\infty}_{T}L^2_x} + \V_{12}\|G\|_{L^{\infty}_{T,x}} )    \big)  \notag \\
& \quad +(1+\V_{12})^{3} \big(  \|V\|_{L^{\infty}_{T}H^{s}_x} + \V_{12}\|v_2\|_{L^{\infty}_{T}H^{s}_x} \|G\|_{L^{\infty}_{T,x}} + ( \|v_1\|_{L^{\infty}_{T}H^{s}_x} +\|v_2\|_{L^{\infty}_{T}H^{s}_x}     ) \|G\|_{L^{\infty}_{T,x}} \big)  \Big\} \notag \\
&\quad +  \|w_1\|_{Y^{0}_{T}} \|G\|_{L^{\infty}_{T,x}} + \|W\|_{Y^{0}_{T}} + T^{\frac{1}{12}} \|W\|_{Y^{0}_{T}}( \V_{12}^{4}+(1+\V_{12})(\|v_1\|_{L^{\infty}_{T}H^{s}_x}+\|v_2\|_{L^{\infty}_{T}H^{s}_x}))  \notag\\
& \quad+ \V_{12}^{2} \|V\|_{L^{4}_{T,x}}  + \V_{12} \|v_2\|_{L^{4}_{T,x}}\|V\|_{L^{\infty}_{T}L^2_x}
\label{Rbddiff}
\end{align}
\end{lemma}

\begin{proof}
From \eqref{R}, we have
\begin{align*}
\|J^{s+\frac 12}_x & R(v,w) \|_{\wt{L^4_{T,x}}} \les    \|J^{s+\frac 12}_x \PbLO[ e^{i\rho F} w]  \|_{\wt{L^4_{T,x}}}+\|J^{s+\frac 12}_x \P_{-,\text{HI}}[ e^{i\rho F} w] \|_{\wt{L^4_{T,x}}} + \|J^{s+\frac 12}_x [e^{i\rho F} \PbLO w] \|_{\wt{L^4_{T,x}}} \\
& \hphantom{X} + \|J^{s+\frac 12}_x\PbHIp[ \Pbhip e^{i\rho F} \cdot  \P_{-,\text{hi}}(e^{-i\rho F}v)]  \|_{\wt{L^4_{T,x}}} +\|J^{s+\frac 12}_x\PbHIp [ e^{i\rho F} \Pblo( e^{-i\rho F} v)] \|_{\wt{L^4_{T,x}}} \\
&=: A_1 +A_2+A_3+A_4+A_5.
\end{align*}
First, we trivially have
\begin{align}
A_1\les \|w\|_{L^{4}_{T,x}} \les \|w\|_{Y^{0}_{T}}. \label{RA1}
\end{align}
For $A_{2}$, as $w$ is supported on non-negative frequencies we see that the exponential factor in $\P_{-,\text{HI}}[ e^{i\rho F} w]$ carries the largest frequency and so for fixed $t\in [0,T]$, by \eqref{PNe1} and \eqref{L4}, we obtain
\begin{align*}
\| \P_{N} [ \P_{-,\text{HI}}[ e^{i\rho F} w]\|_{L^{4}_x} & \les \sum_{N_1 \ges N\vee N_2} \| \P_{N_1}e^{i\rho F} \cdot \P_{N_2}w\|_{L^{4}_x}  \\
&\les  \sum_{N_1 \ges N\vee N_2} \| \P_{N_1}e^{i\rho F}\|_{L^{12}_x} \|\P_{N_2}w\|_{L^{6}_x} \\
& \les  \sum_{N_1 \ges N\vee N_2} \big(N_1^{-2-\frac{1}{12}} \|v\|_{L^2}^{4} +N_{1}^{-\frac 12-\frac{1}{12}} (1+\|v\|_{L^2})^{4}\|\P_{\ges N_1}v\|_{L^2} \big) \|\P_{N_2}w\|_{L^{6}_x}. 
\end{align*}
Therefore, 
\begin{align}
A_2^{2} &\les  T^{\frac 16} \|w\|_{L^{6}_{T,x}}^2   \sum_{N} N^{2s+1} \big(  N^{-4-\frac 16+} \|v\|_{L^{\infty}_{T}L^2}^{8}+N^{-1-\frac 16 -2s+} (1+\|v\|_{L^{\infty}_{T}L^2})^8 \|v\|_{L^{\infty}_{T}H^s}^2  \big) \notag \\
& \les T^{\frac 16} \|w\|^2_{Z^{0}_{T}} ( \|v\|_{L^{\infty}_{T}L^2}^8 + (1+\|v\|_{L^{\infty}_{T} L^2})^8 \|v\|_{L^{\infty}_{T} H^s}^2), \label{RA2}
\end{align}
provided that $0\leq s<\frac 32+\frac{1}{12}=\frac{19}{12}$.
For $A_3$, we decompose $e^{i\rho F} = \P_{\leq 2^{6}}  e^{i\rho F}  +\P_{> 2^{6}} e^{i\rho F} $, where $\P_{\leq 2^{6}}  e^{i\rho F} : = e^{i\rho F}-\P_{>2^{6}}e^{i\rho F}$. It follows that 
$\P_{\leq 2^{6}}  e^{i\rho F} \cdot \PbLO w = \P_{\leq 2^{10}}[ \P_{\leq 2^{6}}  e^{i\rho F} \cdot \PbLO w],$
so that this contribution to $A_3$ is simply bounded by $\|w\|_{L^{4}_{T,x}}$. For the other contribution, we notice that 
$
\P_{N}[ \P_{> 2^{6}} e^{i\rho F}  \cdot \PbLO w] = \P_{N}[ \wt{\P}_{N}\P_{> 2^{6}} e^{i\rho F}  \cdot \PbLO w],
$
so then by Bernstein's inequality and \eqref{PNe1}, this contribution to $A_3$ is bounded by
\begin{align*}
\bigg(& \sum_{N\gg 1} N^{2s+1} \| \wt{\P}_{N} e^{i\rho F}\|_{L^{\infty}_{T}L^4_x}^2 \bigg)^{\frac 12} \|\PbLO w\|_{L^{4}_{T}L^{\infty}_{x}} \\
& \les  \|w\|_{L^{4}_{T,x}}  \bigg( \sum_{N\gg 1} N^{2s+1}N^{-4-\frac 12} \|v\|_{L^{\infty}_{T}L^2_x}^{8} + N^{2s+1} N^{-\frac 32} (1+\|v\|_{L^{\infty}_{T}L^2_x})^2 \|\P_{\ges N}v\|_{L^{\infty}_{T}L^2_x}^{2} \bigg)^{\frac 12} \\
& \les  \|w\|_{Y^{0}_{T}} ( \|v\|_{L^{\infty}_{T}L^2}^4 + (1+\|v\|_{L^{\infty}_{T} L_x^2}) \|v\|_{L^{\infty}_{T} H^s})
\end{align*}
Therefore, we have shown that 
\begin{align}
A_3 \les \|w\|_{Y^{0}_{T}} ( 1+\|v\|_{L^{\infty}_{T}L^2}^4 + (1+\|v\|_{L^{\infty}_{T} L_x^2}) \|v\|_{L^{\infty}_{T} H^s}). \label{RA3}
\end{align}
We move onto $A_4$, which is more delicate. In view of the signs of the frequencies, Bernstein's inequality, and \eqref{PNe1}, we have 
\begin{align}
A_4^{2} & \les \sum_{N\geq 1} N^{2s+1} \|\Pbhip \wt{\P}_{N}(e^{i\rho F}) \cdot \P_{\les N}\P_{-,\text{hi}}[ e^{-i\rho F}v]\|_{L^{4}_{T,x}}^{2} \notag \\
& \les \sum_{N\geq 1} N^{2s+1} \Big\|   \| \wt{\P}_{N}(e^{i\rho F})\|_{L^{4}_{x}} \|\P_{\les N}[ e^{-i\rho F}v]\|_{L^{\infty}_{x}} \Big\|_{L^{4}_{T}}^2\notag  \\
&\les \sum_{N\geq 1} N^{2s+1} N^{-\frac{9}{2}} \|v\|_{L^{\infty}_{T}L^2_x}^{8} \|\P_{\les N}[ e^{-i\rho F}v]\|_{L^{4}_{T}L^{\infty}_{x}}^2  \notag   \\
&\qquad+ \|v\|^2_{L^{\infty}_{T}L^2_x} \sum_{N\geq 1} N^{2s+2}\Big\{  N^{-\frac{5}{2}}(1+\|v\|_{L^{\infty}_{T}L^2_x})^2 \|v\|_{L^{\infty}_{T}L^2_x}^2 \|\P_{\ges N}v\|_{L^{4}_{T}L^2_x}^2 + N^{-2} \|\wt{\P}_{N}v\|_{L^{4}_{T,x}}^{2} \Big\}\notag  \\
&\les \|v\|_{L^{\infty}_{T}L^2_x}^{8} \sum_{N\geq 1} N^{2s-\frac{7}{2}} N^{\frac 12}  \|\P_{\les N}[ e^{-i\rho F}v]\|_{L^{4}_{T,x}}^2 \notag   \\
&\qquad +  (1+\|v\|_{L^{\infty}_{T}L^2_x})^2 \|v\|^4_{L^{\infty}_{T}L^2_x} \sum_{N\geq 1} N^{2s-\frac 12}\|\P_{\ges N}v\|_{L^{4}_{T}L^2_x}^2  + \|v\|^2_{L^{\infty}_{T}L^2_x}  \|J^{s}_x v\|_{\wt{L^4_{T,x}}}^{2}
\notag \\
&\les \|v\|_{L^{\infty}_{T}L^2_x}^8 \|v\|_{L^{4}_{T,x}}^{2}+T^{\frac 12}(1+\|v\|_{L^{\infty}_{T}L^2_x})^2 \|v\|^4_{L^{\infty}_{T}L^2_x} \|v\|_{L^{\infty}_{T}H^s_x}^{2} + \|v\|^2_{L^{\infty}_{T}L^2_x}  \|J^{s}_x v\|_{\wt{L^4_{T,x}}}^{2}. \label{RA4}
\end{align}
In a similar way, we have
\begin{align}
&A_5^2  \les \sum_{N\geq 1} N^{2s+1} \|\wt{\P}_{N}e^{i\rho F}\|_{L^{4}_{T,x}}^{2} \|\Pblo (e^{-i\rho F}v)\|_{L^{\infty}_{T,x}}^{2} \notag \\
& \les   \|v\|_{L^{\infty}_{T}L^{2}_x}^2 \Big( T^{\frac 12} \|v\|_{L^{\infty}_{T}L^{2}_x}^{8} + T^{\frac 12} (1+\|v\|_{L^{\infty}_{T}L^2_x})^{2}  \|v\|_{L^{\infty}_{T}H^{\max(0,s-\frac 34+)}_x}^2 + \|J^{s}_x v\|_{\wt{L^4_{T,x}}}^{2}\Big). 
\label{RA5}
\end{align}
Combining \eqref{RA1}, \eqref{RA2}, \eqref{RA3}, \eqref{RA4}, and \eqref{RA5} then establishes \eqref{Rbd}.

For \eqref{Rbddiff}, we take the difference of $R(v_1,w_1)$ and $R(v_2,w_2)$ using \eqref{R}. Relative to the estimate for \eqref{Rbd}, it remains only to consider the contributions for terms with a factor of $G$ appearing. By the same argument as for \eqref{RA1}, we see that 
\begin{align*}
\|J^{s+\frac 12}_x \PbLO[Gw_1] \|_{\wt{L^{4}_{T,x}}} 
& \les \|G\|_{L^{\infty}_{T,x}} \|w_1\|_{Y^{0}_{T}}.
\end{align*}
Next, similar to \eqref{RA2} and \eqref{RA3} but using \eqref{PNe2}, we have
\begin{align*}
\|J^{s+\frac 12}_x & \P_{-,\text{HI}}[G w_1] \|_{\wt{L^{4}_{T,x}}}  +\|J^{s+\frac 12}_x [ G \PbLO w_1]\|_{\wt{L^{4}_{T,x}}}      \\
\les& \|w_1\|_{Y^{0}_{T}}\Big(  \|V\|_{L^{\infty}_{T}H^{s}_x} + \V_{12}\|v_2\|_{L^{\infty}_{T}H^{s}_x}\|G\|_{L^{\infty}_{T,x}} \\
& + (1+\V_{12})^{6} \{ \|V\|_{L^{\infty}_{T}H^{s}_x}+ (1+\|v_1\|_{L^{\infty}_{T}H^{s}_x}+\|v_2\|_{L^{\infty}_{T}H^{s}_x}) ( \|V\|_{L^{\infty}_{T}L^2_x} + \V_{12}\|G\|_{L^{\infty}_{T,x}})\Big) \\
& + \|w_1\|_{Y^{0}_{T}} \|G\|_{L^{\infty}_{T,x}}.
\end{align*}
Similarly to \eqref{RA4}, using \eqref{PNe2} instead of \eqref{PNe1}, we also have
\begin{align*}
&\|J^{s+\frac 12}_x\PbHIp[ \Pbhip G \cdot  \P_{-,\text{hi}}(e^{-i\rho F_2}v_2)]  \|_{\wt{L^4_{T,x}}} \\
& \les \|v_2\|_{L^{\infty}_{T}L^{2}_x}\Big\{ \|J^{s}_x V\|_{\wt{L^{4}_{T,x}}} + \V_{12}\|J^{s}_x v_2\|_{\wt{L^{4}_{T,x}}} \|G\|_{L^{\infty}_{T,x}}  \\
&  \quad + \V_{12}^{3}(1+\V_{12})^3 \|v_2\|_{L^{4}_{T,x}} \big(  \|V\|_{L^{\infty}_{T}L^2_x} + \V_{12}\|G\|_{L^{\infty}_{T,x}} )    \big) \\
& \quad +(1+\V_{12})^{3} \big(  \|V\|_{L^{\infty}_{T}H^{s}_x} + \V_{12}\|v_2\|_{L^{\infty}_{T}H^{s}_x} \|G\|_{L^{\infty}_{T,x}} + ( \|v_1\|_{L^{\infty}_{T}H^{s}_x} +\|v_2\|_{L^{\infty}_{T}H^{s}_x}     ) \|G\|_{L^{\infty}_{T,x}} \big)  \Big\}.
\end{align*}
A similar bound holds for the term $\PbHIp [ G \Pblo( e^{-i\rho F_2} v_2)]$. This proves \eqref{Rbddiff}.
\end{proof}

The estimates in Lemma~\ref{LEM:s12R} simplify when there are no derivatives acting on $R(v,w)$. 

\begin{lemma} \label{LEM:RLq}
 Let $w$ be such that $\P_{+}w=w$ and $F$ be a real-valued function such that $\dx F= v$. Then,
\begin{align}
\| R(v,w) \|_{L^{4}_{x}} & \les (1+\|v\|_{L^2})^{4}( \|w\|_{L^2}+\|v\|_{L^2}).
\label{RLq}
\end{align}
Moreover, for $j=1,2,$ let  $ w_j\in L^2_x$ such that $\P_{+}w_j=w_j$, $F_j$ be two real-valued functions such that $\dx F_j= v_j$. Let $G:=e^{i\rho F_1}-e^{i\rho F_2}$, $V:=v_1-v_2$, and $W:=w_1-w_2$. Then,
\begin{align}
\begin{split}
\| R(v_1,w_1) - R(v_2,w_2)\|_{L^4}  \les &  (1+\V_{12} + \|w_1\|_{L^2}+\|w_2\|_{L^2})^{8} (\|W\|_{L^2}+\|V\|_{L^2}) \\
& +\|G\|_{L^{\infty}} (1+\V_{12})^{8} (\|w_1\|_{L^2}+\|w_2\|_{L^2}+\V_{12}).
\end{split} \label{RLqdiff}
\end{align}
\end{lemma}
\begin{proof}
We begin with \eqref{RLq}.
From Bernstein's inequality, we have 
\begin{align*}
\| \PbLO[ e^{i\rho F}w]\|_{L^{4}}+\|e^{i\rho F}\PbLO w\|_{L^{4}} + \| \PbHIp[ e^{i\rho F}\Pblo (e^{-i\rho F}v)]\|_{L^{4}} \les \|w\|_{L^2}+\|v\|_{L^2}.
\end{align*}
Thus, from \eqref{R}, it remains to estimate the contributions from the second and fourth terms on the right-hand side of \eqref{R}.
We only give details for the second term as the same arguments work to estimate the fourth. 
By the signs of the frequencies, $e^{i\rho F}$ carries the largest frequency. Thus, by Bernstein's inequality and \eqref{PNe1}, we have
\begin{align*}
\| \P_{-,\text{HI}}[e^{i\rho F}w]\|_{L^{4}} & \les \sum_{N_1 \ges N_2} \|\P_{N_1}e^{i\rho F}\|_{L^{4}} \|\P_{N_2}w\|_{L^{\infty}} \\
& \les \|w\|_{L^2} \sum_{N_1 \ges N_2} N_{2}^{\frac 12} \big( N_1^{-\frac{9}{4}} \V^{4}+N_1^{-\frac{5}{4}}(1+\|v\|_{L^2_x})\|v\|_{L^2_x}^2 + N_{1}^{-1}\|\P_{N_1}v\|_{L^4}\big) \\
& \les  \|w\|_{L^2}  \bigg( \|v\|_{L^2_x}^4 + (1+\|v\|_{L^2_x})\|v\|_{L^2_x}^2 +\sum_{N_1\ges N_2} N_{1}^{-\frac 14} \|v\|_{L^2_x}\bigg) \\
& \les \|w\|_{L^{2}}\|v\|_{L^2_x} (1+\|v\|_{L^2_x})^{3}.
\end{align*}
This proves \eqref{RLq}.
The estimate \eqref{RLqdiff} follows in the same way but using \eqref{PNe2}.
\end{proof}

\section{Multilinear estimates}\label{SEC:tri}

In this section, we detail the key multilinear estimates needed to analyse the left-hand sides of \eqref{weqn} (the ``BO-part") and \eqref{ueq} (the ``Schr\"{o}dinger-part").

\subsection{The BO-part}
The first of these concerns one of the main nonlinear contributions in \eqref{weqn}. The estimate is an adaptation of the key multilinear estimate in \cite[Proposition 3.5]{MP} to the framework of the atomic functions spaces. As such, the estimate in \cite[Proposition 3.5]{MP} does not imply the following.

\begin{lemma}\label{LEM:BO1}
Let $s\geq 0$ and $0<T\leq 1$. Then, 
\begin{align}
\big\| \dx \Pbhip[ (\dx^{-1}w) \P_{-}\dx v] \big\|_{N^{s}_{T}} & \les  \|w\|_{Z^{s}_{T}} \big( \|v\|_{L^2_{T,x}} + \|v\|_{\wt{L^{4}_{T,x}}}+\|v\|_{X_{BO;T}^{-1,1}}\big).  \label{BO1} 
\end{align}
\end{lemma}
\begin{proof}
In order to ease the notation, we define the bilinear form:
\begin{align*}
B(w,v) : =  \dx \Pbhip[ (\dx^{-1}w) \P_{-}\dx v].
\end{align*}
Note that by the signs of the frequencies, we have $|\xi_1|= |\xi|+|\xi_2|$. Thus, it suffices to prove \eqref{BO1} in the case $s=0$.
By Lemma~\ref{LEM:linXsb} and \eqref{Indbdd}, we reduce to showing that
\begin{align}
\bigg| \int_{\R} \int_{\R} B( \ind_{[0,T)}w,v) \cj{ \ind_{[0,T)} h} dx dt \bigg| \les \|h\|_{Y^{0}_{T}} \|w\|_{Z^{0}_{T}} \big( \|v\|_{L^{2}_{T,x}} + \|v\|_{\wt{L^4_{T,x}}} + \|v\|_{X^{-1,1}_{BO;T}}\big), \label{BO10}
\end{align}
for any $h\in Y^{0}_{T}$. 
We then take extensions $\wt{h}, \wt{w}, \wt{v}$ of $h$, $w$, and $v$, respectively, and define $\wt{h}_{T}=\ind_{[0,T)}\wt{h}$ and $\wt{w}_{T}=\ind_{[0,T)}w$, and further reduce, using \eqref{Indbdd}, to showing
\begin{align}
\bigg| \int_{\R} \int_{\R} B(\wt{w}_{T},\wt{v}) \cj{ \wt{h}_{T}} dx dt \bigg| \les \|\wt{h}_{T}\|_{Y^{0}} \|\wt{w}_{T}\|_{Z^{0}} \big( \|\wt{v}\|_{L^{2}_{t,x}} + \|\wt{v}\|_{\wt{L^4_{t,x}}} + \|\wt{v}\|_{X^{-1,1}_{BO}}\big). \label{BO11}
\end{align}
In the following, we drop the tildes, writing $h_{T},w_{T},v$ in place of $\wt{h}_{T}, \wt{w}_{T}, \wt{v}$, respectively. 

We will use the following bounds:
\begin{align}
\| \ind_{[0,T)}f\|_{L^{4}_{t,x}} + \| \mathcal{F}_{t,x}^{-1}\{  |\mathcal{F}_{t,x}\{ \ind_{[0,T)}f\}|    \}\|_{L^{4}_{t,x}} \les (1+T^{\frac 12})\|f\|_{V^{2}_{S}L^2_x}. \label{L4V2}
\end{align}
We only show the second bound, with the one for $\ind_{[0,T)}f$ following from the same arguments. Using \eqref{L4Strich}, \eqref{U2Xsb}, \eqref{Indbdd}, we have
\begin{align*}
 \| \mathcal{F}_{t,x}^{-1}\{  |\mathcal{F}_{t,x}\{ \ind_{[0,T)}f\}| \}\|_{L^{4}_{t,x}} 
 &\les  \| \mathcal{F}_{t,x}^{-1}\{  |\mathcal{F}_{t,x}\{ \ind_{[0,T)}f\}|    \}\|_{X^{0,\frac 38}} \\
 & \sim \|  \ind_{[0,T)}f  \|_{X^{0,\frac 38}}\\
 & \les \| \ind_{[0,T)}f \|_{L^{2}_{t,x}} + \| \ind_{[0,T)}f\|_{\dot{X}^{0,\frac 12,\infty}} \\
 & \les T^{\frac 12} \|f\|_{L^{\infty}_{t}L^2_x}  +\| \ind_{[0,T)}f\|_{V^{2}_{S}L^2_x} \\
 & \les T^{\frac 12} \|S(-t)f\|_{L^{\infty}_{t}L^2_x}+2\| f\|_{V^{2}_{S}L^2_x}  \\
 &\les (1+T^{\frac 12}) \| f\|_{V^{2}_{S}L^2_x}.
\end{align*}
This establishes \eqref{L4V2}.

By dyadic decomposition, we write 
\begin{align}
\text{LHS} \eqref{BO11} = \bigg| \sum_{N_1,N_2,N_3} \int_{\R} \int_{\R} B(\P_{N_1}w_{T},\P_{N_2}v) \cj{ \P_{N_3}h_{T}} \, dx dt \bigg| = \bigg| \sum_{N_1,N_2,N_3} I_{N_1,N_2,N_3}\bigg|. \label{BO12}
\end{align}
Due to the signs of the frequencies in $B(\P_{N_1}w, \P_{N_2}v)$, we have $N_1 \ges\max(N_2,N_3)$ and from the outer $\Pbhi$ in the definition of $B$, we have $N_3\ges 1$.
We first dispense with the trivial case when $N_2 \les 1$. Here, we must have that $N_3= 2^{j}N_1$ for $|j|\leq 2$. Then, by H\"{o}lder's inequality,
\begin{align*}
|I_{N_1,N_2,N_3}|& \les \|\P_{N_3}h_T\|_{L^{4}_{t,x}} \| B(\P_{N_1}w_T, \P_{N_2}v)\|_{L^{\frac 43}_{t,x}}  \les  \|\P_{N_3}h_T\|_{L^{4}_{t,x}} \|\P_{N_1}w_T\|_{L^{4}_{t,x}} \|\P_{N_2} v\|_{L^{2}_{t,x}},
\end{align*}
so then by Cauchy-Schwarz and \eqref{L4}, we have
\begin{align*}
\sum_{ N_1, N_2,N_3 } \ind_{\{N_{2}\les 1\}} |I_{N_1,N_2,N_3}| & \les \|v\|_{L^{2}_{t,x}} \sum_{|j|\leq 2}\sum_{N_1} \|\P_{2^{j}N_1}h_T\|_{L^{4}_{t,x}} \|\P_{N_1}w_T\|_{L^{4}_{t,x}} \\
&\les \|v\|_{L^{2}_{t,x}} \| h_T\|_{\wt{L^{4}_{t,x}}} \|w_T\|_{\wt{L^{4}_{t,x}}} \\
& \les \|v\|_{L^{2}_{t,x}}  \| h_T\|_{Y^0} \|w_T\|_{Z^{0}}.
\end{align*}

We now assume that $N_2 \gg 1$. We write 
\begin{align}
I_{N_1,N_2,N_3} & = \sum_{\l=1}^{3} I^{(\l)}_{N_1,N_2,N_3}, \label{Ils}\\
 I^{(1)}_{N_1,N_2,N_3} &:  = \int_{\R^2} B(\P_{N_1}w_T , \P_{N_2}v) \cj{ \Q_{\ges N_{2}N_3}\P_{N_3}h_T} \,dxdt  \label{I1},\\
  I^{(2)}_{N_1,N_2,N_3} &:  = \int_{\R^2} B(\Q_{\ges N_{2}N_3}\P_{N_1}w_T , \P_{N_2}v) \cj{ \Q_{\ll N_{2}N_3}\P_{N_3}h_T}\, dxdt  \label{I2},\\
  I^{(3)}_{N_1,N_2,N_3} &:  = \int_{\R^2} B(\Q_{\ll N_{2}N_3}\P_{N_1}w_T , \Q_{\ges N_{2}N_3}\P_{N_2}v) \cj{ \Q_{\ll N_{2}N_3}\P_{N_3}h_T}  \, dxdt  \label{I3}.
\end{align}
The decomposition \eqref{Ils} is exhaustive in view of the resonance relation
\begin{align}
(\tau_1+\xi_1^2)+(\tau_2-\xi_2^2) - (\tau+\xi^2) = \xi_1^2 - \xi_2^2 -\xi^2 = -2\xi \xi_2, \label{resrel}
\end{align}
where the second equality holds when $\xi=\xi_1+\xi_2$.
We consider each of $I^{(\l)}_{N_1,N_2,N_3}$ separately. 

\noi
\underline{\textbf{Case $\l=1$:}}
Consider first the case when $N_1\sim N_3$. Using that $\sum_{N_1\sim N_3} \P_{N_1} = \wt{\P}_{N_3}$
along with H\"{o}lder's inequality, \eqref{Qgain} and \eqref{L4},
\begin{align*}
\bigg|\sum_{ \substack{N_1,N_2,N_3 \\ N_1\sim N_3}} I_{N_1,N_2,N_3}^{(1)}  \bigg|
&=  \bigg|\sum_{ N_2 \les N_3} \int_{\R^2} B(\wt{\P}_{N_3} w_T , \P_{N_2}v) \cj{ \Q_{\ges N_2 N_3}\P_{N_3}h_T} dx dt \bigg| \\
& =  \bigg|\sum_{N_3} \sum_{j=0}^{\log_{2}(N_3)}  \int_{\R^2} B(\wt{\P}_{N_3} w_T , \P_{2^{-j}N_3}v) \cj{ \Q_{\ges 2^{-j} N_3^2}\P_{N_3}h_T} dx dt \bigg| \\
& \les  \sum_{N_3} \sum_{j=0}^{\log_{2}(N_3)}  N_3 \|\dx^{-1}\wt{\P}_{N_3}w_T\|_{L^{4}_{t,x}} \| \dx \P_{-}\P_{2^{-j} N_3}v\|_{L^{4}_{t,x}}\| \Q_{\ges 2^{-j}N_3^2} \P_{N_3}h_T\|_{L^{2}_{t,x}} \\ 
& \les  \sum_{N_3} \sum_{j=0}^{\log_{2}(N_3)}  2^{-\frac{j}{2}}  \|\P_{N_3}h_T\|_{V^{2}_{S}L^2_x} \|\wt{\P}_{N_3}w_T\|_{L^{4}_{t,x}} \| \P_{2^{-j}N_3}v\|_{L^{4}_{t,x}} \\
& \les \|h\|_{Y^0} \|w_T\|_{\wt{L^4_{t,x}}} \|v\|_{L^{4}_{t,x}}.
\end{align*}
If instead $N_1\sim N_2$, then we write $N_3= 2^{-j}N_2$ and we similarly obtain
\begin{align*}
\bigg|\sum_{ \substack{N_1,N_2,N_3 \\ N_1\sim N_2}} I_{N_1,N_2,N_3}^{(1)}  \bigg|  
&  = \bigg|   \sum_{ N_2} \sum_{j=0}^{\log_{2}(N_2)}   \int_{\R^2}  \P_{2^{-j}N_2} B(\wt{\P}_{N_2} w_T , \P_{N_2}v) \cj{ \Q_{\ges 2^{-j}N_2^2 }\P_{2^{-j}N_2}h_T} dx dt    \bigg| \\
& \les \sum_{ N_2} \sum_{j=0}^{\log_{2}(N_2)} 2^{\frac{j}{2}} N_{2}^{-1} \|\P_{2^{-j}N_2} B(\wt{\P}_{N_2} w_{T} , \P_{N_2}v) \|_{L^{2}_{t,x}} \|\P_{2^{-j}N_2}h_{T}\|_{V^{2}_{S}L^2_x} \\
& \les \|h_T\|_{Y^0} \|w_T\|_{\wt{L^4_{t,x}}} \|v\|_{L^{4}_{t,x}}.
\end{align*}

\noi
\underline{\textbf{Case $\l=2$:}}
Again, we consider first the case when $N_1\sim N_3$. By summing over $N_1\sim N_3$, integrating by parts, using $N_2= 2^{-j}N_3$, \eqref{Qgain}, \eqref{Qbddness}, and \eqref{L4}, and \eqref{U2Xsb}, we have
\begin{align*}
&\bigg|\sum_{ \substack{N_1,N_2,N_3 \\ N_1\sim N_3}} I_{N_1,N_2,N_3}^{(2)}  \bigg|  
  = \bigg|   \sum_{ N_3} \sum_{j=0}^{\log_{2}(N_3)}   \int_{\R^2}  B(\Q_{\ges 2^{-j} N^2_3}\wt{\P}_{N_3} w_T , \P_{2^{-j}N_3}v) \cj{ \Q_{\ll 2^{-j} N^2_3 }\P_{N_3}h_T} \, dx dt    \bigg| \\
& =  \bigg|   \sum_{ N_3} \sum_{j=0}^{\log_{2}(N_3)}   \int_{\R^2}  \dx^{-1} \Q_{\ges 2^{-j} N^2_3}\wt{\P}_{N_3} w_T  \cdot \dx\P_{-} \P_{2^{-j}N_3}v \cdot \cj{\P_{-,\text{hi}}\dx \Q_{\ll 2^{-j}N^2_3 }\P_{N_3}h_T} \, dx dt    \bigg| \\
& \les \sum_{ N_3} \sum_{j=0}^{\log_{2}(N_3)}  \| \Q_{\ges 2^{-j} N^2_3}\wt{\P}_{N_3} w_T\|_{L^{2}_{t,x}} \| \dx\P_{-} \P_{2^{-j}N_3}v \|_{L^{4}_{t,x}} \| \P_{-,\text{hi}}\dx \Q_{\ll 2^{-j} N^2_3 }\P_{N_3}h_T\|_{L^{4}_{t,x}} \\
&  \les  \sum_{ N_3}  \sum_{j=0}^{\log_{2}(N_3)}2^{-\frac{j}{2}} \| \wt{\P}_{N_3}w_T\|_{U^{2}_{S}L^2_x} \|\P_{2^{-j}N_3}v\|_{L^{4}_{t,x}} \| \P_{N_3}h_T\|_{V^{2}_{S}L^2_x}  \\
& \les \|h_T\|_{Y^0} \|w_T\|_{Z^0} \|v\|_{L^{4}_{t,x}}.
\end{align*}

\noi
If instead $N_1 \sim N_2$, then we argue similarly:
\begin{align*}
\bigg|\sum_{ \substack{N_1,N_2,N_3 \\ N_1\sim N_2}} I_{N_1,N_2,N_3}^{(2)}  \bigg|  
&  = \bigg|   \sum_{ N_2} \sum_{j=0}^{\log_{2}(N_2)}   \int_{\R^2}  B(\Q_{\ges 2^{-j} N^2_2}\wt{\P}_{N_2} w_T , \P_{N_2}v) \cj{ \Q_{\ll 2^{-j} N^2_2 }\P_{2^{-j}N_2}h_T} \, dx dt    \bigg|\\
&  \les  \sum_{ N_3}  \sum_{j=0}^{\log_{2}(N_3)}2^{-\frac{j}{2}} \| \wt{\P}_{N_2}w_T\|_{U^{2}_{S}L^2_x} \|\P_{N_2}v\|_{L^{4}_{t,x}} \| \P_{2^{-j}N_2}h_T\|_{V^{2}_{S}L^2_x}  \\
& \les \|h_T\|_{Y^0} \|w_T\|_{Z^0} \|v\|_{\wt{L^{4}_{t,x}}}.
\end{align*}

\noi
\underline{\textbf{Case $\l=3$:}}
In this case, to avoid logarithmic losses, we need to sum over all of the dyadics beforehand and argue on the space-time Fourier side as in \cite[Proposition 3.5]{MP}. 
By using Parseval's formula and by summation over the dyadics such that $N_2\gg 1$, 
\begin{align*}
&\bigg| \sum_{N_1,N_2, N_3}I^{(3)}_{N_1,N_2,N_3} \bigg|  \\
& \les \int_{\mathcal{R}} |\ft h_T(\tau,\xi)|  \frac{ |\xi|}{|\xi-\xi_2|} |\ft w_{T}(\tau-\tau_2,\xi-\xi_2)| \frac{|\xi_2|^2}{\jb{\tau_2-\xi_2^2}} \frac{\jb{\tau_2-\xi_2^2}}{|\xi_2|} |\ft v(\tau_2,\xi_2)| d\tau_2 d\tau d\xi_2 d\xi,
\end{align*}
where 
\begin{align*}
\mathcal{R} : = \{ (\xi,\xi_2,\tau,\tau_2):\,  |\xi-\xi_2|\geq |\xi_2| \gg 1, |\tau_2 -\xi_2^2| \gg \max(|\tau -\xi^2|, |\tau-\tau_2 -(\xi-\xi_2)^2|)\}.
\end{align*}
In view of the restriction to $\mathcal{R}$ and \eqref{resrel}, we have $ |\tau_2 -\xi_2^2|  \ges |\xi||\xi_2|$, and using \eqref{L4V2}, we have
\begin{align*}
&\les \int_{\mathcal{R}}  |\ft h_T(\tau,\xi)|  |\ft w_{T}(\xi-\xi_2,\tau-\tau_2)| \frac{\jb{\tau_2-\xi_2^2}}{|\xi_2|} |\ft v(\tau_2,\xi_2)| d\tau_2 d\tau d\xi_2 d\xi \\
& \les \| \mathcal{F}_{t,x}^{-1}\{|\ft h_T|\}\|_{L^{4}_{t,x}}\| \mathcal{F}_{t,x}^{-1}\{ |\ft w_{T}|\}\|_{L^{4}_{t,x}} \| v\|_{X^{-1,1}_{BO}}\\
& \les (1+T^{\frac 12})^2 \| \wt{h}\|_{V^{2}_{S}L^2_x} \|\ind_{[0,T)} \wt{w}\|_{V^{2}_{S}L^2_x} \|\wt{v}\|_{X^{-1,1}_{BO}}.
\end{align*}
This completes the proof of \eqref{BO10}.
\end{proof}

The second estimate for the BO-part concerns the nonlinear term $\dx[ e^{-i\rho F} |u|^2]$. As discussed earlier, even though $u$ is $\frac 12$-degree smoother than $v$, we still need to recover an additional $\frac 12$-degree of smoothing to control the derivative. The key for this is the bilinear Strichartz estimate \eqref{bilin2}. 
The remaining term $\dx \Pbhip[ \Pblo e^{-i\rho F} \P_{-}\dx v]$ is of no concern due to the favourable frequency projections.

\begin{lemma} \label{LEM:BO2}
Let $s\geq 0$ and $\dl>0$.
For $j=1,2$, let $F_{j}$ denote two real-valued functions such that $\dx F_j =v_j$. Let $\V_{12}:= \|v_1\|_{L^{\infty}_{T}L^2_x}+ \|v_2\|_{L^{\infty}_{T}L^2_x}$ and $G: = e^{-i\rho F_1}-e^{-i\rho F_2}$. 
 Then, 
\begin{align}
\big\| \dx \PbHIp[& (e^{-i\rho F_1}-e^{-i\rho F_2}) u_1 \cj{u_2}]\big\|_{N^{s}_{T}} \notag  \\
&\les T^{\frac 12} \big( \|V\|_{L^{2}_{T,x}}+ \|v_2\|_{L^{2}_{T,x}} \| G \|_{L^{\infty}_{T,x}} ) \notag  \\
&\quad \times   \Big( \|u_1\|_{X^{s+\frac12, \frac 12+\dl}_{S;T}} \| u_2\|_{X^{\frac 12, \frac 12+}_{S;T}} + \| u_1\|_{X^{\frac12, \frac 12+}_{S;T}} \| u_2\|_{X^{s+\frac 12, \frac 12+\dl}_{S;T}}\Big) \notag \\
& \quad + \|u_1\|_{X^{\frac 12,\frac 12}_{S;T}}\|u_2\|_{X^{\frac 12,\frac 12}_{S;T}} \Big( \|J^{s}_{x}V\|_{\wt{L^{4}_{T,x}}} + \V_{12}\|J^{s}_x v_2\|_{\wt{L^{4}_{T,x}}} \|G\|_{L^{\infty}_{T,x}}  \notag \\  
&  \qquad +  (1+\V_{12})^2 ( \|v_1\|_{L^{\infty}_{T}H^{s}_x}+\|v_2\|_{L^{\infty}_{T}H^s_x})\|V\|_{L^{\infty}_{T}L^{2}_x} \notag \\
& \qquad +(1+\V_{12})^3 (\|V\|_{L^{\infty}_{T}H^{s}_x} + \|v_2\|_{L^{\infty}_{T}H^{s}_x}\|G\|_{L^{\infty}_{T,x}}) \Big), \label{BO2}\\
\big\| \dx  \Pbhip[ &\Pblo( e^{-i\rho F_1}- e^{-i\rho F_2})\mathbf{P}_{-}\dx v] \big\|_{N^{s}_{T}}  \les T^{\frac 12} \|e^{-i\rho F_1}-e^{-i\rho F_2}\|_{L^{\infty}_{T,x}}\|v\|_{L^{\infty}_{T}L^{2}_x} .\label{BO3}
\end{align}
\end{lemma}
\begin{proof}
We begin with showing \eqref{BO2}. 
By Lemma~\ref{LEM:linXsb} we reduce to controlling both of:
\begin{align}
\int_{0}^{T} \int_{\R} \cj{ \jb{\dx}^{s} h} \, \dx \PbHIp \big[  \Pblo G(F_1,F_2)  u_1 \cj{u_2}  \big] dx dt,  \label{BO21}\\
\int_{0}^{T} \int_{\R} \cj{ \jb{\dx}^s h} \, \dx \PbHIp \big[  \Pbhi G(F_1,F_2)  u_1 \cj{u_2}  \big] dx dt. \label{BO22}
\end{align}
Consider first the contribution from \eqref{BO21}. By integration by parts, we have 
\begin{align}
\eqref{BO21} & = -\int_{0}^{T} \int_{\R}  \Pblo G(F_1,F_2) \cdot \Pblo\big[  \cj{ \dx \jb{\dx}^{s} \P_{-,\text{HI}} h} \cdot   u_1 \cj{u_2}  \big] dx dt. \label{BO212}
\end{align}
We take arbitrary extensions of $h$ and $u_1, u_2$, which we denote as $\wt{h}$ and $\wt{u_1}, \wt{u_2}$, respectively.
 We then associate a copy of the indicator $\ind_{[0,T)}$ with both of $\wt{h}$ and $\Pblo G(F_1,F_2)$, and write $\wt{h}_{T}:= \ind_{[0,T)}\wt{h}$.
By dyadic decomposition, we then consider 
\begin{align*}
\sum_{N_1, N_2,N_3}\int_{\R^2}  \ind_{[0,T)} \Pblo G(F_1,F_2) \cdot \Pblo\big[  \cj{ \dx \jb{\dx}^{s} \P_{-,\text{HI}} \P_{N_3} \wt{h}_T} \cdot   \P_{N_1}\wt{u_1} \cdot \cj{ \P_{N_2} \wt{u_2}}  \big] dx dt.
\end{align*}
In view of the outer projection $\Pblo$ we may place a projection $\wt{\P}_{N_3}$ outside of the product  $\P_{N_1}\wt{u_1}\cdot \cj{ \P_{N_2} \wt{u_2}} $. Then, by Cauchy-Schwarz and \eqref{bilin2}, we have
\begin{align*}
&\sum_{N_1, N_2,N_3}\int_{\R^2}  \ind_{[0,T)} \Pblo G(F_1,F_2) \cdot \Pblo\big[  \cj{ \dx \jb{\dx}^{s} \P_{-,\text{HI}} \P_{N_3} \wt{h}_T} \cdot  \wt{\P}_{N_3}[  \P_{N_1}\wt{u_1} \cdot  \cj{ \P_{N_2} \wt{u_2}}]  \big] dx dt \\
& \les \| \Pblo G(F_1,F_2)\|_{L^{\infty}_{T,x}} \sum_{N_1,N_2,N_3} N_3^{1+s} \| \P_{N_3} \wt{h}\|_{L^{2}_{t,x}}
\| \wt{\P}_{N_3}[  \P_{N_1}\wt{u_1}  \cdot \cj{ \P_{N_2} \wt{u_2}}]\|_{L^{2}_{t,x}}  \\
&\les \|  G(F_1,F_2)\|_{L^{\infty}_{T,x}} \sum_{N_1,N_2,N_3} N_{3}^{\frac 12+s}  \| \P_{N_3} \wt{h}_T\|_{L^{2}_{t,x}} \| \P_{N_1}\wt{u_1} \|_{X^{0,\frac 12+}_{S}} \|\P_{N_2}\wt{u_2}\|_{X_{S}^{0,\frac 12+}} \\
&\les  \|  G(F_1,F_2)\|_{L^{\infty}_{T,x}}  \| \wt{h}_T\|_{L^{2}_{t,x}} \Big( \| \wt{u_1}\|_{X^{s+\frac12, \frac 12+}_{S}} \| \wt{u_2}\|_{X^{\frac 12, \frac 12+}_{S}} + \| \wt{u_1}\|_{X_{S}^{\frac12, \frac 12+}} \| \wt{u_2}\|_{X^{s+\frac 12, \frac 12+}_{S}}\Big).
\end{align*}
By Cauchy-Schwarz and \eqref{ZYembed}, we have 
\begin{align*}
\|  \wt{h}_T\|_{L^{2}_{t,x}} \les T^{\frac 12} \|\wt{h}\|_{L^{\infty}_{t}L^2_x} \les T^{\frac 12} \|\wt{h}\|_{V^{2}_{S}L^2_x}.
\end{align*}
By taking an infimum over extensions, we then obtain
\begin{align*}
|\eqref{BO212}| & \les  T^{\frac 12} \|  G(F_1,F_2)\|_{L^{\infty}_{T,x}}  \| h\|_{Y^0_T}   \Big( \|u_1\|_{X^{s+\frac12, \frac 12+}_{S;T}} \| u_2\|_{X^{\frac 12, \frac 12+}_{S;T}} + \| u_1\|_{X^{\frac12, \frac 12+}_{S;T}} \| u_2\|_{X^{s+\frac 12, \frac 12+}_{S;T}}\Big), 
\end{align*}
which completes the estimate for \eqref{BO21}.

Now we consider the term \eqref{BO22}. By integration by parts and using  a dyadic decomposition, we reduce to estimating
\begin{align*}
\sum_{N_1,N_2,N_3,N_4} \int_{0}^{T} \int_{\R} \P_{N_3} G(F_1,F_2) \cdot   \cj{ \dx \jb{\dx}^{s} \P_{-,\text{HI}} \P_{N_4} h} \cdot   \P_{N_1} u_1 \cdot \cj{\P_{N_2} u_2}  dx dt
\end{align*}
where $N_3 \ges 1$. If $N_4 \gg N_3$, then we may proceed as we did for \eqref{BO21} using that in the case $N_4\sim \max(N_1,N_2)$ and $N_3 \gg \min(N_1,N_2)$, we sum over $N_3$ using 
\begin{align}
&\| \P_{N_3}G(F_1,F_2)\|_{L^{\infty}_{T,x}}  \notag\\
&\les N_3^{-1}  \big\{\| \P_{N_3}( (v_1-v_2)e^{-i\rho F_1})\|_{L^{\infty}_{T,x}} +\| \P_{N_3}(v_2 (e^{-i\rho F_1}-e^{-i\rho F_2} )\|_{L^{\infty}_{T,x}}  \big\} \notag \\
& \les N_{3}^{-\frac 12+}\{ \|v_1-v_2\|_{L^{2}_{T,x}} + \|v_2\|_{L^{2}_{T,x}} \|G(F_1,F_2)\|_{L^{\infty}_{T,x}}\}.
\end{align}
If instead $N_{4} \les N_3$, then we place all functions into $L^{4}_{T,x}$ and use \eqref{PNe2}. To perform the dyadic summations for the terms $\|\P_{N_3}V\|_{L^{4}_{T,x}}$ and $\|\P_{N_3}v_j\|_{L^{4}_{T,x}}$, we either gain from $N_3 \les N_1 \vee N_2$ and the $\frac12$-degree of regularity of the $u_1$ or $u_2$ terms, or, we have $N_4 \sim N_3$ and we use
\begin{align*}
\sum_{N_4 \sim N_3} N_{3}^{s}\| \P_{N_3}V\|_{L^{4}_{T,x}} \|\P_{N_4}h\|_{L^{4}_{T,x}} \les \| J^{s}_x V\|_{\wt{L^{4}_{T,x}}}  \|h\|_{\wt{L^{4}_{T,x}}}  \les  \| J^{s}_x V\|_{\wt{L^{4}_{T,x}}} \|h\|_{Y^{0}_{T}}.
\end{align*}
This then establishes \eqref{BO2}.

We now consider the easier estimate \eqref{BO3}. By arguments akin to \eqref{argument}, we have 
\begin{align*}
\Pbhip[ \Pblo G  \cdot\P_{-}\dx v]  =\PbLO \Pbhip[  \Pblo G  \cdot \P_{-}\dx \PbLO v],
\end{align*}
so that 
\begin{align*}
\big\| \dx  \Pbhip[ \Pblo G \cdot \mathbf{P}_{-}\dx v] \big\|_{N^{s}_{T}} & \les \big\| \Pblo G \cdot \mathbf{P}_{-}\dx \PbLO v \big\|_{N^{0}_{T}} \les \|G\|_{L^{\infty}_{T,x}} \|v\|_{L^{2}_{T,x}}.
\end{align*}
This proves \eqref{BO3}, and thus completes the proof of the lemma.
\end{proof}

\subsection{The Schr\"{o}dinger-part}
We move onto proving multilinear estimates for the Schr\"{o}dinger-part. Again, the main point in the following estimates is to recover $\frac 12$-degree of smoothing through using the bilinear Strichartz estimate \eqref{bilin1}, which is possible because of our bootstrap assumption that $w\in Z^{s}_{T}$.

\begin{lemma}\label{LEM:S1}
Let $0\leq s<\frac 32$ and $\dl>0$ sufficiently small. There exists $\ta>0$ such that  
\begin{align}
\| &e^{-i\rho F} u \,\cj{\PbHI w} \|_{X^{s+\frac 12,-\frac 12+2\dl}_{S;T}}  \notag\\
& \les T^{\ta} \Big\{  \|u\|_{L^{\infty}_{T}L^{2}_x}\|w\|_{L^{\infty}_{T}L^{2}_x} + \|u\|_{X^{s+\frac 12, \frac 12+\dl}_{S;T}} \|w\|_{Z^{0}_{T}}  +\|u\|_{X^{0,\frac 12+\dl}_{S;T}} \|w\|_{Z^{s}_{T}} \notag\\
&  \qquad + \|u\|_{X^{0,\frac 12}_{S;T}} \|w\|_{Z^{0}_{T}}(1+\|v\|_{L^{\infty}_{T}L^2_x})[ \|v\|_{L^{\infty}_{T}L^2_x}^{3} + (1+\|v\|_{L^{\infty}_{T}L^2_x})\|v\|_{L^{\infty}_{T}H^{s}_x}]   \Big\}     \label{uw}
\end{align}
Moreover, for $j=1,2$, let $F_{j}$ denote two real-valued functions such that $\dx F_j =v_j$. Let $\V_{12}:= \|v_1\|_{L^{\infty}_{T}L^2_x}+ \|v_2\|_{L^{\infty}_{T}L^2_x}$ and $G: = e^{i\rho F_1}-e^{i\rho F_2}$. Then,
\begin{align}
\| &G(F_1,F_2) u\, \cj{\PbHI w} \|_{X^{s+\frac 12,-\frac 12+2\dl}_{S;T}}  \notag \\
&  \les T^{\ta} \Big(  (1+\V_{12})^{3}\|V\|_{L^{\infty}_{T}H^s_x}+\V_{12}(1+\V_{12})^{2}\{ \|v_1\|_{L^{\infty}_{T}H^{s}_x} +  \|v_2\|_{L^{\infty}_{T}H^{s}_x}\} \|G\|_{L^{\infty}_{T,x}} \Big) \|u\|_{X^{0,\frac 12}_{T}} \|w\|_{Z^{0}_{T}} \notag \\ 
& \quad +T^{\ta}\Big(  \|u\|_{L^{\infty}_{T}L^{2}_x}\|w\|_{L^{\infty}_{T}L^{2}_x} + \|u\|_{X^{s+\frac 12, \frac 12+\dl}_{S;T}} \|w\|_{Z^{0}_{T}}  +\|u\|_{X^{0,\frac 12+\dl}_{S;T}} \|w\|_{Z^{s}_{T}}   \Big) \|G\|_{L^{\infty}_{T,x}}.    \label{uw2}
\end{align}
\end{lemma}
\begin{proof}

We begin with the more involved estimate \eqref{uw}. As the extra projection $\PbHI$ on $w$ plays no role, we ignore it.
 We split into two parts:
\begin{align}
\| e^{-i\rho F} u \cj{w} \|_{X^{s+\frac 12,-\frac 12+2\dl}_{S;T}}  \leq \| \Pblo(e^{-i\rho F}) u \cj{w} \|_{X^{s+\frac 12,-\frac 12+2\dl}_{S;T}} +\| \Pbhi( e^{-i\rho F} )u \cj{w} \|_{X^{s+\frac 12,-\frac 12+2\dl}_{S;T}} .
\label{uw1}
\end{align}
We first consider the term with $\Pblo$. By dyadic decomposition and duality, we estimate
\begin{align*}
  \sum_{N,N_1,N_2} \int_{0}^{T} \int_{\R}\cj{J^{s+\frac 12}_{x} \P_{N}h} \cdot \Pblo(e^{-i\rho F}) \P_{N_1}u \cdot \cj{\P_{N_2}w} \,dxdt =: \sum_{N,N_1,N_2} \I_{N, N_1,N_2},
\end{align*}
where $h\in X^{0,\frac 12-2\dl}_{S;T}$ with norm at most one.
If $N\les 1$, then we simply have
\begin{align*}
\sum_{N, N_1, N_2}  |I_{N, N_1,N_2}| \les (\sup_{N\les 1} \| \P_{N} h\|_{L^{2}_{T}L^{\infty}_{x}}) \|u\|_{L^{\infty}_{T}L^{2}_{x}} \|w\|_{L^{2}_{T,x}} \les T^{\frac 12} \|u\|_{L^{\infty}_{T}L^{2}_{x}} \|w\|_{L^{\infty}_{T}L^{2}_{x}}.
\end{align*}
We now assume that $N\gg 1$.
First, we argue that $N\les N_1 \vee N_2$. Indeed, if $N\gg N_1 \vee N_2$, then for each fixed $t\in [0,T]$, by integration by parts we have
\begin{align*}
\jb{ J^{s+\frac 12}_{x} \P_{N}h,  \Pblo(e^{-i\rho F}) \P_{N_1}u \cdot \cj{\P_{N_2}w}}_{L^2} & = \jb{\cj{\Pblo(e^{-i\rho F})} ,  J^{s+\frac 12}_{x} \P_{N}h \cdot \P_{N_1}u \cdot \cj{\P_{N_2}w}}_{L^2} \\
&  = \jb{\cj{\Pblo(e^{-i\rho F})} ,  \wt{\P}_{N}[ J^{s+\frac 12}_{x} \P_{N}h \cdot \P_{N_1}u \cdot\cj{\P_{N_2}w}]}_{L^2} \\
& = -\jb{\cj{\dx \Pblo(e^{-i\rho F})} ,  \dx^{-1}\wt{\P}_{N}[ J^{s+\frac 12}_{x} \P_{N}h \cdot \P_{N_1}u \cdot \cj{\P_{N_2}w}]}_{L^2} \\
& = 0,
\end{align*}
since in the penultimate line, we can now interpret $\Pblo$ as a projection to frequencies $\{|\xi|\les 1\}$, which is orthogonal to $\wt{\P}_{N}$ since $N\gg 1$. We split into three cases. 

\smallskip
\noi
\underline{$\bullet$ \textbf{Case 1:} $N_1\sim N_2$.}
By H\"{o}lder, Bernstein, Cauchy-Schwarz inequality, \eqref{L4} and \eqref{L4Strich}, we have 
\begin{align*}
\bigg|\sum_{N,N_1,N_2} \I_{N,N_1,N_2} \bigg|&\les   \sum_{N_1 \sim N_2 \ges N} N^{s+\frac 12} \|\P_{N}h\|_{L^{4}_{T,x}} N_1^{-s-\frac 12}  \|J^{s+\frac 12}\P_{N_1}u\|_{L^{2}_{T,x}} \|\P_{N_2}w\|_{L^4_{T,x}}\\
&\les \|h\|_{\wt{L^4_{T,x}} } \sum_{N_1 \sim N_2}\|J^{s+\frac 12}\P_{N_1}u\|_{L^{2}_{T,x}} \|\P_{N_2}w\|_{L^4_{T,x}} \\
&\les \|h\|_{\wt{L^4_{T,x}} }  \sum_{|\l|\leq 5} \sum_{N_2} \|J^{s+\frac 12}\P_{2^{\l}N_2}u\|_{L^{2}_{T,x}}  \|\P_{N_2}w\|_{L^4_{T,x}}  \\
& \les \|h\|_{X^{0,\frac 12-2\dl}_{S;T}} \|u\|_{X^{s+\frac 12, 0}_{S;T}} \| w\|_{Z^{0}_{T}} \\
&\les T^{\frac 12}  \|h\|_{X^{0,\frac 12-2\dl}_{S;T}} \|u\|_{X^{s+\frac 12, \frac 12+\dl}_{S;T}} \| w\|_{Z^{0}_{T}}.
\end{align*}

\smallskip
\noi
\underline{$\bullet$ \textbf{Case 2:} $N_1\gg  N_2.$}
In this case, we have
\begin{align*}
\bigg|\sum_{N,N_1,N_2} \I_{N,N_1,N_2}\bigg| &\les \sum_{N\sim N_1} N^{s+\frac 12} \|\P_{N}h\|_{L^4_{T,x}} \|\P_{N_1}u\|_{L^{2}_{T,x}} \|\P_{\ll N_1}w\|_{L^{4}_{T,x}} \\
&\les  \|h\|_{X^{0,\frac 12-2\dl}_{S;T}} \|u\|_{X^{s+\frac 12, 0}_{S;T}} \| w\|_{Z^{0}_{T}} \\
&\les T^{\frac 12} \|h\|_{X^{0,\frac 12-2\dl}_{S;T}} \|u\|_{X^{s+\frac 12, \frac 12+\dl}_{S;T}} \| w\|_{Z^{0}_{T}}
\end{align*}

\smallskip
\noi
\underline{$\bullet$ \textbf{Case 3:} $N_1\ll  N_2$.}
In this case, we need to gain additional smoothing through the bilinear Strichartz estimate with distant frequency supports \eqref{bilin1}. We have 
\begin{align*}
\bigg|\sum_{N,N_1,N_2} \I_{N,N_1,N_2}\bigg| &\les\sum_{N\sim N_2} N^{s+\frac 12} \| \P_{N}h\|_{L^{2}_{T,x}} \| \P_{\ll N_2}u \cdot \P_{N_2}w\|_{L^2_{T,x}}  \\
&\les \sum_{N\sim N_2} N^{s+\frac 12} N_{2}^{-\frac 12} \|\P_{N}h\|_{L^{2}_{T,x}} \|u\|_{X^{0,\frac 12+}_{S;T}} \|\P_{N_2} w\|_{U^{2}_{S; T}L^2_x}  \\
& \sim \|u\|_{X^{0,\frac 12+}_{S;T}} \sum_{|k|\leq 5}  \sum_{N_2} N_{2}^{s} \|\P_{N_2}w\|_{U^{2}_{S; T}L^2_x} \|\P_{2^{k}N_2}h\|_{L^{2}_{T,x}} \\
&\les \| h\|_{X^{0,0}_{S;T}} \|u\|_{X^{0,\frac 12+}_{S;T}} \| w\|_{Z^{s}_{T}} \\
&\les T^{\frac 12-2\dl}\| h\|_{X^{0,\frac 12-2\dl}_{S;T}} \|u\|_{X^{0,\frac 12+}_{S;T}} \| w\|_{Z^{s}_{T}}
\end{align*}
This completes the proof for the contribution to \eqref{uw} coming from $\Pblo e^{i\rho F}$. As for the contribution from $\Pblo (e^{i\rho F_1} -e^{i\rho F_2})$ in \eqref{uw2}, we apply the same above argument noting that in all of the cases we placed $\Pblo e^{i\rho F}$ into $L^{\infty}_{T,x}$.

For the contribution in \eqref{uw1} from $\Pbhi e^{-i\rho F}$, we use further dyadic decomposition and reduce to estimating:
\begin{align*}
  \sum_{N,N_1,N_2, N_3} \int_{0}^{T} \int_{\R} \cj{J^{s+\frac 12}_{x} \P_{N}h} \cdot \P_{N_3}(e^{-i\rho F}) \P_{N_1}u \cdot \cj{\P_{N_2}w}\, dxdt =: \sum_{N,N_1,N_2,N_3} \II_{N, N_1,N_2,N_3},
\end{align*}
In this case, we have $N\les N_1 \vee N_2 \vee N_3$. 
If $N\les N_1\vee N_2$, then we may proceed as in the contribution from $\Pblo e^{i\rho F}$ as we have control on the derivative $J_x^{s+\frac 12}$ in the norm. It remains to consider the case where $N\gg N_1\vee N_2$ and thus $N\sim N_3$. Then
\begin{align*}
\sum_{N,N_1,N_2,N_3} \II_{N, N_1,N_2,N_3} & \les \int_{0}^{T} \sum_{N\sim N_3} N^{s+\frac 12}\|\P_{N}h\|_{L^{2}_{x}} \| \P_{N_3} e^{-i\rho F}\|_{L^{\infty}_{x}}  \|\P_{\ll N}u\|_{L^{4}_{x}} \|\P_{\ll N}w\|_{L^{4}_{x}}  dt\\
& \les \|u\|_{L^{4}_{T,x}} \|w\|_{L^{4}_{T,x}} \bigg\|\sum_{N_3} \|\P_{N}h\|_{L^{2}_{x}} \|J^s_x \dx \P_{N_3} e^{-i\rho F}\|_{L^{2}_{x}} \bigg\|_{L^2_{T}} \\
& \les T^{\frac 18-2\dl}\|u\|_{X^{0,\frac 12}_{T}} \|w\|_{Z^{0}_{T}} \|h\|_{L^{2}_{T,x}} \|J^s_x \dx e^{-i\rho F}\|_{L^{\infty}_{T}L^{2}_{x}}.
\end{align*}
It remains to estimate the last term here, for which we use \eqref{PNe1} which implies
\begin{align}
\| J^{s}_x \dx e^{-i\rho F}\|_{L^{2}_x} & \les \|v\|_{L^{2}_x} + \bigg( \sum_{N\gg 1} N^{2s+2}\|\P_{N}e^{-i\rho F}\|_{L^2_x}^{2}\bigg)^{\frac 12} \notag \\
& \les  (1+\|v\|_{L^{\infty}_{T}L^{2}_x})\big[ \|v\|_{L^{\infty}_{T}L^{2}_x}^{3}  + (1+\|v\|_{L^{\infty}_{T}L^{2}_x}) \|v\|_{L^{\infty}_{T}H^{s}_x} \big].
 \label{uw12}
\end{align}
Similarly, using \eqref{PNe2}, we get
\begin{align*}
\| J^s_x \dx G\|_{L^{\infty}_{T}L^{2}_{x}} & \les (1+\V_{12})^{3}\|V\|_{L^{\infty}_{T}H^s_x}+\V_{12}(1+\V_{12})^{2}\{ \|v_1\|_{L^{\infty}_{T}H^{s}_x} +  \|v_2\|_{L^{\infty}_{T}H^{s}_x}\} \|G\|_{L^{\infty}_{T,x}}
\end{align*}
 which we use in place of \eqref{uw12} to estimate the contribution with $G$ in \eqref{uw2}. This completes the proof.
\end{proof}

Next we control the term $u \P_{+,\text{HI}}v$. Whilst $v$ belongs to the Fourier restriction norm spaces at a cost of regularity in space, the extra projection $\P_{+,\text{HI}}$ here turns out to be advantageous as it ensures that this product is highly non-resonant; see \eqref{IIphase}. Although it would be possible to replace $\P_{+,\text{HI}}v$ by \eqref{recovery2}, we choose to point out that this particular interaction can be controlled without using \eqref{recovery2}.

\begin{lemma} \label{LEM:SchrodingerII}
Let $s\geq 0$ and $0<\dl\leq \frac{1}{4}$. Then 
\begin{align}
\| u \P_{+,\textup{HI}} v\|_{X^{s+\frac 12,-\frac 12+2\dl}_{S;T}} &\les T^{\ta} \big\{ \|v\|_{L^{\infty}_{T}L^2_x} \|u\|_{X^{s+\frac 12,\frac 12+\dl}_{S;T}} \notag \\
&\hphantom{XXXXXX}+ ( \|v\|_{X^{s-\frac 38,\frac 38}_{\textup{BO};T}}+ \|v\|_{X^{s-1,1}_{\textup{BO};T}}) \| u\|_{X^{0,\frac 12+\dl}_{S;T}}  \big\} 
\label{uPHIv},\\
\| u \PbLO v\|_{X^{s+\frac 12,-\frac 12+2\dl}_{S;T}} &\les T^{\ta} \|u\|_{X^{s+\frac 12, \frac 12+\dl}_{S;T}} \|v\|_{L^{\infty}_{T}L^2_x}.  \label{uPLOv}
\end{align}
\end{lemma} 
\begin{proof}
We begin with \eqref{uPHIv}. By duality, we reduce to estimating
\begin{align}
\I :=\int_{0}^{T} \int_{\R}  \cj{J^{s+\frac 12}_{x} h} \cdot  u \cdot \P_{+,\text{HI}}v \,dxdt  \label{uPHIvdual}
\end{align}
where $h\in X^{0,\frac{1}{2}-2\dl}_{S;T}$ with $\|h\|_{X^{0,\frac{1}{2}-2\dl}_{T}} \leq 1$.
We decompose $h=\Pblo h + \Pbhi h$ and consider the contributions to $\I$, which we write as $\I_{\text{lo}}$ and $\I_{\text{hi}}$, respectively. By H\"{o}lder and Bernstein inequality and \eqref{L4Strich}, we have
\begin{align}
 | \I_{\text{lo}}| \les \|J^{s+\frac 12}_{x}\Pblo h\|_{L^{4}_{T,x}} \| u\|_{L^{4}_{T,x}} \| v\|_{L^{2}_{T,x}} \les T^{\frac 12}  \|h\|_{X_{S}^{0,\frac 12-2\dl}} \|u\|_{L^{4}_{T,x}}\|v\|_{L^{\infty}_{T}L^2_x}.\label{uPHIvlo}
\end{align}
Now we consider the contribution from $\I_{\text{hi}}$, we decompose further as
\begin{align}
\begin{split}
\I_{\text{hi}} &= \sum_{N, N_1} \int_{0}^{T} \int_{\R}  \cj{J^{s+\frac 12}_{x} \P_{N} h} \cdot  \P_{N_1}u \cdot \P_{+,\text{HI}}v \,dxdt  \\
& = \sum_{N \sim N_1} \cdot  + \sum_{N \gg N_1} \cdot  +\sum_{N\ll N_1} \cdot  =: \I_{\text{hi}}^{(1)} +\I_{\text{hi}}^{(2)} + \I_{\text{hi}}^{(3)}.
\end{split} \label{uPHIvhi}
\end{align}
For $\I_{\text{hi}}^{(1)}$, by H\"{o}lder and Bernstein inequalities, we have 
\begin{align}
\begin{split}
|\I_{\text{hi}}^{(1)}| &\les \sum_{|j|\leq 5} \sum_{N \geq 2} \|\P_{N}h\|_{L^4_{T,x}} \|J_{x}^{s+\frac 12}\P_{2^{j}N}u\|_{L^{4}_{T,x}} \| v\|_{L^{2}_{T,x}}  \\
& \les \|h\|_{\wt{L^{4}_{T,x}}} \| J^{s+\frac 12}_{x} u\|_{\wt{L^{4}_{T,x}}} \|v\|_{L^{2}_{T,x}} \\
& \les  T^{\frac 18}\|h\|_{X^{0,\frac 12-2\dl}_{S;T}} \| u\|_{X^{s+\frac 12,\frac 12+\dl}_{S;T}} \|v\|_{L^{\infty}_{T}L^{2}_{x}}.
\end{split} \label{I1hi}
\end{align}
For $\I_{\text{hi}}^{(3)}$ we may place a widened projector $\wt{\P_{N_1}}$ onto $v$ for free.
Then
\begin{align*}
|\I_{\text{hi}}^{(3)}| &\les \sum_{N \ll N_1} (N N_{1}^{-1})^{s+\frac 12} \|\P_{N}h\|_{L^{4}_{T,x}} \|J^{s+\frac 12}\P_{N_1}u\|_{L^{4}_{T,x}} \| \wt{\P_{N_1}} v\|_{L^{2}_{T,x}} \\ 
& \les \| h\|_{\wt{L^{4}_{T,x}}} \sum_{N_1} \|J^{s+\frac 12}\P_{N_1}u\|_{L^{4}_{T,x}} \| \wt{\P_{N_1}} v\|_{L^{2}_{T,x}} \\
& \les  \|J^{s+\frac 12}\P_{N_1}u\|_{\wt{L^{4}_{T,x}}} \|v\|_{L^{2}_{T,x}} \\
&\les T^{\frac 18} \| u\|_{X^{s+\frac 12,\frac 12+\dl}_{S;T}} \|v\|_{L^{\infty}_{T}L^2_x}.
\end{align*}
For $\I_{\text{hi}}^{(2)}$ we need to make use of the phase function. We associate the sharp cutoff $\ind_{[0,T]}$ with the function $u$ and consider extensions $\wt{h}, \wt{u}, \wt{v}$ of $h, \ind_{[0,T]}u, v$ on $[0,T]$, respectively. We place $\wt{\P}_{N}$ onto $\wt{v}$ for free. Using the $\P_{+}$ projection on $v$, the phase function satisfies
\begin{align}
|\xi^2 - \xi_1^2 + \xi_2 |\xi_2|| =|\xi^2 - \xi_1^2 + \xi_{2}^{2}| = 2 |\xi| |\xi_2| \sim N^{2}. \label{IIphase}
\end{align} 
Then, we have 
\begin{align*}
|\I_{\text{hi}}^{(2)}| &\les \sum_{N \gg N_1} \bigg(  \bigg| \int_{\R^2} \cj{ \P_{N}J^{s+\frac 12}_{x}\Q_{\ges N^2} \wt{h}} \cdot  \P_{N_1}\wt{u}\cdot \wt{\P}_{N}v \,dx dt\bigg| \\
&\hphantom{X}+\bigg| \int_{\R^2} \cj{\P_{N}J^{s+\frac 12}_{x}\Q_{\ll N^2} \wt{h}} \cdot  \P_{N_1} \Q_{\ges N^2}\wt{u} \cdot\wt{\P}_{N}v\, dx dt\bigg|  \\
& \hphantom{X}+ \bigg| \int_{\R^2} \cj{\P_{N}J^{s+\frac 12}_{x}\Q_{\ll N^2} \wt{h}} \cdot  \P_{N_1}\Q_{\ll N^2} \wt{u} \cdot\wt{\P}_{N}\Q_{\ges N^2} v \,dx dt\bigg|  \bigg) =: \sum_{N \gg N_1} (A_1 +A_2 +A_3).
\end{align*}
By H\"{o}lder and Bernstein, and \eqref{L4Strich}, 
\begin{align*}
A_1 &\les N^{s+\frac 12} \| \Q_{\ges N^2} \P_{N}\wt{h}\|_{L^{2}_{t,x}} \|\P_{N_1}\wt{u}\|_{L^{4}_{t,x}}
 \|\wt{\P}_{N}\wt{v}\|_{L^{4}_{t,x}} \\
 &\les N^{s+\frac 12}N^{-1+4\dl} N^{-s+\frac{3}{8}} \|\P_{N}\wt{h}\|_{X_{S}^{0,\frac 12-2\dl}} \| \wt{u}\|_{X_{S}^{0,\frac 12-2\dl}} \| \wt{v}\|_{X_{BO}^{s- \frac 38, \frac 38}} \\
 & \les N^{-\frac 18 +4\dl}\|\wt{h}\|_{X_{S}^{0,\frac 12-2\dl}} \| \wt{u}\|_{X_{S}^{0,\frac 12-2\dl}} \| \wt{v}\|_{X_{BO}^{s- \frac 38, \frac 38}},
\end{align*}
which is a negative power of $N$ provided that $\dl<\frac{1}{24}$.
In a similar way, we have
\begin{align*}
A_2 &\les N^{s+\frac 12} \| \Q_{\ll N^2} \P_{N}\wt{h}\|_{L^{4}_{t,x}} \|\Q_{\ges N^2}\P_{N_1}\wt{u}\|_{L^{2}_{t,x}}
 \|\wt{\P}_{N}\wt{v}\|_{L^{4}_{t,x}} \\
 & \les N^{-\frac 18 +4\dl}\|\wt{h}\|_{X_{S}^{0,\frac 12-2\dl}} \| \wt{u}\|_{X_{S}^{0,\frac 12-2\dl}} \| \wt{v}\|_{X_{BO}^{s- \frac 38, \frac 38}},
\end{align*}
Finally, we have 
\begin{align*}
A_{3} & \les N^{s+\frac 12} \| \Q_{\ll N^2} \P_{N}\wt{h}\|_{L^{4}_{t,x}} \|\Q_{\ll N^2}\P_{N_1}\wt{u}\|_{L^{4}_{t,x}} \| \Q_{\ges N^2} \wt{\P}_{N}\wt{v}\|_{L^{2}_{t,x}} \\
&\les N^{s+\frac 12} N^{-2}N^{-s+1}\|\wt{h}\|_{X_{S}^{0,\frac 12-2\dl}} \| \wt{u}\|_{X_{S}^{0,\frac 12-2\dl}}  \| \wt{v}\|_{X_{BO}^{s-1,1}} \\
&\les N^{-\frac 12} \|\wt{h}\|_{X_{S}^{0,\frac 12-2\dl}} \| \wt{u}\|_{X_{S}^{0,\frac 12-2\dl}}  \| \wt{v}\|_{X_{BO}^{s-1,1}} ,
\end{align*}
which is acceptable. In these bounds, we can then gain a factor of $T^{\ta}$ due to the slight slack in modulation for the $u$ terms. This completes the proof of \eqref{uPHIv}.

 For \eqref{uPLOv}, we follow the same argument as for \eqref{uPHIv}. That is, we use duality as in \eqref{uPHIvdual} and split the dual function according to $\P_{> 2^{10}}$ and $\P_{\leq 2^{10}} $. The low frequency portion is handled exactly as in \eqref{uPHIvlo}, while for the high-frequency portion, upon dyadic decompositions as in \eqref{uPHIvhi}, we only have a non-zero contribution from the analogue of $\I^{(1)}_{\text{hi}}$ and we can bound as in \eqref{I1hi}. This completes the bound for \eqref{uPLOv}.
\end{proof}

The following estimates control the remaining terms on the left-hand side of \eqref{ueq}. Although relatively harmless, we rely on the exact $\frac 12$-smoothing property of $R(v,w)$ (Lemma~\ref{LEM:s12R}).

\begin{lemma} \label{LEM:SchrodingerIII}
Let $s\geq 0$, $\dl>0$ sufficiently small, and $0<T\leq 1$. Let $u,u_1,u_2,u_3 \in X^{s+\frac 12,\frac 12+\dl}_{S;T}$, for $j=1,2,$ let $w_j,v_j \in L^{\infty}_{T}L^2_x \cap L^{4}_{T,x}$, and let $F_j$ be two real-valued functions such that $\dx F_j =v_j$. Define $W:=w_1-w_2$, $V:=v_1-v_2$, $\V(f_1,f_2):= \|f_1\|_{L^{\infty}_{T}L^2_x} + \|f_2\|_{L^{\infty}_{T}L^2_x}$, $R_{\textup{diff}}:=R(v_1,w_1)-R(v_2,w_2),$
 and $G:=e^{i\rho F_1}-e^{i\rho F_2}$. 
Then, there exists $\ta>0$ such that
\begin{align}
 &\| u_1 \cj{u_2} u_3\|_{X^{s+\frac 12, -\frac 12+2\dl}_{S;T}} \les T^{\ta} \max_{\s \in S_3} \| u_{\s(1)}\|_{X^{s+\frac 12 ,\frac 12+\dl}_{S;T}} \prod_{j=2}^{3} \| u_{\s(j)}\|_{X^{0,\frac 12+\dl}_{S;T}}, \label{cubicXsb} \\
 &\big\|   u  \cj{R(v_1,w_1)}\big\|_{X^{s+\frac 12,-\frac 12+2\dl}_{S;T}}   \les   T^{\ta} \Big(  (\|w_1\|_{L^{4}_{T,x}}+\|v_1\|_{L^{4}_{T,x}}) \|u\|_{X^{s+\frac 12,\frac 38}_{S;T}}+ \|u\|_{X^{0,\frac 38}_{S;T}} \|J^{s+\frac 12}_x R\|_{L^{4}_{T,x}} \Big),  \label{uRterm1} \\
   &\big\|  u\cj{R_{\textup{diff}}}\big\|_{X^{s+\frac 12,-\frac 12+2\dl}_{S;T}} \les   T^{\ta}   (1+\V(v_1,v_2)+\V(w_1,w_2))^{4} [ \|V\|_{L^{\infty}_{T}L^2_x} + \|W\|_{L^{\infty}_{T}L^2_x}+\|G\|_{L^{\infty}_{T,x}}]  \notag \\
  & \hphantom{XXXXXXXXXXXXXXX}\times  \| u\|_{X^{s+\frac 12,\frac 38}_{S;T}}+T^{\ta}\|u\|_{X^{0,\frac 38}_{S;T}} \|J^{s+\frac 12}_x R\|_{L^{4}_{T,x}}.  \label{uRterm2}
\end{align}
\end{lemma}

\begin{proof}
The simpler bound \eqref{cubicXsb} is standard and follows by a duality argument, H\"{o}lder's inequality and \eqref{L4Strich}. We move onto \eqref{uRterm1}. Regarding
 \eqref{uRterm1},  we have
\begin{align}
 \| u \cj{R} \|_{X^{s+\frac 12,-\frac 12+2\dl}_{S;T}}   \les T^{\frac 12 -2\dl} \big\{ \|  u \cj{R}\|_{L^{2}_{T,x}} + \|D^{s+\frac 12}_x \PbHI [   u \cj{R} ]\|_{L^{2}_{T,x}}  \big\}. \label{uRt1}
\end{align}
By Cauchy-Schwarz, \eqref{L4Strich}, and \eqref{R}, we have  
\begin{align*}
\|  u R\|_{L^{2}_{T,x}} \les  \|u\|_{L^{4}_{T,x}} \|R\|_{L^{4}_{T,x}} \les ( \|w_1\|_{L^{4}_{T,x}} +\|v_1\|_{L^{4}_{T,x}}) \|u\|_{X^{0,\frac 38}_{S;T}}.
\end{align*}
For the second term on the right-hand side of \eqref{uRt1}, Lemma~\ref{LEM:leib} and \eqref{L4Strich} give
\begin{align*}
\|D^{s+\frac 12}_x \PbHI [   u \cj{R} ]\|_{L^{2}_{T,x}} 
& \les\|J^{s+\frac 12}_x u\|_{L^{4}_{T,x}} \|R\|_{L^{4}_{T,x}} + \|u\|_{L^{4}_{T,x}} \|J^{s+\frac 12}_x R\|_{L^{4}_{T,x}} \\
& \les (\|w_1\|_{L^{4}_{T,x}}+\|v_1\|_{L^{4}_{T,x}}) \|u\|_{X^{s+\frac 12,\frac 38}_{S;T}}+ \|u\|_{X^{0,\frac 38}_{S;T}} \|J^{s+\frac 12}_x R\|_{L^{4}_{T,x}} 
\end{align*}
The estimate for \eqref{uRterm2} follows similarly, additionally using \eqref{RLqdiff}. We omit the details.
\end{proof}

\section{Well-posedness}
\label{SEC:LWP}

\subsection{Local well-posedness}

In this section, we prove Theorem~\ref{THM:LWP}. As the general argument here is quite standard, we will be brief. 
For more details, we refer to \cite{MP, CLOP} where similar arguments were made for the BO equation. 
We recall the following result from \cite{LMP}.

\begin{proposition}\label{PROP:LWPH2}
Let $s>\frac 54$. Then, for any $R>0$, there is $T=T(R)>0$ such that for every $(u_0,v_0)\in H^{s+\frac 12}(\R)\times H^{s}(\R)$, with $\|(u_0,v_0)\|_{H^{s+\frac 12}\times H^s}\leq R$, there exists a unique solution $(u,v)\in C([0,T];H^{s+\frac 12}(\R)\times C([0,T];H^{s}(\R)$ of \eqref{SBO} with $(u,v)\vert_{t=0}=(u_0,v_0)$. Moreover, the map $(u_0,v_0)\mapsto (u(t),v(t))$ is continuous from $H^{s+\frac 12}(\R)\times H^{s}(\R)$ to $C([0,T];H^{s+\frac 12})\times C([0,T];H^s)$. 
\end{proposition}

Our goal is then to extend the local well-posedness result in Proposition~\ref{PROP:LWPH2} to any $s\geq 0$. Crossing this gap makes some of our estimates sensitive to the working value of $s$; see for example \eqref{vL4Hs}. Thus, we need to proceed in a two-step process to cover the full range $s\geq 0$:
\begin{enumerate}[(a)]
\item we first extend the result of Proposition~\ref{PROP:LWPH2} to $H^{s}(\T)$ for any $s>\frac 34$
\item using part (a), we then cover the remaining range $0 \leq s \leq \frac 34$.
\end{enumerate}
As the argument is the same for both (a) and (b), and the estimates given in the previous sections hold for any $0\leq s<\frac 32$, we only provide details for (b), \textit{assuming} that (a) has already been proved. Thus, in the following, we fix $0\leq s\leq \frac 34$ and prove (b).

We claim that it is sufficient to consider the case of small initial data in $H^{\frac 12}\times L^2$. Indeed, if $(u,v)$ solves \eqref{SBO} on $[0,T]$, then the rescaled functions $(u_{\ld},v_{\ld})$ defined by
\begin{align}
u_{\ld}(t,x) : = \ld u( \ld^2 t, \ld x) \quad\text{and} \quad v_{\ld}(t,x) : = \ld v( \ld^2 t, \ld x) \label{scaling}
\end{align}
for $0<\ld\leq 1$, which exist on the longer time interval $[0,\ld^{-2}T]$, solve the $\ld$-SBO system:
\begin{equation}
\left\{
\begin{aligned}
  & i\dt u_{\ld}+\dx^2 u_{\ld} = \ld  u_{\ld}v_{\ld}+\be |u_{\ld}|^2 u_{\ld},\\
  & \dt v_{\ld}-\H \dx^2 v_{\ld}= \dx(|u_{\ld}|^2-\rho v_{\ld}^2),\\
   & (u_{\ld},v_{\ld})|_{t=0}=(u_{0,\ld},v_{0,\ld}),
\end{aligned}
 \right. \label{ldSBO}
\end{equation}
where $(u_{0,\ld},v_{0,\ld}): = (\ld u_0 (\ld \cdot) , \ld v_0 (\ld \cdot))$.
In particular, while \eqref{SBO} does not enjoy a perfect scaling property, the term $uv$ is ``sub-critical", with respect to a natural scaling \eqref{scaling}, as compared to the term $|u|^2 u$ and this explains the small factor of $\ld$ out the front of $u_{\ld} v_{\ld}$ in \eqref{ldSBO}. Whilst crucial, it is clear that all of our estimates in the previous sections hold for the $\ld$-SBO system \eqref{ldSBO} \textit{uniformly} in $0< \ld \leq 1$.

Now, since $s\geq 0$, we have
\begin{align*}
\|u_{0,\ld}\|_{H^{s+\frac 12}} \les \ld^{\frac 12}(1+\ld^{s+\frac 12})\|u_0\|_{H^{s+\frac 12}} 
\quad \text{and} \quad \|v_{0,\ld}\|_{H^{s}} \les \ld^{\frac 12}(1+\ld^{s})\|v_0\|_{H^{s}} 
\end{align*}
and so if $\eps_0>0$ is given, we may choose $\ld \les \min( \eps_0 \|(u_0,v_0)\|_{H^{\frac 12}\times L^2}^{-1}, 1)^{2}$ so that
\begin{align}
\|( u_{0,\ld}, v_{0,\ld})\|_{H^{\frac 12}\times L^2} \leq \eps_0. \label{scaledata}
\end{align}

Conversely, if $(u_{\ld},v_{\ld})$ is a solution to \eqref{ldSBO} with initial data $(u_{0,\ld},v_{0,\ld})$ on $[0,T^{\ast}(\ld)]$ for some $T^{\ast}(\ld)>0$, then $(u,v)= (\ld^{-1} u_{\ld}(\frac{t}{\ld^2} ,\frac{x}{\ld}), \ld^{-1}v_{\ld}(\frac{t}{\ld^2}, \frac{x}{\ld}))$ solves \eqref{SBO} with initial data $(u_0,v_0) = (\ld^{-1} u_0 (\frac{x}{\ld}), \ld^{-1}v_0 (\frac{x}{\ld}))$ on the time interval $[0,\ld^{-2}T^{\ast}(\ld)]$. In particular, the result of Proposition~\ref{PROP:LWPH2} and part (a) also applies to the scaled system \eqref{ldSBO} and gives a solution $(u,v)$ in $C_{T_{\max}}H^{s+\frac 12}\times C_{T_{\max}}H^{s}$, for some $T_{\max}>0$, depending only on $\|(u_0,v_0)\|_{H^{s+\frac 12}\times H^s}$, for any $s>\frac 34$.

The first step is to show that the time of existence $T_{\max}$ for these high regularity solutions can be lower bounded by a non-zero quantity only depending on $\|(u_{0}, v_0)\|_{H^{\frac 12}\times L^2}$. Given this, we then obtain difference estimates that allow us to construct rough solutions as limits of regular solutions. 
 
 \smallskip
\noi
\textbf{Step 1: A priori estimates:}
Let $\eps_0>0$ be sufficiently small to be chosen later. 
Given $(u_0,v_0)\in H^{\infty}(\R)\times H^{\infty}(\R)$ satisfying \eqref{scaledata}, let $(u,v)\in C_{T_{\max}}H^{\infty}(\R)^2$ be the solution to \eqref{ldSBO}, and let $w=w(v)$ be the gauged version of $v$ as given in \eqref{gauge}. 
For $T>0$, we define 
\begin{align}
\begin{split}
N_{T}^{s}(u,v)  = \max\bigg( \|v\|_{L^{\infty}_{T}H^s},  &\| J^s_x v\|_{ \wt{L^4_{T,x}}  }, \|w\|_{Z^{s}_{T}},  \|w(0)\|_{H^s} ,  \\
& \|u(0)\|_{H^{s+\frac 12}}, \| \ind_{[0,T)} \NN_{S;\ld}(u,v)\|_{X^{s+\frac 12, -\frac 12+2\dl}_{S}} \bigg)
\end{split} \label{NsT}
\end{align}
where
\begin{align*}
\NN_{S;\ld}(u,v)&: = \ld \big\{\tfrac{1}{\rho} e^{-i\rho F} u \cj{\PbHI w}-\tfrac{i}{\rho}  u\cj{R(v,w)} -i u\PbLO  v -i u \P_{+,\textup{HI}} v \big\} -i \be |u|^2 u \\
& = \ld uv -i \be |u|^2 u
\end{align*}
denotes the nonlinearity for the Schr\"{o}dinger part of \eqref{ldSBO} after replacing $\cj{\PbHIp v}$ by \eqref{recovery2}. 
From the Duhamel formulation of the Schr\"{o}dinger part of \eqref{ldSBO} and the linear estimates (Lemma~\ref{LEM:linXsb}), we have 
\begin{align}
\|u\|_{X^{s+\frac 12, \frac 12+\dl}_{S;T}} &\les \|u(0)\|_{H^{s+\frac 12}} +T^{\dl} \| \ind_{[0,T)} \NN_{S;\ld}(u,v)\|_{X^{s+\frac 12, -\frac 12+2\dl}_{S}} \les N^{s}_{T}(u,v). \label{uNsbd}
\end{align}
Next, as $u,v,$ and $w$ are smooth,  it follows from Lemma~\ref{LEM:Zcts} and Remark~\ref{RMK:Ncts} that the map $T\mapsto N_{T}^{s}(u,v)$ is continuous. Moreover, this map is non-decreasing in $T$.

Note that by \eqref{vXsb}, we have 
\begin{align*}
\sup_{0\leq \ta \leq 1}\|v\|_{X^{s-\ta,\ta}_{\text{BO};T}} \les N^{s}_{T}(u,v)+ N^{s}_{T}(u,v)^2 \les N_{T}^{s}(u,v).
\end{align*}
By the Duhamel formula for $w$, Lemma~\ref{LEM:linXsb}, \eqref{vXsb}, \eqref{BO1}, \eqref{BO2}, \eqref{BO3}, and \eqref{eFg1}, we have 
\begin{align*}
\limsup_{T\to 0^{+}} \|w\|_{Z^{s}_{T}} \les \|w(0)\|_{H^{s}} +\|u(0)\|_{H^{s+\frac 12}}\les (1+\|v(0)\|_{L^2})^{4} \|v(0)\|_{H^{s}}+\|u(0)\|_{H^{s+\frac 12}}.
\end{align*}
Therefore, with Lemma~\ref{LEM:SchrodingerII} and Lemma~\ref{LEM:SchrodingerIII}, we have 
\begin{align}
\limsup_{T\to 0^{+}}N_{T}^{s}(u,v) \les (1+\|v(0)\|_{L^2})^{4} \|v(0)\|_{H^{s}}+ \|u(0)\|_{H^{s+\frac 12}}. \label{NT0}
\end{align}

\noi
By combining Lemma~\ref{LEM:linXsb}, Lemma~\ref{LEM:linZsb}, Lemma~\ref{LEM:CCMXsb}, \eqref{Rbd}, Lemma~\ref{LEM:BO1}, Lemma~\ref{LEM:BO2}, Lemma~\ref{LEM:S1}, Lemma~\ref{LEM:SchrodingerII}, Lemma~\ref{LEM:SchrodingerIII}, we obtain
\begin{align}
\begin{split}
N^{\s}_{T}(u,v) & \leq  \|u(0)\|_{H^{\s+\frac 12}}+C_1 (1+\|v(0)\|_{L^2})^{a} \|v(0)\|_{H^{\s}}  \\
&+C_{2}(1+N_{T}^{0}(u,v))^{5} N_{T}^{0}(u,v) N_{T}^{\s}(u,v) +C_2 T^{\frac{1}{2}}N_{T}^{0}(u,v)
, 
\end{split} \label{apriori1}
\end{align}
where $\s\in \{0,s, 1\}$, and
for some constants $a>0$ and $C_1, C_2>0$. \footnote{ Whilst the constants depend on the choices $\{0,s,1\}$, we only use these three regularities and can thus take the maximum of the given constants over $\{0,s,1\}$.}

We first put $\s=0$ in \eqref{apriori1}.
We then choose $T >0$ so that $C_2 T^{\frac 12} \leq \frac 13$, and obtain
\begin{align}
N_{T}^{0}(u,v) \leq \tfrac{3}{2}\|u(0)\|_{H^{\frac 12}} + \tfrac{3}{2}C_1( 1+\|v(0)\|_{L^2})^a \|v(0)\|_{L^2} +\tfrac{3}{2}C_2 (1+N_{T}^0 )^5 (N_{T}^{0})^{2}.
\end{align}
In view of \eqref{NT0} and \eqref{scaledata}, by choosing $\eps_0>0$ sufficiently small, we have
\begin{align}
N_{T}^{0}(u,v) \leq 2\|u(0)\|_{H^{\frac 12}} +2C_1( 1+\|v(0)\|_{L^2})^a \|v(0)\|_{L^2} 
\end{align}
Using this information in \eqref{apriori1} at regularity $\s=1$, we obtain
\begin{align}\label{apriori2}
\| (u,v)\|_{L^{\infty}_{T'}H^{\frac 32}\times L^{\infty}_{T'}H^{1}} \leq N_{T'}^{1}(u,v) \leq 2\|u(0)\|_{H^{\frac 32}}+ 2C_1 (1+\|v(0)\|_{L^2})^{a} \|v(0)\|_{H^{1}},
\end{align}
for any $0<T'\leq T$
which implies that the maximal time of existence for these solutions is bounded from below by $T_{\ast}=T_\ast( \| (u(0), v(0))   \|_{H^{\frac 12}\times L^2} )$.

 \smallskip
\noi
\textbf{Step 2: Difference estimate:}
As for the uniqueness and continuity of the flow map, we consider differences of $H^{\infty}(\R)$ solutions and derive a difference estimate. Given two such solutions $(u_1,v_1),(u_2,v_2)$ to \eqref{ldSBO} with initial data $(u_j(0), v_j(0))\in H^{\infty}(\R)$, $j=1,2$. We assume that $\max_{j=1,2}\|(u_j(0), v_j(0))\|_{H^{\frac 12}\times L^2} \leq \eps_0$,
and that $v_j(0)$ are equal on low frequencies:
\begin{align}
\PbLO v_1(0) = \PbLO v_2(0). \label{LOsame}
\end{align}
This property is crucial; see Lemma~\ref{LEM:Fdiffs} below. By Step 1, there exists $T_{\ast}>0$ such that
\begin{align*}
\max_{j=1,2}N_{T}^{s}(u_j, v_j) \les \eps_0  \quad \text{and} \quad \max_{j=1,2} \sup_{0\leq \ta\leq 1} \| v_j\|_{X^{s-\ta, \ta}_{\text{BO}:T}} \les  \eps_0
\end{align*}
for all $0<T\leq T_{\ast}$.
We consider the difference $U:=u_1-u_2$ and $V:=v_1-v_2$. 
We also define the primitives $F_{j}= F_{j}[u_j]$ as in \eqref{F} and the attendant gauged variables $v_{j}=v_{j}[u_j]$ as in \eqref{gauge} with corresponding difference $W:=w_1-w_2$.

We then obtain difference estimates for the norms appearing in \eqref{NsT} using additionally \eqref{eFg2}, 
Lemma~\ref{LEM:Xsbdiffs}, \eqref{Rbddiff}, \eqref{RLqdiff}, Lemma~\ref{LEM:BO2}, \eqref{uw2}, Lemma~\ref{LEM:SchrodingerIII}. In order to control the differences $e^{\pm i\rho F_1}-e^{\pm \rho F_2}$, we use the mean value theorem to get $\| e^{\pm i\rho F_1}-e^{\pm \rho F_2}\|_{L^{\infty}_{T,x}} \les \|F_1 -F_2\|_{L^{\infty}_{T,x}}.$
To estimate the difference $F_1-F_2$, we need to use the assumption \eqref{LOsame}. For this purpose, we rely on the following result, which will be proved at the end of this step.

\begin{lemma}\label{LEM:Fdiffs}
Let $(u_1,v_1), (u_2,v_2)$ be two $H^{\infty}(\R)^2$ solutions to \eqref{ldSBO} with initial data satisfying 
\begin{align}
\PbLO v_1(0) = \PbLO v_2 (0). \label{sameLO}
\end{align}
For $j=1,2$, let $F_j=F_j[u_j, v_j]$ denote the primitive defined in \eqref{F}. Then, it holds that 
\begin{align}
 \| (F_1 -F_2) \vert_{t=0}\|_{L^{\infty}_{x}}  & \les \|v_1 (0) -v_2 (0)\|_{L^2_x}, \label{Fdiff0}\\
 \|F_1 -F_2\|_{L^{\infty}_{T,x}} & \les  \|v_1 (0) -v_2 (0)\|_{L^2_x}+ T( \|u_1\|_{L^{\infty}_{T}L^2_x} +\|u_2\|_{L^{\infty}_{T}L^2_x} ) \|u_1-u_2\|_{L^{\infty}_{T}L^2_x} \notag \\
 & \hphantom{XXXXXXXXXX} + T( \|v_1\|_{L^{\infty}_{T}L^2_x} +\|v_2\|_{L^{\infty}_{T}L^2_x} ) \|v_1-v_2\|_{L^{\infty}_{T}L^2_x}.
  \label{Fdifft}
\end{align}
\end{lemma}

At the end of this procedure, by reducing $T_{\ast}=T_\ast(\eps_0 )$ if necessary, we obtain 
\begin{align}
\| J^{\s} V\|_{L^{\infty}_{T}L^2_x \cap \wt{L^{4}_{T,x}}}  + \|W\|_{Z^{s}_{T}} + \|U\|_{X^{\s+\frac 12,\frac12+\dl}_{T}}  \leq C_3 \| (U(0), V(0))\|_{H^{\s+\frac 12}\times H^{\s}}, \label{LipbdT}
\end{align}
for any $0<T\leq T_{0}$ and $\s\in \{0,s\}$.

\begin{proof}[Proof of Lemma~\ref{LEM:Fdiffs}]
We begin by showing \eqref{Fdiff0}. In the following, we apply the single-variable operators $\Pblo$, $\PbLO$, $\Pbhi$, and $\PbHI$, as defined in \eqref{projs} and \eqref{PbloE}, to functions of two-variables. We keep track of which variable these operators act (keeping the remaining variable fixed) using a superscript notation: $\PbLO^{x}( h(x,y))$ and $\PbHI^{y}( h(x,y))$, etc.
The argument is similar to that in \cite[Section 4B]{MP}, tracking the variables these operators act on. 

We first derive a useful identity. From \eqref{F}, we have $F_1 (x) - F_1(y) = \int_{y}^{x} v_1(t,z)dz $. Then
\begin{align}
\PbLO^{x}\bigg( \int_{y}^{x} v_1(t,z)dz \bigg) = \PbLO^{x} F_1(x) - F_1(y), \label{Fdiff1}
\end{align}
where we used that $\PbHI^{x}F_1(y)=0$. It then follows from the commutativity of $\dx$ and $\PbLO$ that $\dx \PbLO F_1=\PbLO v_1$. Integrating both sides of this then gives 
\begin{align*}
\int_{y}^{x}\PbLO v_1 (t,z) dz = (\PbLO F_1)(x) - (\PbLO F_1)(y).
\end{align*}
 Combining this with \eqref{Fdiff1} then implies 
 \begin{align*}
\int_{y}^{x} \PbLO v_1(t,z) dz- \PbLO^{x} \bigg( \int_{y}^{x} v_1(t,z)dz \bigg) = (\PbHI F_1)(y).
\end{align*}
 Finally, we then apply $\Pblo^{y}$ to both sides to find the desired result:
 \begin{align}
\Pblo^{y} \bigg( \int_{y}^{x} \PbLO v_1(t,z)dz \bigg) = \Pblo^{y} \PbLO^{x} \bigg( \int_{y}^{x} v_1 (t,z)dz \bigg). \label{Fdiff2}
\end{align}
Clearly \eqref{Fdiff2} applies to any $H^{\infty}(\R)\times H^{\infty}(\R)$ solution $(u,v)$ to \eqref{ldSBO}.
Now we establish \eqref{Fdiff0}. Using that $G_j (0)=0$, the operators $\Pblo^{y}$ and $\PbLO^{x}$ commute, \eqref{Fdiff2}, \eqref{sameLO}, and that $\PbLO^{x}( h(y))= h(y) -\PbHI^{x}(h(y))= h(y)$, we have 
\begin{align}
\PbLO^{x}(F_1-F_2)(0,x) & = \int_{\R}\psi(y) \PbLO^{x} \bigg( \int_{y}^{x}(v_1-v_2)(0,z)dz \bigg) dy \notag\\
& = \int_{\R}\psi(y)\Pblo^{y} \PbLO^{x} \bigg( \int_{y}^{x}(v_1-v_2)(0,z)dz \bigg) dy \notag\\
& \quad + \int_{\R}\psi(y)\Pbhi^{y} \PbLO^{x}  \bigg( \int_{y}^{x}(v_1-v_2)(0,z)dz \bigg) dy\notag \\
& = \int_{\R}\psi(y)  \Pblo^{y} \bigg( \int_{y}^{x} \PbLO(v_1-v_2)(0,z)dz \bigg) dy \notag\\
&  \quad +  \int_{\R} (\Pbhi\psi)(y) \PbLO^{x} \dy^{-1} \dy\Pbhi^{y} \bigg( \int_{y}^{x}(v_1-v_2)(0,z)dz \bigg) dy \notag\\
& =  \int_{\R} (\dy^{-1} \Pbhi\psi)(y) \PbLO^{x} \Pbhi^{y}[ v_1-v_2](0,y) dy \notag \\
& =  \int_{\R} (\dy^{-1} \Pbhi\psi)(y) \Pbhi^{y}[ v_1-v_2](0,y) dy. \label{Fdiff4}
\end{align}
Then, by Cauchy-Schwarz we obtain
\begin{align}
\| \PbLO(F_1-F_2)\vert_{t=0}\|_{L^{\infty}_x} \les \| v_1(0)-v_2(0)\|_{L^{2}_x}. 
\label{Fdiff5}
\end{align}
To see \eqref{Fdiff0}, we note that $\PbHI(F_1-F_2)= \dx^{-1} \PbHI[ v_1-v_2]$ so Bernstein's inequality implies 
\begin{align}
\| \PbHI(F_1-F_2)\vert_{t=0}\|_{L^{\infty}_x} \les \| v_1(0)-v_2(0)\|_{L^{2}_x}. \label{Fdiff3}
\end{align}
This completes the proof of \eqref{Fdiff0}. 

For \eqref{Fdifft}, we split $F_1-F_2= \PbLO (F_1-F_2)+\PbHI(F_1-F_2)$. As in \eqref{Fdiff3}, we have $\|\PbHI(F_1-F_2)\|_{L^{\infty}_{T,x}}\les \|v_1-v_2\|_{L^{\infty}_{T}L^2_x}$. For the low part $\PbLO (F_1-F_2)$, we use the Duhamel formulation of \eqref{Feqn}:
\begin{align*}
\PbLO (F_1-F_2)(t) = e^{it \H\dx^2}\PbLO[F_1-F_2]\vert_{t=0} + \int_{0}^{t}  e^{i(t-t') \H\dx^2} \PbLO[ |u_1|^2 -|u_2|^2 +\rho (v_2^2 -v_1^2)]dt'.
\end{align*}
For the linear solution part, we use \eqref{Fdiff4} and \eqref{Fdiff5} to obtain
\begin{align*}
\|e^{it \H\dx^2}\PbLO[F_1-F_2]\vert_{t=0}\|_{L^{\infty}_{T,x}} = \|\PbLO[F_1-F_2]\vert_{t=0}\|_{L^{\infty}_{x}}  \les \|v_1(0) -v_2(0)\|_{L^{2}_x}.
\end{align*}
For the Duhamel integral piece, we note that since $u_j,v_j\in L^2$, we can interpret $\PbLO$ as a smooth projection to frequencies $\{ |\xi|\les 1\}$, and thus by the Bernstein inequality, this term in $L^{\infty}_{T,x}$ is bounded by
\begin{align*}
T( \|u_1\|_{L^{\infty}_{T}L^2_x} +\|u_2\|_{L^{\infty}_{T}L^2_x}) \|u_1-u_2\|_{L^{\infty}_{T}L^2_x} + T( \|v_1\|_{L^{\infty}_{T}L^2_x} +\|v_2\|_{L^{\infty}_{T}L^2_x}) \|v_1-v_2\|_{L^{\infty}_{T}L^2_x}.
\end{align*}
This completes the proof of \eqref{Fdifft}. 
\end{proof}

 \smallskip
\noi
\textbf{Step 3: Existence, uniqueness, and continuity property:}
For the existence of solutions in $H^{s+\frac 12}(\R)\times H^{s}(\R)$, we fix $(u_0,v_0)\in H^{s+\frac 12}(\R)\times H^{s}(\R)$ such that \eqref{scaledata} is satisfied and consider the sequence of approximations 
\begin{align*}
u_{0,j} = \mathcal{F}^{-1}\{ \ind_{[-j,j]} \ft u_0\} \qquad \text{and} \qquad v_{0,j} = \mathcal{F}^{-1}\{ \ind_{[-j,j]} \ft v_0\}
\end{align*}
 with corresponding $H^{\infty}(\R)\times H^{\infty}(\R)$ solutions $\{(u_j, v_j)\}_{j\in \N}$. By the previous results, these all belong to $C([0,T_{\ast}];H^{\infty}(\R))$, where $T_\ast = T_\ast(\|u_0\|_{H^{s_0}})>0$.  For all $j\in \N$ sufficiently large, the sequence $\{v_{0,j}\}_{j}$ satisfies \eqref{LOsame} by virtue of equalling $v_0$.
 Thus, by \eqref{LipbdT}, we see that the sequence $\{(u_j,v_j)\}_{j\in \N}$ is then Cauchy in the norm appearing in $N_{T}^{s}(u,v)$ and hence converges to a limit $(u,v)$ there, which satisfies \eqref{ldSBO} in the distributional sense and has $(u,v)\vert_{t=0}=(u_0,v_0)$. The continuous dependence with respect to initial data then follows by a standard argument. See for instance \cite[p. 387]{MP}.
 This completes the proof of the local well-posedness in Theorem~\ref{THM:LWP}.

\subsection{Global well-posedness}

In this section, we prove Theorem~\ref{THM:GWP} which relies on the following a priori bound coming from the conservation laws.

\begin{proposition}[A priori bound]
Let $\rho,\be\in \R$. There exists $r>0$ such that for every $(u,v)$ a global smooth solution to \eqref{SBO} with parameters $(\be,\rho)$ and initial data $(u_0,v_0)$ satisfying 
\begin{align}
\| u_0\|_{L^2}\leq r, 
\label{EnergL2}
\end{align}
 there exists a $C:\R_{+}^{2} \to \R_{+}$  increasing in its arguments such that
\begin{align*}
\| u(t)\|_{H^{1}(\R)}^{2} +\|v(t)\|_{H^{\frac 12}(\R)}^{2} \leq C( \|u_0\|_{H^1(\R)}, \|v_0\|_{H^{\frac 12}(\R)}).
\end{align*}
\end{proposition}
\begin{proof}
By the conservation of $E_{2}$ and $E_{1}$, Cauchy-Schwarz and Young's inequality, we have 
\begin{align}
\tfrac{1}{2}\|v(t)\|_{L^2}^{2}&  \leq \bigg| \text{Im} \int_{\R} u(t,x)  \cj{ \dx u(t,x)} \bigg| + \|u_0\|_{L^2}\|\dx u_0\|_{L^2} +  \tfrac{1}{2}\|v_0\|_{L^2}^2  \notag \\
& \leq \| u(t)\|_{L^2} \|\dx u(t)\|_{L^2}+ \|u_0\|_{L^2}\|\dx u_0\|_{L^2} +  \tfrac{1}{2}\|v_0\|_{L^2}^2 \notag \\
& \leq \|u_0\|_{L^2} \|\dx u(t)\|_{L^2} + C_1( \|u_0\|_{H^1}, \|v_0\|_{L^2}). 
\label{Energ1}
\end{align}
Next by the conservation of $H$, we may write 
\begin{align*}
\tfrac{1}{2}\|  |\dx|^{\frac 12} v(t)\|_{L^2}^{2} + \| \dx u(t)\|_{L^2}^{2} = H(u_0, v_0) + \frac{\rho}{3} \int_{\R}v(t)^3 dx -\int_{\R} v(t) |u(t)|^2 dx -\frac{\be}{2}\int_{\R} |u(t)|^4 dx,
\end{align*}
and we need to estimate the remaining time dependent terms on the right-hand side.
By the Gagliardo-Nirenberg inequality, $E_1$ conservation, and Young's inequality,
\begin{align}
\frac{|\be|}{2} \int_{\R} |u(t)|^4 dx \les \|u(t)\|_{L^2}^{3} \|\dx u(t)\|_{L^2} \leq C_{\eps}\|u_0\|_{L^2}^{6} +\eps \|\dx u(t)\|_{L^2}^2,\label{Energ2}
\end{align}
for any $\eps>0$. Next, by Cauchy-Schwarz, \eqref{Energ1} and \eqref{Energ2}, we have 
\begin{align}
\bigg| \int_{\R} v(t) & |u(t)|^2 dx \bigg| \leq \|v(t)\|_{L^2} \| u(t)\|_{L^4}^{2}  \notag \\
& \leq \eps^{\frac 12}\| v(t)\|_{L^2}^{2} + C\eps^{-\frac 12} \| u(t)\|_{L^4}^{4}  \notag\\
& \les  \eps^{\frac 12}\| v(t)\|_{L^2}^{2}+ C\eps^{-\frac 12} \|u(t)\|_{L^2}^{3}\|\dx u(t)\|_{L^2} \notag\\
& \les  \eps^{\frac 12}\| v(t)\|_{L^2}^{2}+ C\eps^{-\frac 12} \|u_0\|_{L^2}^{3}\|\dx u(t)\|_{L^2} \notag\\
& \leq \eps^{\frac 12}  \|u_0\|_{L^2} \|\dx u(t)\|_{L^2}^{2} +C_1( \|u_0\|_{H^1}, \|v_0\|_{L^2}) +\eps^{\frac 12} \|\dx u(t)\|_{L^2}^{2} + C_{\eps}\|u_0\|_{L^2}^{6}, \label{Energ3}
\end{align}
for any $\eps>0$. 
Lastly, by the Gagliardo-Nirenberg inequality, \eqref{Energ1} and \eqref{EnergL2}, we have 
\begin{align*}
 \frac{|\rho|}{3} \bigg| \int_{\R} v(t)^3 dx \bigg|
 & \leq C|\rho| \| v(t)\|_{L^2}^{2} \| |\dx|^{\frac 12}v(t)\|_{L^2}  \\
 &\leq C|\rho| \|u_0\|_{L^2} \|\dx u(t)\|_{L^2}  \| |\dx|^{\frac 12}v(t)\|_{L^2} +CC_1|\rho|\| |\dx|^{\frac 12}v(t)\|_{L^2} \\
 &\leq  Cr^{2} (\tfrac{1}{2}\|  |\dx|^{\frac 12} v(t)\|_{L^2}^{2} + \| \dx u(t)\|_{L^2}^{2} ) + C_{2}.
\end{align*}
Combining this with \eqref{Energ2}, \eqref{Energ3} and choosing $\eps>0$ and $r>0$ sufficiently small, we obtain
\begin{align*}
\tfrac{1}{2}\|  |\dx|^{\frac 12} v(t)\|_{L^2}^{2} + \| \dx u(t)\|_{L^2}^{2} \leq C(\|u_0\|_{H^1}, \| v_0\|_{L^2}).
\end{align*}
Inserting this into \eqref{Energ1} then gives control on $\|v(t)\|_{L^2}$. This completes the proof.
\end{proof}

\begin{ackno}\rm 
P. Dai is supported by the Commonwealth through an Australian Government Research
Training Program Scholarship.
J.F. was partially supported by the ARC project FT230100588.
\end{ackno}


\begin{thebibliography}{99}


\bibitem{SBO4}
J. Angulo, C. Matheus, D. Pilod, {\it Global well-posedness and non-linear stability of periodic traveling waves for a Schrödinger-Benjamin-Ono system}, Commun. Pure Appl. Anal. 8, 815--844 (2009).

 


\bibitem{SBO1}
D. Bekiranov, T. Ogawa, G. Ponce, {\it Interaction equations for short and long dispersive waves}, J.Funct. Anal. 158, 357--388 (1998).

\bibitem{App5}
 D.J. Benney, {\it Significant interactions between small and large scale surface waves}, Stud. Appl. Math.
55 (1976), 93--106.




\bibitem{BOP}
A. Bényi, T. Oh, Tadahiro, O. Pocovnicu,
{\it On the probabilistic Cauchy theory of the cubic nonlinear Schrödinger equation on $\R^d$, $d\geq 3$},
Trans. Am. Math. Soc., Ser. B 2, 1--50 (2015).

\bibitem{BL}
H.~Biagioni, F.~Linares, 
{\it Ill-posedness for the derivative Schr\"odinger and generalized Benjamin-Ono equations},
 Trans. Amer. Math. Soc. 353 (2001), no. 9, 3649--3659.




\bibitem{BO93}
J.~Bourgain, 
{\it Fourier transform restriction phenomena for certain lattice subsets and applications to nonlinear evolution equations, I: Schr\"odinger equations,} Geom. Funct. Anal. 3 (1993), 107--156.

\bibitem{BO932}
J. Bourgain, {\it Fourier transform restriction phenomena for certain lattice subsets and applications to non- linear evolution equations, Part II: The KDV-equation}, Geom. Funct. Anal. 3, 209--262 (1993).

\bibitem{BO98}
J. Bourgain,
{\it Refinements of Strichartz’ inequality and applications to 2D-NLS
 with critical nonlinearity},
Internat. Math. Res. Notices 1998, no. 5, 253–-283.

\bibitem{BL}
J.~Bourgain, D.~Li, 
{\it On an endpoint Kato-Ponce inequality},
 Differential Integral Equations 27 (2014), no. 11-12, 1037--1072.

\bibitem{BP}
N. Burq, F. Planchon,
{\it On well-posedness for the Benjamin-Ono equation},
Math. Ann. 340, No. 3, 497--542 (2008).


\bibitem{CFL}
A.~Chapouto, J.~Forlano, T. Laurens,
{\it  On the well-posedness of the intermediate nonlinear Schrödinger equation on the line},
Proc. Amer. Math. Soc. Ser. B 11 (2024), 452--468.


\bibitem{CFLOP-2}
A.~Chapouto, J.~Forlano, G.~Li, T.~Oh, D.~Pilod,
{\it  Intermediate long wave equation in negative Sobolev spaces},
Proc. Amer. Math. Soc. Ser. B, B 11 (2024), 452--468.


\bibitem{CLOZ}
A.~Chapouto, G. Li, T. Oh, T. Zhao,
{\it Shallow-water convergence of the intermediate long wave equation in $L^2$},
	arXiv:2511.15905 [math.AP].

\bibitem{CLOP}
A.~Chapouto, G.~Li, T.~Oh, D.~Pilod,  {\it Deep-water limit of the intermediate long wave equation in $L^2$}, Math. Res. Lett. 31 (2024), no. 6, 1655--1692.

\bibitem{CW}
M. Christ, M. Weinstein, {\it Dispersion of small amplitude solutions of the generalized Korteweg-de Vries
equation}, J. Funct. Anal. 100 (1991), 87--109.

\bibitem{CHT}
J. Colliander, J. Holmer, N. Tzirakis, 
{\it Low regularity global well-posedness for the Zakharov and Klein-Gordon-Schrödinger systems},
Trans. Am. Math. Soc. 360, No. 9, 4619--4638 (2008).

\bibitem{CL} 
A. J. Corcho, F. Linares, 
{\it Well-posedness for the Schrödinger-Korteweg-de Vries system}, 
Trans. Am. Math. Soc. 359, No. 9, 4089--4106 (2007).



\bibitem{App4}
 V. Djordjevic, L. Redekopp, 
 {\it On two-dimensional packet of capillary-gravity waves}, J. Fluid Mech. 79
(1977), 703--714.

\bibitem{SBO3}
L. Domingues, {\it Sharp well-posedness results for the Schrödinger-Benjamin-Ono system}, Adv. Differ. Equ. 21, 31--54 (2016).



\bibitem{App1}
M. Funakoshi, M. Oikawa, {\it The resonant interaction between a long internal gravity wave and a surface gravity wave packet}, J. Phys. Soc. Japan, 52 (1983), 1982--1995.




\bibitem{GL}
L. Gassot, T. Laurens,
{\it Global well-posedness for the ILW equation in $H^s(\T)$ for $s>-\frac 12$},
	arXiv:2506.05149 [math.AP].
	


\bibitem{GKT-2}
P.~G\'erard, T.~Kappeler, P.~Topalov,
{\it 
Sharp well-posedness results of the Benjamin-Ono equation in
$H^s(\T, \R)$ 
and qualitative properties of its solutions}, 
Acta Math.
231 
(2023), no. 1, 
 31--88.



\bibitem{GO}
L.~Grafakos, S.~Oh, 
{\it The Kato-Ponce inequality},
 Comm. Partial Differential Equations 39 (2014), no. 6, 1128--1157.
 
 

\bibitem{App3}
R. Grimshaw, {\it The modulation of an internal gravity-wave packet and the resonance with the mean motion}, Stud. Appl. Math. 56 (1977), 241--266.



\bibitem{U2V22}
M. Hadac, S. Herr, H. Koch, {\it Well-posedness and scattering for the KP-II equation in a critical space},
Ann. Inst. H. Poincaré Anal. Non Linéaire 26 (2009), no. 3, 917--941. {\it Erratum to “Well-posedness and
scattering for the KP-II equation in a critical space”}, Ann. Inst. H. Poincaré Anal. Non Linéaire 27
(2010), no. 3, 971--972.

\bibitem{U2V23}
S. Herr, D. Tataru, N. Tzvetkov, {\it Global well-posedness of the energy critical nonlinear Schr\"{o}dinger
equation with small initial data in $H^{1}(\T^3)$}, Duke Math. J. 159 (2011) 329--349.


\bibitem{IS}
M.~Ifrim, J.-C.~Saut,
{\it The lifespan of small data solutions for Intermediate Long Wave equation (ILW)},
Comm. Partial Differential Equations 50 (2025), no. 3, 258--300.

\bibitem{IT}
M. Ifrim, D. Tataru,
{\it Well-posedness and dispersive decay of small data solutions for the Benjamin-Ono equation},
Ann. Sci. Éc. Norm. Supér. (4) 52, No. 2, 297--335 (2019).

\bibitem{IK}
A. D. Ionescu, C. E. Kenig, 
{\it Global well-posedness of the Benjamin–Ono equation in low-regularity spaces},
J. Am. Math. Soc. 20, No. 3, 753--798 (2007).


\bibitem{Iorio}
R. J. Iório, {\it On the Cauchy problem for the Benjamin-Ono equation}, Commun. Partial. Differ. Equ. 11, 1031--1081 (1986).



\bibitem{App2}
 V. Karpman, {\it On the dynamics of sonic-Langmuir soliton,} Physica Scripta. 11 (1975), 263--265.
 

\bibitem{KLV}
R. Killip, T. Laurens, M. Vişan, 
{\it Sharp well-posedness for the Benjamin-Ono equation},
Invent. Math. 236, No. 3, 999--1054 (2024).

\bibitem{U2V21}
 H. Koch, D. Tataru, {\it A priori bounds for the 1D cubic NLS in negative Sobolev spaces}, Int. Math. Res.
Not. (2007), no. 16, Art. ID rnm053, 36 pp.

\bibitem{Obernotes}
H. Koch, D. Tataru, M. Vişan, 
{\it Dispersive equations and nonlinear waves. Generalized Korteweg-de Vries, nonlinear Schrödinger, wave and Schrödinger maps}, 
Oberwolfach Seminars 45. Basel: Birkhäuser/Springer (ISBN 978-3-0348-0735-7/pbk; 978-3-0348-0736-4/ebook). xii, 312 p. (2014).

\bibitem{KT}
H. Koch, N. Tzvetkov, {\it On the local well-posedness of the Benjamin-Ono equation in $H^{s}(\R)$}, Int. Math. Res. Not. 2003, 1449--1464 (2003).

\bibitem{LMP}
F. Linares, A. J. Mendez, D. Pilod, 
{\it Well-posedness for the extended Schrödinger-Benjamin-Ono system}, Vietnam J. Math. 52, No. 4, 1043--1066 (2024).

\bibitem{Molinet1}
L. Molinet, 
{\it Global well-posedness in $L^2$
 for the periodic Benjamin-Ono equation}, 
Am. J. Math. 130, No. 3, 635--683 (2008).

\bibitem{MP}
L.~Molinet, D.~Pilod, 
{\it The Cauchy problem for the Benjamin-Ono equation in $L^2$ revisited},
 Anal. PDE 5 (2012), no. 2, 365--395. 
 
 \bibitem{MP2}
L.~Molinet, D.~Pilod, 
{\it Global well-posedness and limit behavior for a higher-order Benjamin-Ono equation}, 
Comm. Partial Differential Equations 37 (2012), no. 11, 2050--2080.



\bibitem{MST}
L. Molinet, J. C. Saut,  N. Tzvetkov, {\it Ill-posedness issues for the Benjamin-Ono and related equations}, SIAM J. Math. Anal. 33, 982--988 (2001).



\bibitem{OhSBO}
T. Oh, {\it Invariance of the Gibbs measure for the Schrödinger-Benjamin-Ono system}, SIAM J.Math.Anal. 41, 2207--2225 (2010).

\bibitem{OT}
T. Ozawa, Y. Tsutsumi,
{\it Space-time estimates for null gauge forms and nonlinear Schrödinger equations}, 
 Differential Integral Equations 11 (1998), no. 2, 201--222.



\bibitem{SBO2}
 H. Pecher, {\it Rough solutions of a Schrödinger-Benjamin-Ono system}, Differ. Integral Equ. 19, 517--535 (2006).



\bibitem{TAO04}
T.~Tao, 
{\it Global well-posedness of the Benjamin-Ono equation in $H^1(\mathbf R)$}, 
J. Hyperbolic Differ. Equ. 1 (2004), no. 1, 27--49. 

\bibitem{TAObook}
T.~Tao, 
\textit{Nonlinear dispersive equations.
Local and global analysis}, 
CBMS Reg. Conf. Ser. Math., 106.
Published for the Conference Board of the Mathematical Sciences, Washington, DC; by the American Mathematical Society, Providence, RI, 2006. xvi+373 pp.





\end{thebibliography}
\end{document}